%% file: SurveyRicam.tex
\documentclass[reqno,8pt]{amsart}
\usepackage{mystyle2}
\usepackage{hyperref}

\title{Optimal Transportation Theory with repulsive costs}

\author{Simone Di Marino}
\address{Laboratoire de Math\'{e}matiques d'Orsay,Universit\'e Paris-Sud, FRANCE }
\email{simone.dimarino@math.u-psud.fr}

\author{Augusto Gerolin}
\address{Dipartimento di Matematica, Universit\`a di Pisa, Largo B.
Pontecorvo 5, 56126 Pisa, ITALY}
\email{gerolin@mail.dm.unipi.it }

\author{Luca Nenna}
\address{INRIA, MOKAPLAN, Domaine de Voluceau Le Chesnay, FRANCE }
\email{luca.nenna@inria.fr}

\begin{document}

\maketitle

\tableofcontents

\input{abstract}
\input{introduction}

\input{DFTmeetsOptimalTransportationTheory}

\input{multimarginalOTproblemwithCoulombCost}
\input{multimarginalOTproblemwithRepulsiveHarmonicCost}
\input{multimarginalOTproblemdeterminant}

\input{Numerics}
\input{conclusion}

\bibliographystyle{plain}
\bibliography{refs}

\end{document}

%% file: abstract.tex
\begin{abstract}
RICAM survey repulsive cost in multi-marginal optimal transportation theory and predictions for the future!

\end{abstract}

%% file: introduction.tex

\section*{Introduction}
The main objective of this survey is to present recent developments in optimal transport theory for repulsive costs with finitely many marginals. We are interested in characterizing, for a class of repulsive cost functions, the minimizers of the multi-marginal optimal transportation problem
\[
(\MK) \ \ \min \Bigg\lbrace
\int_{\R^{dN}}c(x_1,\dots,x_N)d\gamma : \gamma \in
\Pi(\R^{dN},\mu_1,\dots,\mu_N) \Bigg\rbrace,
\]
where $c:(\R^d)^N\longrightarrow \R\cup\lbrace \infty\rbrace$ is a lower semi-continuous cost function and $\Pi(\R^{dN},\mu_1,\dots,\mu_N)$ denotes the set of couplings, i.e. probability measures $\gamma(x_1,\dots,x_N)$ having $\mu_i$ as marginals, for all $i =1,\dots,N$. $c$ is, typically, a sum of convex decreasing functions, as the Coulomb interaction cost \cite{BuDePGor,CFK, Lieb83, GorSeiVig, Lieb83, Sei} 
\[
c(x_1,\dots,x_N) = \sum^N_{i=1}\sum^N_{j=i+1}\dfrac{1}{\vert x_j-x_i\vert},
\]
which is the case of most interest in DFT (see section \ref{sec:why}).  Also related to Quantum Mechanics, in the so-called weak interaction regime \ \cite{BuDePGor,FMPCK,PazMorGorBac,TouColSav}, we will consider the repulsive harmonic potential as cost function $c$, 
\[
c_w(x_1,\dots,x_N) = -\sum^N_{i=1}\sum^N_{j=i+1}\vert x_j-x_i\vert^2.
\]
In the last section, we focus on the multi-marginal problem for the determinant cost 
\[
c(x_1,\dots,x_N) = \operatorname{det}(x_1,\dots,x_N), \quad \textit{or} \quad c(x_1,\dots,x_N) = \vert \operatorname{det}(x_1,\dots,x_N) \vert,
\]
studied by Carlier and Nazaret in \cite{CarNa}.


Multi-marginal optimal transport showed to be powerful in many applications as, for instance, in economics \cite{CarEke, ChiGalSal, ChiMcCNes, Eke, Pass14}, mathematical finance \cite{BeiHenPen, BeiJul, BeiCoxHue, DolSon, DolSon2, GalHenTou}, image processing \cite{XferPeyAu}, tomography \cite{AbrAbrBerCar} and statistics \cite{CarGalChe}.

Although very natural in such applied problems,  optimal transport problems with $3$ or more marginals was only considered in literature during the 90's by Knott and Smith, Olkin and Rachev and R\"uschendorf and Uckelmann \cite{KnoSmi, OlkRac,RusUck}, and later by Gangbo and \'Swiech \cite{GaSw}. In \cite{GaSw}, the authors proved the existence and uniqueness of Monge-type solution for the problem $(\MK)$ for the quadratic cost $c(x_1,\dots,x_N) = \sum^N_{i,j=1} \vert x_j-x_i\vert^2$.

A systematic study of the structure of transport plans and general conditions in order to guarantee the existence of Monge-type solutions for the problem $(\MK)$ was made by Pass \cite{Passthesis, Pass3, Pass2} and Kim and Pass \cite{KimPass, KimPass14}. Several theoretical developments were made last years, as the study of partial multi-marginal transport \cite{KitPass}, symmetric Monge-Kantorovich problem \cite{CoDMa,CoDePDMa, GhoMau,GhoMoa,GhoPass,KolZae,LopMen,LopOliThi,Zae} and duality \cite{BeiLeoSch, DeP, Kel, Zae}. Applications in Geometry and Analysis as the Wasserstein Barycenter problem and decoupling a system of PDE's can be found in \cite{AguCar, GhoPass,KimPasspreprint}. 

For a complete presentation of the multimarginal optimal transport theory with general costs functions, including a discussion of many applications, we refer to the survey \cite{PassSurvey}. \

In this paper, we are rather analyzing the Multimarginal Optimal Transport Theory starting from the point of view of researchers working on Density Functional Theory. 

Paola Gori-Giorgi, Michel Seidl \textit{et al} were interested in computing accurately the ground state of a molecule with a large number of atoms. As we will discuss later, after doing suitable approximations, the solutions of the optimal transport $(\MK)$ for Coulomb cost (more precisely, the optimal potential), or more generally repulsive costs, can be seen as a good \textit{ansatz} in order to construct solutions for Hohenberg-Kohn functional \eqref{eq.FHK}. 

From a mathematical perspective the interesting case is when all marginals of the coupling $\gamma$ in $(\mathcal{MK})$ are absolutely continuous with respect to Lebesgue measure and are all the same. In particular, since the cost functions are \textit{``repulsive''}, if Monge-type solutions exists, they should follow the rule \textit{``the further, the better!''}, which means that we want to move the mass as much as we can. In other words, in the present case, optimal transport plans tend to be as spread as possible.  But when a cost can be considered \textit{repulsive}? It is difficult to find a good notion for this concept since usually it is based on the intuition that comes solving the particular problem. However a possible formal definition for repulsive cost is to require a mild condition, namely that we don't want any two particles in the \textit{same} position; it should be cheaper to have them in different position. For example can write the inequality for the particles $1$ and $2$:
\begin{align*} c(x_1,x_1, x_3, \ldots, x_N) + c(x_2,&\, x_2, x_3, \ldots, x_N ) \geq \\ & c(x_1,x_2, x_3, \ldots, x_N) + c(x_2,x_1,x_3, \ldots, x_N).
\end{align*}
Then, also all other inequalities for $i \neq j$ should be considered. This condition is satisfied by all the considered costs, even the ones that doesn't seem so ``repulsive'' at first, such as those of the form $f(x_1+ \ldots + x_n)$ with $f$ convex function.
\subsection*{Organization of this Survey}
In section \ref{sec:why}, we present the physical motivation in considering repulsive costs and give an introduction of Density Functional Theory from a mathematical point of view. The main goal of this section is to explain how optimal solutions of Monge-Kantorovich problem can be useful to compute the ground state energy of a quantum system of $N$ electrons.

Section \ref{sec:DFTOT} introduces briefly the general theory of multimarginal optimal transport, including a discussion about the Kantorovich duality and the notion of $c$-cyclical monotonicity applied to this framework. In addition, we give some highlights in the interesting case of symmetric transport plans. 

In sections \ref{sec:Coulomb}, \ref{sec:OTHarm} and \ref{sec:Det}, respectively, we discuss the recent development for the specific cases of transportation problem with  Coulomb, repulsive harmonic and determinant cost.

Finally, in section \ref{numerics} we introduce an iterative method, recently used in many applications of optimal transport \cite{Ben,BenCaNen}, and we present some numerical experiments for the problems studied in previous sections.

\subsection*{What is new?}

We present in section \ref{sec:OTHarm} some new results we obtained for cost functions 
\[
c(x_1,\dots,x_N) = f( x_1+\dots+x_N), \quad f:\R^d\to\R \ \textit{is convex},
\]
which, in particular, include the case of repulsive harmonic costs (see corollary \ref{OTHARM:maincor}). In addition, we show: \medskip

\begin{enumerate}
\item[(i)] \textit{Existence of Fractal-like cyclical Monge solutions:} In Theorem \ref{OTHarm:mainthm}, we show the existence of couplings $\gamma = (Id, T,T^{(2)},\dots,T^{(N-1)})_{\sharp}\mu$ concentrated on the graphs of functions $T,T^{(2)},\dots,T^{(N-1)}$, where $\mu$ a uniform probability measure on the $d$-dimensional cube. Moreover, $T$ is not differentiable in any point, which differs substantially to the classical regularity theory in the classical $2$ marginals case \cite{Caf,Caf2}.\medskip
\item[(ii)] \textit{Optimal Transport plans supported on sets of Hausdorff dimension $> 1$:} In standard optimal transportation (the 2 marginals case) the target is to understand if solutions of $(\mathcal{MK})$ are well-known to be Monge-type solutions: optimal transport plans are supported by the graph of a function $T$. In Section \ref{sec:OTHarm}, we give a concrete example (see example \ref{OTHarm:ex:trueplan}) where optimal transport plans are supported on $2$-dimensional Hausdorff sets.\medskip
\item[(iii)] \textit{Numerical results:} we present some interesting results for the repulsive harmonic and for the determinant cost. 

\end{enumerate}

\section{Why Multi-marginal Transport Theory for repulsive costs?}\label{sec:why}

\subsection{Brief introduction to Quantum Mechanics of $N$-body systems}

$\quad$ Energies and geometries of a molecule depend sensibly on the kind of
atom on a chemical environment. Semi-empirical models as \textit{Lewis
structure} explain some aspects, but they are far from being satisfactory,
mostly because they are not quantitative.

Classical behavior of atoms and molecules is described accurately, at
least from a theoretical point of view, by quantum mechanics. However,
in order to predict the chemical behavior of a molecule with a large
number of electrons we need to deal with computational aspects and
approximations are needed. For example, if we use a direct
approximation of 10 grid points in $\R$ of the time-dependent
Schr\"odinger equation in order to simulate the chemical behavior of
the water molecule ($H_2O$), which has 10 electrons and its position
is represented in the space $\R^{30}$, then we need to compute
$10^{30}$ grid points.

Quantum Chemistry studies the ground state of individual atoms and
molecules, the excited states and the transition states that occur
during chemical reactions. Many systems use the so-called {\it
Born-Oppenheimer approximation} and many computations involve
iterative and other approximation methods.

The main goal of Quantum Chemistry is to describe accurately results
for small molecular systems and to increase the size of large
molecules that can be processed, which is very limited by scaling
considerations. 

It means that from a Quantum Model, for which the only input is the
atomic number of a molecule, we want to predict the evolution of a
molecular system. Classical models consider a molecule with $N$
electrons and $M$ nuclei. We denote by a point  $x_i \in \R^3$ the
position coordinates of the $i$-th electron of mass $m_e$ and by $s_i$
its $i$-th
electron spin. The charges, the masses and the positions of the $\alpha$-th nucleus are represented,
respectively, by $Z_{\alpha} \in \R^{M}$, $m_{\alpha} \in \R_+^{M}$
and $R_{\alpha} \in (\R^3)^M$ 

The state of the system is described by a time-dependent wave function
$\psi \in L^2( [0, T] ;  \otimes^{N}_{i=1} L^2(\R^3\times\Z^2) \times
\otimes_{\alpha=1}^{M} L^2(\R^3) )$,
$\psi(x_1,s_1,\dots,x_N,s_N;R_1,\dots,R_M,t)$. The group $\Z^2 = \lbrace \uparrow,\downarrow \rbrace$
represents the spin of a particle.

We say that a wave function is antisymmetric or \textit{fermionic} if it changes sign under a simultaneous exchange of two electrons
$i$ and $j$ with the space coordinates $x_i$ and $x_j$ and spins $s_i$ and $s_j$, that is, 
\[
\psi(x_{\sigma(1)},s_{\sigma(1)},\dots,x_{\sigma(N)},s_{\sigma(N)}) =
\operatorname{sign}(\sigma)\psi(x_1,s_1,\dots,x_N,s_N) \quad \sigma
\in \S_N
\]
where $\S_N$ denotes the permutation set of $N$ elements.

In physics literature, the \textit{Fermionic} wave function obeys the \textit{Fermi-Dirac statistics} and, in particular, the \textit{Pauli exclusion principle}. On the other hand, a wave function is called symmetric or {\it bosonic} when it obeys \textit{Bose-Einstein statistics} which also implies that when one swaps two bosons, the wave function of the system is unchanged.

If it is not mentioned explicitly, we will suppose that the wave function is spinless, i.e. $\psi$ is a function depending only on time, electrons and nuclei position
$\psi(t,x_1,x_2,\dots,x_N,R_1,\dots,R_M)$. 

The evolution of a wave function $\psi$ boils down to
a (time-dependent) Schr\"odinger Equation
\begin{equation}\label{GSE}
i\partial_t \psi = H \psi \quad , \quad H = T_n + T_e + V_{ne} + V_{ee} + V_{nn}
\end{equation}
where the operators $T_n, T_e, V_{ne}, V_{ee}$ and $V_{nn}$ are defined by
\begin{itemize}
\item[] $T_n = -\h\summ_{\alpha=1}^M\dfrac{1}{2m_{\alpha}}\Delta_{R_{\alpha}}
\quad \textit{(Nuclei Kinetic Energy)}$
\item[] $T_e = -\h\summ^N_{i=1}\dfrac{1}{2m_e}\Delta_{x_i}$ \quad
\it{(Electron Kinetic Energy)}
\item[] $V_{ne} = -\summ^{N}_{i=1}v_R(x_i)
=-\summ^N_{i=1}\bigg(\sum^M_{\alpha=1}\dfrac{Z_{\alpha}}{\vert
x_i-R_{\alpha}\vert} \bigg)  \quad \quad \Big(\begin{array}{ll}
\small{\textit{Potential Energy of Inte-}} \\
\small{\textit{raction Nuclei-Electron}} \end{array}\Big)$
\item[] $V_{ee} = \summ^{N}_{i=1}\summ^{N}_{j=1,j\neq i}
\dfrac{1}{\vert x_i - x_j \vert} \quad \textit{(Interaction
Electron-Electron)} $
\item[] $V_{nn} = \summ^{M}_{\alpha=1}\summ^{M}_{\beta=1,\beta\neq
\alpha} \dfrac{Z_{\alpha}Z_{\beta}}{\vert R_{\alpha} - R_{\beta}
\vert} \quad \textit{(Interaction Nuclei-Nuclei)}$ $\medskip$
\end{itemize}

\medskip {\it\bf Born-Oppenheimer Approximation}: The Born-Oppenheimer
approximation (named for its original inventors, Max Born and Robert
Oppenheimer) is based on the fact that nuclei are several thousand
times heavier than electrons. The proton, itself, is approximately
2000 times more massive than an electron. Roughly speaking, in
Born-Oppenheimer approximation we suppose that nuclei behave as point
classical particles. This is reasonable, since typically $2000m_e \leq m_{\alpha} \leq 100000m_e$.

In other words, we suppose $m_{\alpha} \gg m_e$, we fix $m_e =
1$ and we consider an \textit{ansatz} of type $\psi(x,R,t) =
\psi_e(x,R)\chi(R,t)$, which physically means that the dynamics of
the electron $\psi_e$ are decoupled of the dynamics of the atomic nuclei $\chi$. Substituting
that \textit{ansatz} in \eqref{GSE}, we can show that the electronic
part $\psi_e(x,R)$ solves the following eigenvalue problem, so-called
\textit{Eletronic Schr\"odinger Equation}
\begin{equation}\label{ESE}
 H_e\psi_e = \lambda_e \psi_e, \quad  \quad H_{e} = T_e + V_{ne} +
V_{ee} + V_{nn}
\end{equation}
where the eigenvalue $\lambda_e = \lambda_e(R)$ depends on the
position vector $R = (R_1,\dots,R_M)$ of the atomic nuclei. The function
$H_{e} = H_{e}(R)$ is called \textit{Eletronic Hamiltonian}. From the other
side, through a formal argument, the nuclei wave function $\chi(R,t)$ is a solution of a Schr\"odinger Equation restricted to the nuclei with potential energy $\lambda_e$,
$$ i\partial_t \chi(R) = (T_n + \lambda_e(R))\chi(R). $$

We refer to \cite{Bodo} for all computations and formal deduction of those formulas. Another consequence of that approximation is that the nuclear components of the wave function are spatially more localized than the electronic component of the wave function. In the classical limit, the nuclei are fully localized at single points representing classical point particles. \medskip

In Quantum Chemistry it is interesting to study the so called \textit{Geometric Optimization Problem}.

\medskip {\bf The Geometric Optimization Problem}:  Compute the
following minimizer
\be\label{pb:GOP}
\inf \bigg\lbrace  E(R_1,\dots,R_M) +
\summ^{M}_{\alpha=1}\summ^{M}_{\beta=1,\beta\neq \alpha}
\dfrac{Z_{\alpha}Z_{\beta}}{\vert R_{\alpha} - R_{\beta} \vert} :
(R_1,\dots,R_M) \in \R^{3M} \bigg\rbrace
\ee
where
\begin{itemize}
\item The second term is the nuclei-nuclei iteration $V_{nn}$, as
defined in the \eqref{GSE}.
\item The function $E(R) = E(R_1,\dots,R_M)$ corresponds to the
effective potential created by the electrons and is, itself, given by
a minimization problem (see \textit{Electronic minimization problem}
bellow).
\end{itemize}

The value of the potential $E_0 = E_0(R_1,\dots,R_M)$ for given fixed
positions of the nuclei is obtained by solving the electronic problem:

\medskip {\bf The Electronic minimization problem}:  Compute the
lowest eingenvalue ({\it ``ground state energy''}) $E_0$ of the
following linear operator, called {\it Electronic Hamiltonian}

\begin{deff}[Eletronic Hamiltonian]\label{op.Hel}
The Eletronic Hamiltonian is a linear operator $H_{el}: L_{anti}^2(\R^{3N})\to L_{anti}^2(\R^{3N}),$
$$ \quad  \Hel
=-\dfrac{\h^2}{2}\summ^N_{i=1}\Delta_{x_i} + \summ^N_{i=1}v_R(x_i) +
\summ^N_{i =1}\summ^N_{ j =i+1}f(x_j - x_i), $$
\end{deff}
where $L_{anti}^2(\R^{3N})$ denote the set of square-integrable
antisymmetric functions $\psi:\R^{3N}\to \R$, $v_R:\R^3\rightarrow \R$
is a $L^2(\R^{3N})$ function and $f:\R\rightarrow \R$ is a continuous
function. $\medskip$

The first term of the Eletronic Hamiltonian is the Kinetic Energy and
the second term, the function $v_R(x_i)$, depends only on the single
electron. Typically, $v_R(x_i) = -\sum^M_{\alpha=1}
\frac{Z_{\alpha}}{\vert x_i-R_{\alpha}\vert}$ is the interaction
electron-nuclei energy of the single electron. The function $f$
depends only on the distance of two electrons and measures the
electron-electron potential interaction; typically $f$ is the Coulomb
Interaction $f(x_j - x_i) = 1/\vert x_j - x_i \vert$ or
the repulsive harmonic interaction $f(x_j - x_i) = -\vert
x_j - x_i \vert^2$. From both mathematical and physical viewpoint, it may be also useful to
study more general convex and concave functions.

The {\it ground state} of the functional (\ref{op.Hel}) is given by
the {\it Rayleigh-Ritz variational principle}
\[
E_0 = E(R_1,\cdots,R_M) = \min \lbrace \langle \psi, \Hel\psi \rangle
: \psi \in \Hi \rbrace,
\]
where $\Hi = \lbrace \psi \in H^1(\R^{3N}) : \psi \text{ is
antisymmetric and } \Vert \psi \Vert = 1 \rbrace$. Equivalently,
\begin{equation}\label{eq.groundstateenergy}
E_0 = \min\Big\lbrace  T_e[\psi] + V_{ne}[\psi] + V_{ee}[\psi] : \psi
\in \Hi \Big\rbrace
\end{equation}
where $T_e[\psi]$ is the \textit{Kinetic energy},
\[
\ds T_e[\psi] = \dfrac{\hbar^2}{2}\summ_{s_1\in
\Z^2}\int_{\R^{3}}\dots\summ_{s_N\in
\Z^2}\int_{\R^{3}}\summ^N_{i=1}\vert
\nabla_{x_i}\psi(x_1,s_1\dots,x_N,s_N)\vert^2dx_1\dots d x_N
\]
$V_{ne}[\psi]$ is the electron-nuclei interaction energy,
\[
\ds V_{ne}[\psi] = \summ_{s_1\in
\Z^2}\int_{\R^{3}}\dots\summ_{s_N\in
\Z^2}\int_{\R^{3}}\summ^N_{i=1}v_R(x_i)\vert
\psi(x_1,s_1\dots,x_N,s_N)\vert^2dx_1\dots d x_N
\]
and, $V_{ee}[\psi]$ is the electron-electron interaction energy
\[
\ds V_{ee}[\psi] = \summ_{s_1\in
\Z^2}\int_{\R^{3}}\dots\summ_{s_N\in
\Z^2}\int_{\R^{3}}\summ^N_{i=1}\summ^N_{j=i+1}\dfrac{1}{\vert
x_j-x_i\vert}\vert \psi(x_1,s_1\dots,x_N,s_N)\vert^2dx_1\dots d x_N
\medskip
\]

Here, the ground state quantum $N$-body problem refers to the problem of finding equilibrium states for system of type (\ref{GSE}). A
theorem proved first by Zhislin \cite{Zh60}, guarantees the existence
of the minimum in $(\ref{eq.groundstateenergy})$. Some variants can be
found in the literature due, for instance, to Lieb \& Simon
\cite{LiebSimon}, Lions \cite{Lions87} and Friesecke \cite{Fri03}.

In computational chemistry, the Electronic minimization problem methods are used to
compute the equilibrium configuration of molecules and solids. However, the most
computationally practicable methods in this context are not numerical
methods (in the sense in which this terminology is used in
mathematics), but power series of analytical solutions of reduced models.

\subsection{Probabilistic Interpretation and Marginals} 

The square norm of the wave function
$\psi(x_1,s_1,\cdots,x_N,s_N)$ can be interpreted as a $N$-point
probability density distribution for the electrons are in the points
$x_i$ with spins $s_i$. 
\[
\int_{\R^{3N}} \vert \psi(x_1,s_1,\cdots,x_N,s_N) \vert^2 dx = 1.
\]
If $\psi$ is a $L^2(\R^{3N})$ function we can define the single particle density by
\be\label{eq.singledensity}
\rho(x_1) =
N\int_{\R^{3(N-1)}}\gamma_N(x_1,s_1,\cdots,x_N,s_N)dx_2,\cdots
dx_N,
\ee
where $\gamma_N$ represents the $N$-point position density
\be\label{eq.Npointdensity}
\gamma_N(x_1,\cdots,x_N) = \sum_{s_1,\cdots,s_N \in \Z_2} \vert
\psi(x_1,s_1,\cdots,x_N,s_N)\vert^2.
\ee
Analogously, we can define the $k$-density
\be\label{eq.pairdensity}
\gamma_k(x_1,\dots,x_k) = {N\choose
k}\int_{\R^{3(N-k)}}\gamma_N(x_1,s_1,\cdots,x_N,s_N)dx_{k+1}.\cdots
dx_N
\ee

We remark that when the particle is spinless the single particle density \eqref{eq.singledensity} can be simply be written as
\[
\rho(x_1) =
N\int_{\R^{3(N-1)}}\vert \psi(x_1,\dots,x_N)\vert^2 dx_2,\cdots
dx_N.
\] 

The relevance of $\rho$ and $\gamma_2$  to compute the
ground state energy (\ref{eq.groundstateenergy}) is due to the fact
that the interation nuclei-nuclei energy $V_{ne}[\psi]$ and
electron-electron energy $V_{ee}[\psi]$ depends only on these.

\begin{lemma}[\cite{CFK}]\label{lemma:rhoresem} If $\psi, \nabla \psi \in
L^2((\R^3\times \Z^2)^N,\R)$, $\psi$ antisymmetric and $\Vert \psi
\Vert_2 = 1$, then
\[
V_{ne}[\psi] = \int_{\R^3} v(x)\rho(x)dx, \quad V_{ee}[\psi] =
\int_{\R^6} \dfrac{1}{\vert x-y\vert} \gamma_2(x,y)dxdy.
\]
\end{lemma}

. 

It is natural to wonder about the space of those densities arising from such a density $\psi$ is defined, i.e, the space $\mathcal{A}$ of densities
$\rho:\R^{3}\to \R$ verifying \eqref{eq.singledensity} for an
antisymmetric $\psi$ such that $\psi, \nabla \psi \in L^2((\R^3\times
\Z^2)^N,\R)$ and $\Vert \psi \Vert_{2} = 1$. An explicit characterization exists (due to Lieb \cite{Lieb83}):

\be\label{eq.spacerho}
\mathcal{A} = \bigg\lbrace \rho:\R^{3}\to\R : \rho \geq 0, \sqrt{\rho}
\in H^1(\R^{3}) \text{ and } \int_{\R^3} \rho(x)dx = N \bigg\rbrace.
\ee

\subsection{Density Functional Theory (DFT)}

$\quad$ The Density Functional Theory is the standard approximation to quantum
mechanics in simulations of a system with more than a dozen electrons
and showed to be successful in many instances but has rare drastic
failures as, for instance, in predicting the behavior of $Cr_2$
molecules \cite{CoMoYa}. DFT theory approximates quantum mechanics via variational
principles for the marginal density
\be\label{eq:rho1}
\rho(x_1) := \int_{\R^{d(N-1)}} \vert \psi(x_1,\dots,x_N)\vert^2 dx_2\dots dx_N,
\ee
where $\psi$ is a wave function associated to the $N$-body quantum
problem as in equation \eqref{eq.singledensity}. Roughly speaking, DFT
models are semi-empirical models of the pair density
\be\label{eq:rho2}
\gamma_2(x_1,x_2) := \int_{\R^{3(N-2)}} \vert \psi(x_1,\dots,x_N)\vert^2 dx_3\dots dx_N.
\ee
in terms of its marginal $\rho$. We simply write $\psi \rightarrow \rho$ and $\psi \rightarrow \gamma_2$ to denote the relation between $\psi$ and $\rho$, and $\psi$ and $\gamma_2$. This means, respectively, that $\psi$ has single particle density $\rho$ and pair density $\gamma_2$.

Concerning the ground state problem \eqref{eq.groundstateenergy}, we
notice that, after the lemma \ref{lemma:rhoresem}, we are able to
write the \textit{Electronic minimization problem}
\eqref{eq.groundstateenergy} as
\be\label{eq.groundstatedensity}
E_0 = \inf_{\rho \in \mathcal{A}} \bigg\lbrace F_{HK}[\rho] +
\int_{\R^3}v(x)\rho(x)dx \bigg\rbrace,
\ee
with
\be\label{eq.FHK}
F^{HK}[\rho] = \inf\bigg\lbrace T[\psi] + V_{ee}[\gamma_2] : \rho \in
\mathcal{A}, \psi \to \rho \bigg\rbrace . \medskip
\ee
where $F_{HK}$ is the so called {\it Hohenberg-Kohn} or {\it Levy-Lieb functional}. For a mathematical sake of completeness we also consider a weak repulsive interaction given by
\be\label{eq:confinedrepulsiveharmonic}
\hat{V}_{ee}[\rho] = \ds\int_{K}\sum^N_{i=1}\sum^N_{j=1}-\vert x_j-x_i\vert^2\gamma_N(x_1,\cdots,x_N)dx_1\dots dx_N,
\ee
in a compact set $K \subset \R^{3N}$ and $+\infty$ otherwise.

The next theorem is the main result in Density Functional Theory. 
\begin{teo}[Hohenberg-Kohn, Levy-Lieb]\label{th:HK}
Let $\rho \in \P(\R^3)$ be a probability density such that $\rho \in
H^1(\R^3)$ and $f:\R\to\R$ be a continuous function. There exists a functional $F^{HK}:\P(\R^{3})\rightarrow
\R$ depending only on the single-particle density $\rho$ such that for
any potential $V_{nn}$, the exact Quantum Mechanics ground state
energy (\ref{eq.groundstateenergy}) satisfies
\begin{equation}\label{eq.HK}
E_0 = \min \Big\lbrace \Big(F^{HK}[\rho] +
N\int_{R^3}v_R(x)\rho(x)dx\Big) : \rho \in \P(\R^{3})\Big\rbrace
\end{equation}
Moreover, $F^{HK}[\rho]$ is given itself as a minimum problem
\[
F^{HK}[\rho] =\min \bigg\lbrace \Big\langle
\psi,\Big(-\frac{\hbar^2}{2}\Delta +
\sum^N_{i=1}\sum^N_{j=i+1}f(x_j -
x_i)\Big)\psi\Big\rangle \text{ } : \text{ } \psi \in
\Hi, \psi \rightarrow \rho \bigg\rbrace
\]
where, $\psi \to \rho$ means that $\psi$ has single-particle density $\rho$ and $\mathcal{A}$ is the set of $\rho:\R^{3}\to\R$ such that $\rho \geq 0$, $\sqrt{\rho}
\in H^1(\R^{3})$  and $\int_{\R^3} \rho(x)dx = N$.

\end{teo}

The Hohenberg-Kohn theorem states that the functional $F^{HK}$ is
\textit{universal} in the sense that it does not depend on the
molecular system under consideration. From a Physics point of view, it
garantees that in a molecular system of $N$ electrons, the single electron density $\rho$ determines the pair density of the system (see corollary \ref{cor:HK}).  From a
mathematical perspective, the proof of this theorem is a functional
analysis exercise that, for sake of completeness, we are going to
present in a version we learnt from Gero Friesecke.

\begin{proof}
Firstly we remark that the non-universal part of the energy only depends on
$\rho$:
\[
\langle \psi, \sum^{N}_{i=1}v(x_i)\psi \rangle = \int_{\R^{3N}}
\sum^{N}_{i=1}v(x_i)\vert \psi(x_1,\dots,x_N)\vert^2d\mathbf{x} =
N\int_{\R^3} v(x)\rho(x)dx
\]
Then, we rewrite problem
(\ref{eq.groundstateenergy}) as a double $\inf$
problem
\[
\begin{array}{ll}
E_0
&=\ds \inf_{\psi} \Big( \langle \psi,\H_{el}\psi\rangle +
N\int v(r)\rho(r)dr\Big) \\
&=\ds \inf_{\rho} \Big(\inf_{\psi\mapsto\rho}\Big( \langle
\psi,\H_{el}\psi\rangle\Big) + N\int v(r)\rho(r)dr\Big) \\
&=\ds \inf_{\rho} \Big(F^{HK}[\rho] + N\int v(r)\rho(r)dr\Big)
\end{array}
\]
where  $F^{HK}[\rho] =
\inf_{\psi\mapsto\rho}\Big( \langle \psi,\H_{el}\psi\rangle\Big)$.
\end{proof}

The present version stated in theorem \ref{th:HK} is due by Levy and
Lieb. The next corollary contains the main physical and mathematical
consequence of the Hohenberg-Kohn theorem (\ref{th:HK}).

\begin{cor}[HK Theorem and Coupling problem]\label{cor:HK}
Let $\rho \in \P(\R^{3})$ be a measure and $\psi \in H^1(\R^{3N})$ be
an antisymmetric function with $\Vert \psi \Vert_{2} = 1$. There
exists a universal map from single-particle density $\rho(x_1)$ to a
pair densities $\gamma_2(x_1,x_2)$ - or even $k$-density
$\gamma_k(x_1,\dots,x_k)$ - which gives the exact pair density of any
$N$-electron molecular ground state $\psi(x_1,\dots,x_N)$ in terms of
its single-particle density.
\end{cor}

\begin{proof}
Consider
\[
\overline{\psi} \in \operatorname{argmin} \lbrace \langle \psi,
H_{el}\psi\rangle : \psi \to \rho \rbrace
\]
and define by $\gamma_k$ the universal k-point density of that minimizer, i.e.
\[
\gamma_k(x_1,\cdots,x_k) = \int \vert
\overline{\psi}(x_1,\cdots,x_N)\vert^2d_{k+1}\cdots dx_{N}
\]
\end{proof}

Of course, $\gamma_k$ may be not unique, since the \textit{ground state}
may be degenerate. The existence of a map $T:\R^3\to\R^3$ which
determines the density of $\gamma_2$ from $\rho$ is highly not trivial
and not feasible because still uses high dimension wave function.

\subsection{``Semi-classical limit'' and Optimal Transport problem}

A natural approach to understand the behavior of the Hohenberg-Kohn functional $F^{HK}[\rho]$ \eqref{eq.FHK} is to study separately, the contributions of the kinetic energy and the Coulomb interaction electron-electron energy. A well-know method in DFT literature, known as adiabatic connection, is to neglect the kinetic energy, in which the electron-electron interaction is rescaled with a real parameter $\lambda$ while keeping the density $\rho$ fixed \cite{LanPer,Yang}
\be\label{eq:FHKlamba}
\ds F_{\lambda}^{HK}[\rho] =\min_{\substack{\psi \in \Hi  \\ \psi \rightarrow \rho}} \bigg\lbrace \langle \psi,\big(-\frac{\hbar^2}{2}\Delta +\sum^N_{i=1}\sum^N_{j=i+1}\frac{\lambda}{\vert x_j - x_i\vert})\psi\rangle \text{ } \bigg\rbrace.  \\
\ee

The strictly correlated electron limit $(\lambda \to \infty)$, up to relaxing in the space of probability measures, was first considered by two papers in physics literature: Seidl \cite{Sei} and Seidl, Gori-Giorgi \& Savin \cite{SeiGorSav}. In \cite{GorSeiVig}, Gori-Giorgi, Seidl and Vignale interpreted the strong-interaction regime as a mass transportation problem with a Coulomb cost. 

Later, an equivalent limit, so-called \textit{``semi-classical limit''}, was made mathematically rigorous in the two particles case by Cotar, Friesecke \& Kl\"uppelberg \cite{CFK}. They considered the Hohenberg-Kohn functional $F^{HK} = F_{\hbar}^{HK}$ as a function of both $\rho$ and $\hbar$ $(\lambda = 1)$ and proved that - up to passage of the limit $\hbar\to 0$  - the $F_{HK}$ reduces to the following functional
\be\label{eq.HKpairdensities}
\tilde{F}[\rho] = \inf \bigg\lbrace \int_{\R^6}\dfrac{1}{\vert
x-y\vert}\gamma_2(x,y)dxdy : \gamma_2 \in \mathcal{A}_2, \gamma_2\to\rho
\bigg\rbrace,
\ee
where $\gamma_2\to\rho$ means that $\gamma_2$ satisfies the equation
\eqref{eq.pairdensity}, and the set $\mathcal{A}_2$ of admissible pair
density functions is defined by the image of $\mathcal{A}$ under the map $\psi
\to \gamma_2$. As pointed out in \cite{CFK}, instead of the corresponding single particle density
case, we do not know any characterization of the space of admissible pair density function
$\mathcal{A}_2$.

We state in the next theorem the semi-classical limit for the $2$-particles case. The general case of $N$ particles is still an open problem.

\begin{teo}[``Semi-classical limit'' for $N=2$,
Cotar-Friesecke-Kl\"uppelberg, \cite{CFK}]\label{th.semiclassicalCFK}
Let $\rho:\R^3\to\II$ be a probability density such that  $\sqrt{\rho}
\in H^1(\R^3)$, $N=2$ and $f:\R\to\R$ be the Coulomb or the repulsive harmonic interaction potentials. Then,
\be\label{eq.semiclassicallimit}
\begin{array}{ll}
F_{\hbar}^{HK}[\rho]
&=\underset{\psi \in H^1(\R^{6}), \psi \rightarrow \rho}{\min} \langle \psi,\big(-\frac{\hbar^2}{2}\Delta +
f(x_2 - x_1)\big)\psi\rangle \text{ }  \\
&\ds\underset{\hbar\to 0}{\to} \min_{\gamma \in \Pi_2(\R^6, \rho)} \int_{\R^{3N}}
f(x_2-x_1)d\gamma(x_1,x_2) =: F^{OT}[\rho],
\end{array}
\ee
where the set $\Pi_2(\R^6,\rho)$ denotes the set of measures $\gamma \in
\P(\R^{6})$ having $\rho$ as marginals, i.e. $(e_i)_{\sharp}\gamma = \rho$ where, for $i=1,2$, $e_i:\R^{6}=\R^3\times \R^3\to \R^3$ are the projection maps. 
\end{teo}

The main difficulty in proving Theorem \ref{th.semiclassicalCFK} is that
for any transport plan given by a density $\psi$, $\gamma = \vert
\psi(x_1,\dots,x_N) \vert^2dx_1\dots x_N$, it may be that $\psi \not\in
H^1(\R^{3N}), \psi \not\in L^2(\R^{3N})$ and $T[\psi] = \infty$.
Moreover, smoothing the optimal $\gamma$ does not work at this
level, because this may change the marginals of the problem
(\ref{eq.semiclassicallimit}).  Cotar, Friesecke \& Kl\"uppelberg
developed a smoothing technique in order to deal with this problem
without changing the marginals. A complete proof of the previous
theorem can be found in \cite{CFK}.

Let us precise somewhat the terminology. In Quantum Mechanics, the \textit{``semi-classical limit''} has a precise meaning: it is an asymptotic regime for the 
Hamiltonian dynamics of a Quantum system defined in a Hilbert space and it is given by a
\textit{Weyl-Wigner quantization} (or \textit{quantization by
deformation} studied in the more abstract context of Poisson manifolds).
In those specific cases, the limit $\hbar\to 0$ is called
\textit{``semi-classical''} limit, because the first and second order terms of an asymptotic expansion of the Hamiltonian operator is given by \textit{``classical''} terms, functions of the Hamiltonian function in a symplectic manifold \cite{EvaZwo,Rob87}. 

In DFT context, that limit seems to be up to now merely a question of re-scaling.  The minimizers of $F^{OT}[\rho]$ are candidates of \textit{ansatz} to develop approximating methods to compute the ground state energy of the \textit{Electronic Hamiltonian}.

\medskip

At that point, it is natural to define the \textit{DFT-Optimal
Transportation ground state} $E^{DFT-OT}_{0}$ \cite{CFK},

\begin{equation}\label{eq:DFTOT}
\begin{array}{ll}
E_0^{DFT-OT}[\rho]
&=\ds\inf \bigg\lbrace T[\psi] + V_{ne}[\psi] +
E_{OT}[\psi] : \psi \in \Hi \bigg\rbrace \\
&=\ds\inf \bigg\lbrace T_{QM}[\rho] + V_{ne}[\rho] + E_{OT}[\rho] :
\rho \in \mathcal{A} \bigg\rbrace.
\end{array}
\end{equation}

where $\rho$ represents the single particle density (see
\eqref{eq.singledensity}), $T_{QM}$ and $E^{OT}$ are defined,
respectively, by
\[
\ds T_{QM}[\rho] = \inf \bigg\lbrace T[\psi] : \psi \in \Hi, \psi \to
\rho \bigg\rbrace \quad \text{and}
\]
\be\label{eq:OTgroundstate}
E^{OT}[\rho] =\ds\inf \bigg\lbrace
\int_{\R^{3N}}\summ^N_{i=1}\summ^N_{i=1}f(x_j-x_i)d\gamma(x_1,\dots,x_N)
: \gamma \in \Pi((\R^d)^N,\rho) \bigg\rbrace,
\ee
where $\Pi((\R^d)^N,\rho)$ denotes the set of probability measures $\gamma:(\R^d)^N \to \R_+$ having $\rho$ has marginals. 

The problem of minimizing $E_0^{DFT-OT}$ in (\ref{eq:DFTOT}) is called
DFT-OT problem. Notice that, in the case where $f$ is the Coulomb
potential,
\[
V_{ee}[\gamma_2] \geq E_{OT}[\rho], \quad \text{for every probability
density } \rho \in H^{1}(\R^3).
\]
Finally, taking the infimum in both sides of the previous equation, we have

\begin{teo}[Cotar-Friesecke-Kl\"uppelberg, \cite{CFK}]
Consider $f(\vert x_j-x_i\vert) = 1/\vert x_j-x_i\vert$. For every
$N$, and any potential $v_R \in L^{3/2}+L^{\infty}(\R^3)$, the density
functional with electron-electron interaction energy
(\ref{eq:OTgroundstate}) is a rigorous lower bound of the
\textit{Electronic minimization problem} \eqref{eq.groundstateenergy}
\[
E_{0} \geq E^{DFT-OT}_{0}
\]
\end{teo}

It turns out now that the central question of DFT-OT problem with Coulomb potentials is to characterize the minimizers $\gamma$ of $E^{OT}[\rho]$. Those minimizers could be used as \textit{ansatz} in order to find minimizers of the Hohenberg-Kohn $F^{HK}[\rho]$. In \cite{Sei}, the physicist Seidl formulated the following conjecture. 

\begin{conj}[Weak Seidl Conjecture \cite{Sei}]\label{conj:Seidl}
There exists a deterministic minimizer of $E^{OT}[\rho]$ with Coulomb potential/cost
\[
c(x_1,\dots,x_N) = \sum^{N}_{i=1}\sum^N_{j=i+1}\dfrac{1}{\vert x_i - x_j\vert}.
\]
\end{conj}

 In other words,  for Coulomb-type electron-electron interactions, Seidl conjectures\footnote{There is also a stronger version of the conjecture, the Strong Seidl Conjecture, where the author describes explicitly a possible optimal map in the radially symmetric case: for the precise statement see the conjecture \ref{conj.Seidl} and the lemma \ref{lemma:radialcase} in section \ref{sec:Coulomb}.} state that, at least for radially symmetric measures $\rho$, there exists an optimal measure $\gamma$ for $E^{OT}$ of form $\gamma = (Id\times
T_1\times \dots \times T_{N-1})_{\sharp}\rho$, with
$T_1,\dots,T_{N-1}:\R^3\to\R^3$. This maps $T_i$ are called co-motion functions and, due to the symmetries of the problem, a natural group law is required for them, namely that for every $i, j$ there exists $k$ such that $T_i \circ T_j = T_k$; this is satisfied if for example $T_i=T^{(i)}$ (cyclical case, see Equation \eqref{pb:MKcyc}).

It is natural to ask the same kind of question for other cost functions. In the following we mention some interesting results: \medskip
\begin{itemize}
\item[(i)] In the $2$-electrons case, the minimizers of $E^{DFT-OT}_{0}$ are well-understood \cite{CFK} and there
exists a map which correlates the density of a given minimization plan of $E^{OT}[\rho]$, see Theorem \ref{th:cotar}. \medskip
\item[(ii)] Gangbo \& \'Swi\c{e}ch (without being aware of Seidl) reinforces the conjecture \ref{conj:Seidl} showing in \cite{GaSw} that  for the attractive harmonic cost $c(x_1,\dots,x_N) = \sum^{N}_{i,j=1} \vert x_i - x_j\vert^2$ there exists a unique optimal plan, and it is a deterministic one. \medskip
\item[(iii)] Colombo, De Pascale \& Di Marino \cite{CoDePDMa} answer affirmatively (strong) Seidl conjecture in the one dimensional case $(d=1)$, as we will describe in section \ref{sec:Coulomb}.  \medskip
\item[(iv)] The authors proved that \textit{``Seidl's conjecture''} is not true for the repulsive harmonic cost $c(x_1,\dots,x_N) = \sum^{N}_{i,j=1} -\vert x_i - x_j\vert^2$ (see theorem \ref{OTHarm:mainthm}). Moreover, for the same cost and for a particular single particle density, there exists \textit{``fractal-like''} optimal transport maps, i.e. an optimal map $T$ which are not differentiable at every point. We are going to state and prove those statements, see section \ref{sec:OTHarm}. 
\end{itemize}

%% file: DFTmeetsOptimalTransportationTheory.tex
\section{DFT meets Optimal Transportation Theory}
\label{sec:DFTOT}

\subsection{Couplings and Multi-marginal Optimal Transportation problem}

\begin{deff}[Coupling]\label{def:coupling}
Let $(X,\mu),(Y,\nu)$ be two probability spaces. A coupling $\gamma$ is a measure on the product $X\times Y$ such that $\gamma$ admits $\mu$ and $\nu$ as marginals, respectively, on
$X$ and $Y$, i.e.
\[
(e_1)_{\sharp}\gamma = \mu, \quad (e_2)_{\sharp}\gamma = \nu
\]
where $e_1,e_2$ are respectively the projection maps $(x,y) \in
X\times Y \mapsto x \in X$ and $(x,y) \in X\times Y \mapsto y \in Y$.
\end{deff}

It is easy to show that a coupling of $\mu$ and $\nu$ is equivalent to
\begin{itemize}
\item[(i)] For all measurable sets $A\subset X$ and $B\subset Y$, one
has $\gamma[A\times Y] = \mu[A]$ and  $\gamma[X\times B] = \nu[B]$.
\item[(ii)] For all integrable measurable functions $u,v$ on $X,Y$,
\[
\ds\int_{X\times Y} (u(x) + v(y))d\gamma(x,y) = \int_X u(x)d\mu(x) +
\int_Y v(y)d\nu(y)
\]
\end{itemize}

A trivial coupling of $(X,\mu)$ and $(Y,\nu)$ is given by the product
measure $\mu\times\nu$. The symbol $\rho \to \psi$ defined in eq. \eqref{eq:rho1} means that the measure $\gamma = \vert \psi(x_1,\dots,x_N)\vert^2dx_1\dots dx_N$ is a coupling, for instance, between $\rho_1$ and $\rho_{k-1}$. In the classical theory of $2$-marginals optimal transport \cite{AGS,VilON} a coupling is also called a transport plan. 

In the case of a  $2$-molecular system with Coulomb interactions, up to passage to the limit $\hbar\to 0$, the theorem \ref{th.semiclassicalCFK} states that there exists a special coupling $\gamma$, minimizer of \eqref{eq.semiclassicallimit}, which is concentrated on the graph of a measurable function. That motivates the next definition

\begin{deff}[Deterministic Coupling]\label{def:deterministicoupling}
A coupling $\gamma$ of two probability spaces $(X,\mu)$ and $(Y,\nu)$ is said to be deterministic if
does there exists a measurable function $T:X\to Y$ such that $\gamma = (Id,T)_{\sharp}\mu$.
\end{deff}

In that terminology, the main question is to understand when the coupling given by the DFT-Optimal Transport \eqref{eq:OTgroundstate} is deterministic. Moreover, Seidl's conjecture can be rephrased in the following way: \textit{Does exist a deterministic coupling in the Optimal Transportation problem \eqref{eq:OTgroundstate} for Coulomb costs?}

\subsection{Multimarginal Optimal Transportation Problem}
The \textit{``semi-classical limit"} in the $2$ particles case (thm. \ref{th.semiclassicalCFK}) and Seidl's Conjecture \ref{conj:Seidl}, motivate us to introduce a more general coupling problem.

We denote by $N$  the number of particles (marginals), $d$ the dimension of the
space where those particles live (typically $d=3$) and $I = \lbrace 1,\dots,\N\rbrace$ an index set. Take 
$x_i:\R^d\rightarrow\R$ as the position of a single particle, then the
configuration of a system of $N$-particles is given by a vector $x =
(x_1,x_2,\dots, x_N)$ or, more precisely, by a function $x:I \to \R^d$.

Let us consider the distribution of the $i$-th particle given by a density function $\rho_i:\R^{d}\rightarrow \01$, $i\in I$ and $c:\R^{Nd}\times \R^{Nd}\rightarrow \II$, a continuous cost
function. We set $\mu_i = \rho_i\mathcal{L}^d, \forall i \in I$,
where $\mathcal{L}^d$ is the Lebesgue measure on $\R^d$, and we want to characterize the optimal coupling of the Monge-Kantorovich problem
\be\label{pb:MKN}
(\MK) \quad  \inf_{\gamma \in \Pi(\R^{dN},\mu_1,\dots,\mu_N)} 
\int_{(\R^{d})^N}c(x_1,\dots,x_N)d\gamma(x_1,\dots,x_N),
\ee
where $\Pi(\R^{dN},\mu_1,\dots,\mu_N)$ denotes the set of couplings $\gamma(x_1,\dots,x_N)$ having $\mu_i$ as marginals, for all $i \in I$,
i.e. the set $\Pi(\R^{dN},\mu_1,\dots,\mu_N) = \big\lbrace \gamma \in \P(\R^{dN})
: (e_i)_{\sharp}\gamma = \mu_i, \forall i \in I \big\rbrace$ and, for
all $i \in I$, $e_i:\R^{dN}\to\R^d$ are the projection on the $i$-th
component and $\P(\R^{dN})$ denotes the set of probability measures in $\R^{dN}$. The elements of $\Pi(\R^{dN},\mu_1,\dots,\mu_N)$ are called transport plans. In the DFT-OT problem it was mentioned that an interesting case is when all
marginals are the same and equal to a measure $\mu$ on $\R^{d}$ (it is quite natural in the DFT framework as the electrons are indistinguishable). In that case, we
denote the set of transport plans simply by $\Pi_N(\mu)$.

Notice that the existence of a minimum for \eqref{pb:MKN} is not a big deal and it is quite standard in Optimal Transport Theory. Indeed, $\Pi(\R^{dN}, \mu_1,\dots,\mu_N)$ \textit{is trivially non empty}, since the independent coupling $\mu_1\times\dots\times\mu_N$ belongs to $\Pi(\R^{dN}, \mu_1,\dots,\mu_N)$; the set $\Pi(\R^{dN}, \mu_1,\dots,\mu_N)$ is convex  and compact for the weak*-topology thanks to the imposed marginals; and moreover the quantity to be minimized $\gamma \mapsto \int c \, d \gamma$ is linear with respect to $\mu$. Hence, we can guarantee the existence of a minimum for (\ref{pb:MKN}) by imposing a very weak hypothesis on the cost function, such as lower-semicontinuity.

However, we are interested in characterizing some class of optimal transport plans or, at least, understand when that optimal coupling is \textit{deterministic}. Thus, we would like to study when minimizers for $(\MK)$
correspond to minimizers the multi-marginal Monge problem 
\be\label{pb:MN}
(\MN) \quad \min_{T = (T_i)^{N}_{i=1} \in \mathcal{T}_N} \int_{(\R^d)^N}
c(x,T_2(x),\dots,T_{N}(x)) d\mu_1(x),
\ee
where  $\mathcal{T}_N = \big\lbrace T = (T_i)^{N}_{i=1} : T_i \textit{
is a Borelian map such that } T_{i\sharp}\mu_1 = \mu_i, \forall i \in \lbrace 1,\dots,N \rbrace \text{ and } T_1 = Id  \big\rbrace$. As usual, it is easy to see that the set of \textit{admissible plans in \eqref{pb:MKN} ``includes" the set of Monge transport maps in \eqref{pb:MN}}: in fact, given a transport map  $T = (Id, T_2,\dots,T_N)$,  we
can consider a measures $\gamma_T$ on $(\R^{d})^N$ defined by
$\gamma_T = (\operatorname{Id}\times T_2\times \dots \times
T_N)_{\sharp}\mu$. Let $h:(\R^{d})^N\rightarrow \R$ be a
$\gamma_T$-measurable function, then
$$\int_{(\R^{d})^N} h(x_1,\dots,x_N)d\gamma_T(x_1\dots,x_N) =
\int_{\R^{d}}h(x_1,T_2(x_1),\dots,T_N(x_1))d\mu_1(x_1); $$
and if we have $\gamma_T$-measurable function $f:\R^d\rightarrow [0,\infty]$
$$ \forall i \in I, \quad \int_{(\R^{d})^N}f(x_i)d\gamma_T(x_1,\dots,x_N) = \int_{\R^d}
f(x_i)d\mu_i(x_i), $$
and so, $\gamma_T$ is a transport plan. In particular, the value of the minimization over Monge-Kantorovich couplings is less than or equal to the value of the minimization over the deterministic couplings
\[
\min(\MK) \leq \inf(\MN).
\]
Contrary to in \eqref{pb:MKN}, the existence of the optimal deterministic coupling in (\ref{pb:MN}) is not obvious and the difficulties do not lie on the ``multi-marginality" of the problem. In
fact, at that point, it is not possible to apply standard methods in
Calculus of Variations even in the classical $2$-marginals case in order to guarantee the existence of the
minimum (\ref{pb:MN}). 

Intuitively, the difference between  \eqref{pb:MKN} and \eqref{pb:MN} is that in
$\MN$ almost every point $x_1 \in \R^d$ is coupled with exactly one
point $x_i \in \R^d$ for each $i \in \lbrace 2, \dots, N\rbrace$
whereas throughout the problem $(\ref{pb:MKN})$ we allow the splitting
of mass in the ``transportation'' process.

The problem (\ref{pb:MKN}), when $N=2$, was studied by L. Kantorovich
in 1940's \cite{Kant1,Kant2}, without being aware of the famous article of Monge \cite{Monge}. In 1975, L. Kantorovich won the Nobel
Prize together with C. Koopmans \textit{``for their contributions to
the theory of optimum allocation of resources"}.

An important disadvantage in using the relaxed approach \eqref{pb:MKN} in the multi-marginal OT is that the set of optimal transportation plans could be very large and we could need to select some special classes of transportation plans.

In the $2$-marginal setting, the two formulations \eqref{pb:MKN} and \eqref{pb:MN} are proved to have the same value for general Polish Spaces by Pratelli in \cite{Pra12}. 

\begin{teo}[Monge equals Monge-Kantorovich in the 2-marginals case, \cite{Pra12}]\label{thm:PratelliMongeKantorovich} Let $(X,\mu), (Y,\nu)$ be Polish spaces, where $\mu \in \P(X)$ and $\nu \in \P(Y)$ are non-atomic probability measures. If $c:X\times Y\to \R\cup\lbrace \infty\rbrace$ is a continuous cost, then
\begin{multline}
\min\Bigg\lbrace \int_{X\times Y} c(x,y)d\gamma(x,y) \text{ } : \text{ } \gamma \in \Pi(X\times Y,\mu,\nu) \Bigg\rbrace = \\ 
= \inf\Bigg\lbrace \int_{X}c(x,T(x))d\mu(x), \text{ }  :  \begin{array}{ll}
T_{\sharp}\mu = \nu ,\\
T:X\to Y \text{ Borel }
\end{array} \Bigg\rbrace \text{ } 
\end{multline}
where $\Pi(X\times Y)$ denotes the set of $2$-marginals transport plans $\gamma \subset X\times Y$ having $\mu$ and $\nu$ as marginals.
\end{teo}

In the particular case of the multi-marginal OT problem in (\ref{pb:MN}) and (\ref{pb:MKN}) when all marginals are equal to $\rho\mathcal{L}^d$, we can apply the previous theorem (\ref{thm:PratelliMongeKantorovich}) on the Polish spaces $X = \R^d, Y= (\R^d)^{N-1}$ with measures $\mu = \rho\mathcal{L}^d$ and $\nu = (e_2, \ldots , e_n)_{\sharp} \gamma$, for every $\gamma \in \Pi(\R^{dN}, \rho \mathcal{L}^d)$, and obtain the following corollary 
\begin{cor}\label{cor:MNMK}
Let $\mu_1 = \dots = \mu_N = \rho\mathcal{L}^d$ be probability measures on $\R^d$ with density $\rho$ and $c:(\R^d)^N\to\R$ be a continuous cost function. Then, 
\[ \min(\MK) = \inf(\MN).\] 
\end{cor}

Notice that, in general, the equivalence between \eqref{pb:MKN} and \eqref{pb:MN} is not an immediate consequence of the theorem \ref{thm:PratelliMongeKantorovich}. In particular, remark that the image measure of $Y = \R^{d(N-1)}$ is not prescribed, but only its marginals.

Finally, we conclude this section by giving an example of a particular continuous cost function where the optimal coupling $\gamma$ is not deterministic, even in the $2$ particles case. The nonexistence of an optimal transport maps typically happens when minimizing sequences of $\gamma$ exhibit strong oscillations. 

\begin{exam}[G. Carlier]\label{carlier}
Suppose $\mu$ uniformly distributed in $[0,1]$, $\nu$ uniformly distributed in $[-1,1]$ and the continuous cost $c:\R^2\to \R$ given by
\[
c(x_1,x_2) = (x_1^2-x_2^2)^2.\]
Since $c \geq 0$ we have that an optimal plan for this cost is given by $\overline{\gamma} = \mu\otimes(\frac{1}{2}\delta_{x} + \frac{1}{2}\delta_{-x})$ since $\int c \, d \overline{\gamma}=0$. We want to show there doesn't exit a map $T$ such that the cost attains zero: suppose in fact that $\int (x^2-T(x)^2)^2 \, dx =0$ but then one should  have $T(x) \in \lbrace x, -x \rbrace$ almost everywhere. Hence,
\[
T(x) = (\chi_A - \chi_{[0,1]\backslash A})x
\]
for a measurable $A \subset [0,1]$. But then $T_{\sharp} \mu = \mathcal{L}^1|_{A} + \mathcal{L}^1|_{\tilde{A}} \neq \nu$, where $x \in \tilde{A}$ iff $-x\in [0,1] \setminus A$, and so we obtain a contradiction since $T$ wouldn't admissible.
It is however easy to construct admissible maps $T_n$, linear on $[i/2n, (i+1)/2n)$, such that $\int c(x_1,T_n(x_1))d\mu(x_1)$ tends to zero.

Remark that, we state by simplicity the example in the $2$-marginal case, but the same conclusion and arguments apply in the larger interactions case under the cost 
$$c(x_1,\dots,x_N) = \sum^N_{i=1}\sum^N_{j=i+1}(x_i^2-x_j^2)^2. $$

\end{exam} \medskip

\begin{figure}[htbp]

\begin{tabular}{@{}c@{\hspace{2mm}}c@{}}
\centering
\includegraphics[ scale=0.3]{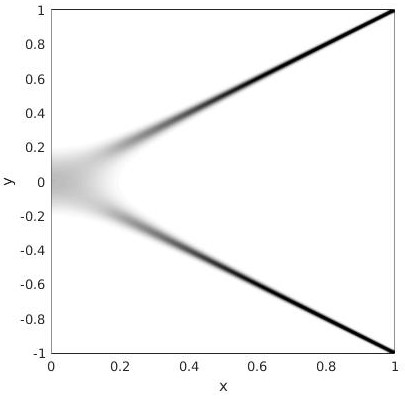}&
\includegraphics[ scale=0.3]{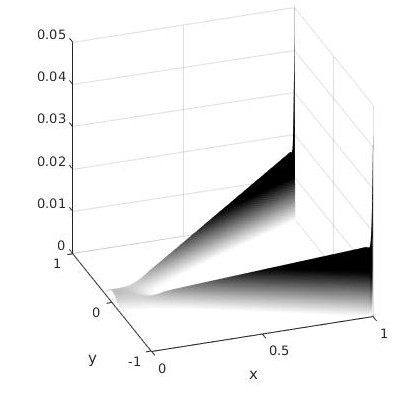}\\

\end{tabular}
\caption{\textit{(example \ref{carlier} ) Left: support of the optimal coupling $\overline{\gamma}$ . Right: optimal coupling $\overline{\gamma}$.
The simulation has been performed by using the iterative method described in section \ref{numerics}.
}}
\label{figure:carlier}
\end{figure}

\subsection{Dual formulation} We consider the formal convex dual formulation of \eqref{pb:MKN}

\begin{equation}\label{pb:KN}
(\KN) \quad \sup \Bigg\lbrace \int_{\R^d} \sum^N_{i=1}u_i(x_i)d\mu_i(x_i)
: \sum^N_{i=1}u_i(x_i) \leq c(x_1,\dotsc,x_N) \Bigg\rbrace,
\end{equation}
where $u_i \in L^1(\R^d,\mu_i)$, called Kantorovich potentials or simply potentials, are upper semicontinuous for all $i \in \lbrace 1,\dots,N\rbrace$.
This is often called the dual problem of (\ref{pb:MN}). 

Notice that, when all marginals are equal to $\mu_1 = \rho(x)\mathcal{L}^d$ and the cost is symmetric, we can assume the Kantorovich potentials $u_i(x_i)$ are all the same $u(x_i)$, and we can rewrite the constraint in \eqref{pb:KN} as 
\[
\sum^N_{i=1} u(x_i) \leq c(x_1,\dots,x_N). \medskip
\]

Among all Kantorovich potentials, a particular class is of special interest.

\begin{deff}[$c$-conjugate function]
Let $c:(\R^d)^N\to \R\cup\lbrace \infty \rbrace$ a Borel function. We say that the $N$-tuple of functions $(u_1,\dots,u_N)$ is a $c$-conjugate function for $c$, if
\[
u_i(x_i) = \inf \bigg\lbrace c(x_1,\dots,x_N) - \sum^N_{j=1, j\neq i}u_j(x_j) : x_j \in X_j \ , \ j\neq i \bigg\rbrace, \ \ \forall i = 1,\dots,N.
\]
\end{deff} \medskip

As we will see in the next theorem \ref{th:DualKellerer}, the optimal potentials $u_1,\dots,u_N$ in $(\mathcal{K}_N)$ are $c$-conjugates, and therefore semi-concave. So, they admit super-differentials $\overline{\partial}u_i(x_i)$ at each point $x_i \in \R^d$ and then, at least in the compact case, we can expect that for each $x_i \in \R^d$, there exists $(x_j)_{j\neq i} \in \R^{d(N-1)}$, such that 
\[
u_i(x_i) = c(x_1,\dots,x_N) - \sum_{j \neq i} u_j(x_j).
\] 

Moreover, if $c:(\R^d)^N\to\R$ is differentiable and $\vert u_i(x_i)\vert < \infty$, for some $x_i \in \R^d$, $u_i$ is locally Lipschitz. The well-known general result linking both \eqref{pb:MKN} and \eqref{pb:KN} has been proven by Kellerer in 1984.

\begin{teo}[Kellerer, \cite{Kel}]\label{th:DualKellerer}
Let $(X_1,\mu_1),\dots,(X_N,\mu_N)$ be Polish spaces equipped with Borel probability measures $\mu_1,\dots,\mu_N$. Consider $c:X_1\dots X_N \to \R$ a Borel cost function and assume that $\overline{c} = sup_X c < \infty$ and $\underline{c} = \inf_X c > - \infty$. Then,

\begin{itemize}
\item[(i)] There exists a solution $\gamma$ of the Monge-Kantorovich
problem \eqref{pb:MKN} and a $c$-conjugate solution $(u_1, u_2, ...,
u_N)$ of the dual problem (\ref{pb:KN}). 

\item[(ii)] ``Duality holds",
\be
\ds\inf \bigg\lbrace \int_{X_1\dots X_N}cd\gamma \ : \ \gamma \in \Pi(\mu_1,\dots,\mu_N) \bigg\rbrace = \sup \bigg\lbrace \sum^N_{i=1} \int_{X_i}u_i(x_i)d\mu_i : u_i \in \mathcal{D} \bigg\rbrace,
\ee
where $\mathcal{D}$ is the set of functions $u_i:\R^d\to\R$, $i=1,\dots,N$, such that  
\[
\sum^N_{i=1}u_i(x_i) \leq c(x_1,\dots,x_N), \quad \textit{and} \quad \dfrac{1}{N}\overline{c} - (\overline{c}-\underline{c})\leq \sum^N_{i=1}u_i(x_i)  \leq \frac{1}{N} \overline{c}.
\]

\item[(iii)] For any solution $\gamma$ of \eqref{pb:MKN}, any conjugate solution of \eqref{pb:KN} and $\gamma$ a.e. $(x_1,\dots,x_N)$, we have
\[
\sum^{N}_{i=1} u_i(x_i) = c(x_1,\dots,x_N).
\]
\end{itemize}
\end{teo}

We remark that despite its generality,  Kellerer's theorem  can not be applied directly to Coulomb-type costs, since Coulomb cost does not satisfies the boundedness hypothesis $\overline{c} = sup_X c < \infty$ and can not be bounded by functions in $L^1(\R^d,\rho\mathcal{L}^d)$.

A generalization of Kellerer's theorem due to M. Beiglboeck, C. Leonard and W. Schachermayer \cite{BeiLeoSch}, extends this duality for more general costs, but also cannot be applied in this context.

Recently, De Pascale \cite{DeP} extended the proof of duality which can be applied also for Coulomb-type costs, see Theorem \ref{lem:DualDeP}. In fact, suppose $\mu = \rho(x)dx$ a probability measure on $\R^d$ which does not give mass to sets of cardinality less than  or equal to $d-1$, then
\be\label{eq:DualDeP}
\min_{\gamma \in \Pi_N(\mu)} \int_{(\R^d)^N}\sum^N_{i=1}\sum^N_{j=i+1}\dfrac{1}{\vert x_i-x_j\vert}d\gamma = \max_{u \in \tilde{\mathcal{D}}} \int_{\R^d}Nu(x)\rho(x)dx
\ee
where $\tilde{\mathcal{D}}$ is the set of potentials  $u \in L^1(\R^d,\mu)$ such that
\be\label{eq:ColConstaintDual}
\sum^N_{i=1}u(x_i) \leq \sum^N_{i=1}\sum^N_{j=i+1}\dfrac{1}{\vert x_j-x_i\vert}, \quad \otimes^N_{i=1}\rho-\textit{ almost everywhere.}
\ee
Moreover, we get the conclusions $(i)$ and $(iii)$ of the theorem \ref{th:DualKellerer}. 
Physically, the constraint \eqref{eq:ColConstaintDual} means that, at optimality, the allowed $Nd$ configuration space should be a minimum of the classical potential energy.

\begin{oss} More generally, the main theorem in \cite{DeP} shows that the Kantorovich duality holds, for instance, for costs of form 
\[
\sum^N_{i=1}\sum^N_{j=i+1}\dfrac{1}{\vert x_j-x_i\vert^s}, \quad s\geq 1.
\]
\end{oss}

\medskip

\subsection{Geometry of the Optimal Transport sets} $\quad$ Our goal in this section is to present the relation between the support of a coupling $\gamma$ and optimality in the Monge-Kantorovich problem \eqref{pb:MKN}. In the following, we will just summarize key results necessary to present recent developments of optimal transportation with Coulomb and repulsive harmonic-type costs. 

Roughly speaking, it is well-known in optimal transport theory with $2$-marginals that, for a wide class of costs $c$, a coupling $\gamma$ is optimal if, and only if, the support of $\gamma$ is concentrated in a $c$-cyclically monotone set \cite{AGS, Pra07, SchTei}:

\begin{deff}[$c$-cyclically monotone, $2$-marginal case]\label{cmonotone2marginals}
Let $X,Y$ be Polish spaces, $c:X\times Y \to \R\cup\lbrace\infty\rbrace$ be a cost function and $\Gamma$ be subset of the product space $X\times Y$. We say that $\Gamma$ is a $c$-cyclically monotone set if, for any finite couples of points $\lbrace (x_i,y_i)\rbrace^{M}_{i=1} \subset \Gamma$,
\[
\sum^M_{i=1}c(x_i,y_i) \leq \sum^M_{i=1} c(x_i,y_{\sigma(i)}). \medskip
\]
\end{deff} 

In \cite{Passthesis}, Pass suggested a possible extension of the concept of $c$-cyclically monotone sets for the multi-marginal case.  For our proposes, it will be enough to consider this notion in $\R^d$.

\begin{deff}[$c$-cyclically monotone set]\label{def:setcyclicmonotone} 

Let $c:\R^{dN}\to \R$ be a cost function. A subset $\Gamma \subset \R^{dN}$ is said to be $c$-monotone with respect to a partition $(p,p^c), p \subset I$, if, for all $x = (x_1,\dots,x_N), y = (y_1,\dots, y_N) \in \Gamma$
\[\label{eq:cciclicalymonotonicity}
c(x)+c(y) \leq c(X(x,y,p)) + c(Y(x,y,p)),
\]
where $X(x,y,p),Y(x,y,p) \in \R^N$ are functions obtained from $x$ and $y$ by exchanging their coordinates on the complement of $p$, namely
\[
X_i(x,y,p) = 
\Bigg\lbrace\begin{array}{ll}
x_i ,\text{     if     } i \in p\\
y_i, \text{     if     } i \in p^c
\end{array}
\quad \text{and} \quad
Y_i(x,y,p) = 
\Bigg\lbrace\begin{array}{ll}
y_i ,\text{     if     } i \in p\\
x_i, \text{     if     } i \in p^c
\end{array} 
\quad \forall i \in I
 \]  
 
We say that $\Gamma \subset \R^{dN}$ is $c$-cyclically monotone, if it is $c$-monotone with respect to any partition $(p,p^c)$, $p\subset I$.
\end{deff}
\medskip

The notion of $c$-cyclical monotonicity in the multi marginal setting was studied by Ghoussoub, Moameni, Maurey, Pass \cite{GhoMau,GhoMoa, Passthesis, Pass2} and successfully applied in the study of decoupling PDE systems \cite{GhoPass}. For Coulomb costs, it was useful in order to characterize the optimal transport maps in the one-dimensional case \cite{CoDePDMa}.

More recently, Beiglboeck and Griessler \cite{BeiGri}, presented a new notion called \textit{finistic optimality} inspired by the martingale optimal transport, which is equivalent, in the $2$-marginals case, to the definition \ref{cmonotone2marginals}.

The next proposition gives a necessary condition on the support for an optimal transport plan $\gamma$. For the proof, which is based on the 2-marginal result \cite{SmiKno}, we refer the reader to Lemma 3.1.2 of \cite{Passthesis}.

\begin{prop}[Support of transport plan are $c$-cyclically monotonic]\label{prop:optccyclicmonotone}
Let $c:(\R^{d})^N\to\II$ be a continuous cost and $\mu_1,\dots,\mu_N$ absolutely continuous probability measure on $\R^d$. Suppose that $\gamma \in \Pi((\R^d)^N,\mu_1,\dots,\mu_N)$ is an optimal transport plan for Multimarginal Monge-Kantorovich Problem ($\ref{pb:MKN}$) and assume $(\MK) < \infty$. Then $\gamma$ is concentrated in a $c$-cyclically monotone set.
\end{prop}

\subsection{Symmetric Case} \label{sec:sym}

$\quad$ We are going to show in this section that we can reduce our study of the Monge-Kantorovich \eqref{pb:MKN} and the Monge problems \eqref{pb:MN}, respectively, to a suitable class of transport maps (called cyclic maps) and transport plans (symmetric plans). Having in mind the DFT-OT problem, it is natural to suppose that the cost functions does not change if there are some symmetries among the configuration system of particles or, equivalently, we can suppose the couplings $\gamma \in \Pi_N(\mu)$ of those particles are cyclically symmetric. 

\begin{deff}[Cyclically symmetric probability measures]
A probability measure $\gamma \in \P(\R^{dN})$ is cyclically symmetric if 
\[
\int_{\R^{dN}}f(x_1,\dots,x_N)d\gamma = \int_{\R^{dN}} f(\overline{\sigma}(x_1,\dots,x_N))d\gamma, \text{ } \forall f \in  \C(\R^{Nd})
\]
where $\overline{\sigma}$ is the cyclical permutation
\[
\overline{\sigma}(x_1,x_2,\dots,x_N) = (x_2,x_3,\dots,x_N,x_1). 
\] 
We will denote by $\Pi^{sym}_N(\mu_1)$, the space of all symmetric probability measures on $\R^{dN}$ having the same first marginal $\mu_1$.
\end{deff}

Notice that if $(\mu_1,\dots,\mu_N)$ are the marginals of $\gamma \in \Pi^{sym}(\R^{dN},\mu_1)$, then all marginals $\mu_1,\dots,\mu_N$ are necessarily the same: $\forall \ f \in \C(\R^d)$,

\[
\begin{array}{ll}
\ds\forall i \in I, \quad \int_{\R^{d}}f(x_i)d\mu_i(x_i)
&=\ds \int_{\R^{dN}}\tilde{f}(x_1,\dots,x_N)d\gamma(x_1,\dots,x_N) \medskip \\ 
&=\ds \int_{\R^{dN}}\tilde{f}(\overline{\sigma}^{1-i}(x_1,\dots,x_N))d\gamma \medskip \\
 &=\ds\int_{\R^{d}}f(x_1)d\mu_i(x_1), $\medskip$ \quad 
 \\  \end{array} \]

where $\tilde{f}(x_1,\dots,x_N) = f(x_i)$. Briefly, 
\[
(e_i)_{\sharp}\gamma = \mu_1, \space \forall i \in I, 
\]
where $e_i$ is the canonical projection.

As previously, we can define the multi marginal symmetric Monge-Kantorovich Problem as follows:
\be\label{pb:MKsym}
(\mathcal{MK}_{sym})  \quad \inf_{\gamma \in \Pi^{sym}_N(\mu)} \int_{\R^{dN}} c(x_1,\dots,x_N)d\gamma(x_1,\dots,x_N),
\ee
and the associated multi marginal cyclically Monge Problem
\be \label{pb:MKcyc}
(\mathcal{M}_{cycl}) \quad \inf_{{\substack{T_{\sharp}\mu = \mu ,\\ T^{(N)}=I}}} \ \int_{\R^{d}}c(x,T^{(1)}(x),T^{(2)}(x), \dots,T^{(N-1)}(x)), 
\ee
Notice that, $(\mathcal{M}_{cycl})$ is related to $(\mathcal{M})$ by choosing  $T_{\sharp}\mu = \mu$, $T_{i+1}(x) = T^{(i)}(x)$ for every $i \in \lbrace 1,\dots,N-1\rbrace, T^{(N)} = Id$.

\begin{prop} Suppose $\mu \in \P(\R^d)$ is an absolutely continuous probability measure with respect to the Lebesgue measure and $c:(\R^d)^N\to\R$ is a continuous symmetric cost. Then,
\[ \inf(\MK) = \inf(\MK_{sym}). \]
\end{prop}
\begin{proof} The fact that the infimum of (\ref{pb:MKN}) is less than or equal to (\ref{pb:MKsym}) is obvious. Now, we need to show that for every transport plan $\gamma \in \Pi_N(\mu)$ we can associate a symmetric plan in $\Pi^{sym}_N(\mu)$. Indeed, given $\gamma \in \Pi_N(\mu)$, we define
$$ \overline{\gamma} = \dfrac{1}{N!}\sum_{\sigma\in \S_N}\sigma_{\sharp}\gamma, $$
and finally we notice that $\overline{\gamma}$ has the same cost as $\gamma$.
\end{proof}

We remark that prove  $\inf(\mathcal{M}_{cycl}) \geq \inf(\MN)$  is not obvious . However, we can avoid this problem in order to prove the equivalence between the Multi-marginal cyclic problem and Multi-marginal OT problem, by noticing that $\inf(\mathcal{M}_{sym} \leq \inf(\M_{cycl})$. The last part was proved by M. Colombo and S. Di Marino \cite{CoDMa}. We give the precise statement of this theorem.

\begin{teo}[M. Colombo and S. Di Marino, \cite{CoDMa}]\label{CoDMa}
Let $c:(\R^d)^N\to \R$ be a continuous function and bounded from below. If $\mu$ has no atoms, then 
\[ (\mathcal{MK}_{sym}) = (\mathcal{M}_{cycl}).  \]
\end{teo}

\medskip

Finally, we have the equality between the Multi-marginal cyclic problem and Multi-marginal OT problem.

\begin{cor}\label{cor:CoDMa}
If $\mu$ has no atoms and $c:(\R^d)^N\to\R$ is a continuous function and bounded from below, then  
$$ \inf(\mathcal{MK}_{sym}) = \inf(\mathcal{MK}) = \inf(\MN) = \inf(\mathcal{M}_{cycl}). $$
\end{cor}
\medskip

The last result of this section is due Pass \cite{Pass4}, which states that optimal transport plans for symmetric costs can be supported on sets of Hausdorff dimension equal or bigger than $2N-2$. 

\begin{teo}[Pass, \cite{Pass4}]\label{th:sym:Pass} Let $c:(\R^d)^N\to \R$ be a symmetric cost and $\lbrace \mu_i\rbrace^N_{i=1}$ be radially symmetric and absolutely continuous with respect to Lebesgue measure in $\R^d$. Suppose that for every radii $(r_1, ...r_N)$ the minimizers 
\[
(x_1,\dots,x_N) \in \operatorname{argmin}_{\vert y_i \vert = r_i} c(y_1, y_2, \dots y_N) \quad \textit{are not all co-linear.}
\]
Then, there exists solutions $\gamma$ for the Monge-Kantorovich symmetric problem \eqref{pb:MKsym} whose support is at least $(2N-2)$-dimensional.
\end{teo}

%% file: multimarginalOTproblemwithCoulombCost.tex

\section{Multi-marginal OT with Coulomb Cost}\label{sec:Coulomb}

$\quad$ This section is devoted to the summary of the results present in literature on the multi marginal optimal transport with Coulomb cost. We already highlighted the fact that the interesting case is when all the marginals are equal, since this is the physical case. We point out that, as always, the focus will be on characterizing the optimal plans as well as looking for cyclical transport maps: in this direction an important conjecture has been made by Seidl in its seminal paper \cite{Sei}, where he gave an explicit construction of a map for measures with radial symmetry in $\R^d$. In the last years the conjecture was proven to be right: first in the 1D case when $\mu$ is the uniform measure on an interval, and then generalized to any diffuse measure on the real line. However quite recently various authors disproved the conjecture in the case $\R^d$ for $d \geq 2$ and for $N=3$ marginals, yet the Seidl map is still optimal for some class of measures.

We begin by analyzing the general problem, and then we will proceed to look at the radial case.

\subsection{General theory: duality, equivalent formulations and many particles limit}

We begin by stating the duality theorem in this case proved by De Pascale \cite{DeP}, by means of $\Gamma$-convergence of finite dimensional problems for which the duality is classical.

\begin{teo}\label{lem:DualDeP} Let $\mu \in \mathcal{P}(\R^d)$ be a measure that is not concentrated on a set of cardinality smaller or equal than $N-1$. Then the duality \eqref{eq:DualDeP} holds and the dual problem has a bounded maximizer.
\end{teo}

The boundedness of the maximizer $u$ can let us prove also some regularity properties. The first one is the fact that any optimal plan is supported away from the diagonals, while the second one proves second order regularity of $u$.

\begin{lemma}[Corollary 3.13 in \cite{DeP}]\label{lem:diagonal} Let $\mu$ as in Theorem \ref{lem:DualDeP}; then there exists $\alpha >0$ such that for every optimal plan $\gamma$ we have ${\rm supp}(\gamma) \subseteq \{ (x_1, \ldots, x_N) \; : \; |x_i - x_j| \geq \alpha,  \; \forall \, i\neq j \}$.
\end{lemma}

\begin{lemma}\label{lem:regularitypot} Let $\mu$ as in Theorem \ref{lem:DualDeP} and $u$ be a bounded maximizer in the dual problem in \eqref{eq:DualDeP}; then $u$ has a $\mu$-representative that is Lipschitz and semi concave.
\end{lemma}

\begin{proof} As in the classical case we first prove a structural property of the maximizer $u$, namely that the constraint \eqref{eq:ColConstaintDual} is saturated, that is $(u,\ldots, u)$ is a $c$-conjugate $N$-tuple.
Let  us consider  
$$ w(x_1) = \operatorname*{ess-inf}_{ (x_2, \ldots, x_N) \in \R^{d(N-1)} } \left\{ \sum_{1 \leq i<j \leq N } \frac 1{ |x_i - x_j | }  - \sum_{ i=2}^N u (x_i) \right\}, $$
where the ess-inf is made with respect to the measure $\mu^{\otimes N-1}$. It is obvious that $w(x_1) \geq u(x_1)$ for $\mu$-a.e $x_1$, thanks to \eqref{eq:ColConstaintDual}.  Suppose that $w >u$ in a set of positive measure; but then $(w,u, \ldots, u)$ would be a better competitor in the (not symmetric) dual problem contradicting the fact that $u$ is a maximizer (we use the fact that the symmetric dual problem and the not symmetric one have the same values). Now we have that $w$ is defined everywhere and so we can talk about its regularity.

Let $d>0$ be a number such that there exists points $p_1, \ldots p_N$ in the support of $\mu$ such that $|p_i- p_j| \geq d$; this number exists as long as $\sharp \, {\rm supp} (\mu ) \geq N$. Let us fix $\ep >0$ such that $\frac 1 {2\ep} >  \frac {2N(N-1)}d  +2N \| u \|_{\infty}$ and $\ep \leq d/8$; in this way it is true that if we define
$$ w_{x_0} (x_1) = \operatorname*{ess-inf}_{ |x_i-x_0| \geq \ep, \, i=2, \ldots, N } \left\{ \sum_{1 \leq i<j \leq N } \frac 1{ |x_i - x_j | }  - \sum_{ i=2}^N u (x_i) \right\}, $$
we have $w_{x_0}=w$ on $B(x_0, \ep)$. In fact in the definition of $w$ we can choose $x_i$ among the $p_i$ (or very close to them, as they belong to the support) such that $|x_i-x_j|> d/2 - 2\ep \geq d/4$ for every $i \neq j$, and so we have that $w \leq  \frac {2N(N-1)}d  +N \| u \|_{\infty} $ but then it is clear that for any $(x_2, \ldots, x_N)$ such that $|x_i-x_0| \leq \ep $ we will have 
$$\sum_{1 \leq i<j \leq N } \frac 1{ |x_i - x_j | }  - \sum_{ i=2}^N u (x_i) \geq \frac 1{2\ep} > w(x_1) \qquad \forall x_1 \in B(x_0, \ep);$$
this proves that in fact $w_{x_0}=w$ on the set $B(x_0, \ep)$ and so in particular also in $B(x_0, \ep/2)$. But in this set $w_{x_0}$ is Lipschitz and semi concave since it is an infimum of uniformly $C^{\infty}$ functions on $B(x_0, \ep/2)$. Moreover the bounds on the first and second derivatives don't depend on $x_0$ but only on $\ep$, that is fixed a priori, and so by a covering argument we obtain the thesis.
\end{proof}

Another interesting reformulation of the Coulomb-like problem, or more generally when we have only interaction between two particles, can be found in \cite{FMPCK} where, seeking a dimensional reduction of the problem, the authors use the fact that if $\gamma \in \Pi_N (\mu)$ is a symmetric plan then
$$ \int_{\R^{Nd}} \sum_{ 1 \leq i < j \leq N} \frac 1{|x_i - x_j |} \, d \gamma = {N \choose 2} \int_{\R^{2d}} \frac 1{|x-y|} \, d (e_1,e_2)_{\sharp} \gamma, $$
that is, since the electron are indistinguishable, it is sufficient to look at the potential energy of a couple of electrons and then multiply it by the number of couples of electron. It is clear that $\gamma_2 = (e_1 , e_2)_{\sharp} \gamma$ is a $2$-plan whose marginal are still $\mu$; we will say that $\eta \in \Pi_2(\mu)$ is $N$-representable whenever it exists $\gamma \in \Pi_N(\mu)$ such that $\eta=(e_1,e_2)_{\sharp} \gamma$. The equivalent formulation they give is

$$ \min_{\gamma \in \Pi_N(\mu)} \int_{\R^{Nd}} \sum_{ 1 \leq i < j \leq N} \frac 1{|x_i - x_j |} \, d \gamma = \min_{\substack{\eta \in \Pi_2(\mu)  \\ \eta \text{ is $N$-representable}}} {N \choose 2} \int_{\R^{2d}} \frac 1{|x-y|} \, d \eta. $$

Unfortunately the conditions for being $N$-representable are not explicit and they are very difficult to find; this is correlated to the $N$-representability for matrices, but here, since we are in the semiclassical limit, we deal with densities instead. In \cite{FMPCK} the authors propose, as a method for reducing the dimension of the problem, to substitute the condition of being $N$-representable with that of being $k$-representable with $k\leq N$; the resulting problem will give a lower bound to the SCE functional . 

We present also a general theorem regarding the many particles limit, that embodies the well-known fact that when $N \to \infty$ then the solution is the mean field one

\begin{teo}[\cite{CotFriPass}] Let $\mu \in \mathcal{P}(\R^d)$. Then we have that
$$ \lim_{N \to \infty} \frac 1{{N \choose 2}} \min_{\gamma \in \Pi_N (\mu)} \int_{\R^{Nd}} \sum_{1 \leq i<j \leq N} \frac 1{|x_i-x_j|} \, d \gamma = \int_{\R^{2d} } \frac 1 { |x-y|} \, d \mu \otimes \mu.$$
In terms of DFT we are saying that $E^{OT}(\mu) = {N \choose 2} E^{MF}(\mu) + o(N^2)$, where $E^{MF}$ is the normalized mean field energy $E^{MF}(\mu)= \int \frac 1{|x-y|} \, d \mu(x) \, d \mu(y)$.
\end{teo}

\begin{oss} The statement about the Coulomb cost in the physical case is quite classical. In fact, for a measure $\rho \in L^{4/3}(\R^3)$ the Lieb-Oxford  bound holds \cite{Lieb79, LiebOxford}:
$${ N \choose 2} E^{MF}(\rho) \geq E^{OT}(\rho) \geq N^2 E^{MF}(\rho) - C N^{4/3} \int_{\R^3} \rho^{4/3}(x) \, d x,$$ and so the conclusion is immediate. However, in \cite{CotFriPass} the proof is completely different and relies on the fact that a measure $\gamma \in \Pi_2(\mu)$ that is $N$-representable for every $N$ must be in the convex envelope of measures of the type $\nu \otimes \nu$, and then on a direct computation using the Fourier transform. In particular, aside from the Coulomb cost in dimension $3$ they prove the theorem for more general costs with binary interaction of the form $\sum_{ 1 \leq i < j \leq N} l( x_i -x_j)$, where $l \in C_b(\R^d) \cap L^1(\R^d)$ satisfies $l(z)=l(-z)$ and also $\hat{l} \geq 0$.
\end{oss}

\subsection{The Monge problem: deterministic examples and counterexamples}

In this section we will illustrate the cases in which we know that there exists a deterministic solution; in those cases moreover there is also a result of uniqueness. We remark that these are the two extreme case, namely the case $N=2$ (and any $d$) and the case $d=1$ (and any $N$). We will just sketch the proofs in order to make clear the method used here, and why it is difficult to generalize it. We begin with the $2$-particles case, in every dimension $d$: this result was proved in \cite{CFK} by means of standard optimal transport techniques.  

\begin{teo}\label{th:cotar} Let $\mu \in \mathcal{P}(\R^d)$ be a probability measure that is absolutely continuous with respect to the Lebesgue measure. Then there exists a unique optimal plan $\gamma_O \in \Pi_2 (\mu)$ for the problem
$$ \min_{ \gamma \in \Pi_2(\mu) } \int_{\R^{2d} } \frac 1{|x_1-x_2|} \, d  \gamma. $$
Moreover this plan is induced by an optimal map $T$, that is, $\gamma= (Id,T)_{\sharp} \mu$, and $T(x)=x +\frac{ \nabla \phi } {| \nabla \phi|^{3/2} }$ $\mu$-almost everywhere, where $\phi$ is a Lipschitz maximizer for the dual problem.
\end{teo}

\begin{proof} Let us consider $\gamma$ a minimizer for the primal problem and $\phi$ a bounded and Lipschitz  maximizer of the dual problem (it exists thanks to Lemma \ref{lem:regularitypot}). Then we know that  
$$ F(x_1,x_2) = \phi(x_1) + \phi(x_2) - \frac{1}{|x_1-x_2|} \leq 0 \qquad \text{ for $\mu \otimes \mu$-almost every $(x_1,x_2)$,}$$
and we know also that $F=0$ $\gamma$-almost everywhere. But then $F$ has a maximum on the support of $\gamma$ and so $\nabla F=0$ in this set; in particular we have that $\nabla \phi (x_1) = - \frac{ x_1 - x_2}{ |x_1- x_2|^3}$ on the support of $\gamma$. Finding $x_2$ we have that $x_2=x_1 - \frac { \nabla \phi }{| \nabla \phi|^{ 3/2}} (x_1) = T(x_1)$ on the support of $\gamma$ and this implies $\gamma = (Id,T)_{\sharp}\mu$ as we wanted to show.
\end{proof}

The first positive $N$-marginal result for the Coulomb cost is instead shown in \cite{CoDePDMa} where, in dimension $d=1$, the authors can prove that for non-atomic measure an optimal plan is always ``induced'' by a cyclical optimal map $T$. 

\begin{teo}\label{teo:1DN} Let $\mu \in \mathcal{P}(\R)$ be a diffuse probability measure. Then there exists a unique optimal symmetric plan $\gamma_O \in \Pi^{sym}_2 (\mu)$ that solves
$$ \min_{ \gamma \in \Pi^{sym}_N(\mu) } \int_{\R^{N} } \sum_{1 \leq i < j \leq N}  \frac 1{|x_j-x_i|} \, d  \gamma. $$
Moreover this plan is induced by an optimal cyclical map $T$, that is, $\gamma_O=\frac 1{N!} \sum_{\sigma \in \S_N} \sigma_{\sharp} \gamma_T$, where $\gamma_T=(Id,T,T^{(2)} , \ldots, T^{(N-1)})_{\sharp} \mu$. An explicit optimal cyclical map is  
$$T(x) =\begin{cases}  F_{\mu}^{-1} (F_{\mu}(x) + 1/N) \qquad & \text{ if }F_{\mu}(x) \leq (N-1)/N \\ F_{\mu}^{-1} ( F_{\mu}(x) +1 - 1/N ) & \text{ otherwise.} \end{cases}$$
Here $F_{\mu}(x)=\mu ( -\infty , x]$ is the distribution function of $\mu$, and $F_{\mu}^{-1}$ is its lower semicontinuous left inverse. 
\end{teo}

\begin{proof} We begin by observing that if $\gamma$ is a symmetric optimal plan then a stronger statement than Proposition \ref{prop:optccyclicmonotone} holds, namely we have that for every $x,y \in {\rm supp }(\gamma)$:
$$ c(x)+ c(y) = \min \{  c(X(x,\sigma(y),p)) + c(Y(x,\sigma(y),p)) : \; \forall p \subset \{1, \ldots ,N\} , \forall \sigma \in \S_N \},$$
where $c$ is the Coulomb cost on the $N$-tuple of points $x=(x_1, \ldots, x_N)$ and $\sigma$ acts on the indices.
The key point here is that one can show (Proposition 2.4 \cite{CoDePDMa}) that this property holds if and only if a more geometric condition holds true for $x$ and $y$: they are well-ordered. This property amounts to the fact that, up to permute the coordinates of both points, we have
$$ x_1 \leq y_1 \leq x_2 \leq \cdots \leq x_N \leq y_N,$$
or the other way around. Once this property is proven the rest of proof is rather straightforward, since we proved that the support of $(e_1,e_2)_{\sharp}\gamma|_D$, where $D = \{ x_i \leq x_j \text{ and } y_i \leq y_j, \, \forall \, i< j\}$, {\color{ red} is  monotone in the usual sense}.
We remark that the key property does not easily generalize to $d>1$ since it (heavily) uses the ordered structure of $\R$.
 \end{proof}

Let us notice that in this case we don't have uniqueness of optimal plan neither of cyclical optimal maps, as pointed out in Remark 1.2 in \cite{CoDePDMa}.

For the general case we thus have to assume $d>1$ and $N>2$: from now on we will look at the case in which the measure $\mu$ has radial symmetry. We can easily see that in this case we can reduce our problem to a multi marginal problem in $\R$.

\begin{lemma}[Radial case]\label{lemma:radialcase} Let $\mu \in \mathcal{P}(\R^d)$ be a radially symmetric measure. In particular $\mu$ is determined by $\mu_r= | \cdot |_{\sharp} \mu$. Then for every optimal plan $\gamma \in \Pi_N(\mu)$, if we consider $\gamma_r=R_{\sharp} \gamma$, where $R: (x_1, \ldots , x_N ) \mapsto (|x_1|, \ldots , |x_N|)$, we have that $\gamma_r$ is an optimal plan for the $1D$ multi marginal problem
\begin{equation}\label{eq:red_rad}
 \min_{ \eta \in \Pi_N(\mu_r) } \int_{\R^N} c_1(r_1, \ldots, r_N ) \, d \eta.
 \end{equation}
where  $c_1$ is the reduced Coulomb cost 
$$c_1(r_1, \ldots, r_N ) = \min \left\{ \sum_{1\leq i<j \leq N}  \frac 1{| x_j - x_i |} \; : \; |x_i|=r_i \; \forall i=1, \ldots, N \right\}.$$
Moreover $ \int_{\R^{dN} } \sum_{ 1 \leq i < j \leq N} \frac 1 {|x_i - x_j|} \, d \gamma = \int_{\R^d} c_1( r_1, \ldots r_N ) \, d \gamma_r$.
\end{lemma}

Given the one dimensional character of this problem, one could expect that the solution is similar to the one depicted in Theorem \ref{teo:1DN}. In fact in \cite{SeiGorSav} the authors conjecture a similar structure:

\begin{conj}[Strong Seidl Conjencture]\label{conj.Seidl} Let $\mu \in \mathcal{P}(\R^d)$ be an absolutely continuous  measure with respect to the Lebesgue measure, with radial symmetry, and let $\mu_r=| \cdot |_{\sharp} \mu$.  Let $0=r_0 <r_1< \ldots < r_{N-1}< r_N= \infty$ such that the intervals $A_i=[r_i,r_{i+1})$ have all the same radial measure $\mu_r(A_i) = 1/N$. Then let $F(r)= \mu_r (0,r]$ be the cumulative radial function and let $S: [0, \infty) \to [0,\infty)$ be defined piecewise such that the interval $A_i$ is sent in the interval $A_{i+1}$ in an anti monotone way:
$$ S(r) =F^{-1} (  2i/N-F(r)) \qquad \text{ if }r_{i-1} \leq r < r_i \text{ and }i <N$$
$$ F(S(r)) = \begin{cases} F^{-1}(F(r) + 1/N -1) \quad & \text{ if $N$ is even}\\ F^{-1}(1 - F(r)) & \text{ if $N$ is odd, } \end{cases} \qquad \text{ if }r_{N-1} \leq r < r_N. $$
Then $S$ is an optimal cyclical map for the problem \eqref{eq:red_rad}.
\end{conj}

However in the recent papers \cite{CoStra, DMaGGNe} this conjecture is proven to be wrong, looking at radial measures $\mu$ concentrated in some thin annulus $\{ 1- \delta \leq |x| \leq 1+\delta\}$; in particular in \cite{DMaGGNe} the $1D$ problem is proven to be equivalent, when $\delta \to 0$, to the repulsive harmonic one, for which high non-uniqueness holds and the Seidl map is not optimal (see the following section).
However in \cite{CoStra} also a positive example is found, namely a class of measures for which the conjecture holds.

These results show that the solution to the multi marginal problem with Coulomb cost is far from being understood: in particular there is no clear condition on the marginal $\mu$ for the strong Seidl conjecture to hold.

%% file: multimarginalOTproblemwithRepulsiveHarmonicCost.tex

\section{Multi-marginal OT with repulsive Harmonic Cost}
\label{sec:OTHarm}

$\quad$ This section is devoted to the study of the repulsive harmonic cost. In DFT-OT problem (see section
\ref{sec:why}), we replace the electron-electron Coulomb interaction  by a weak force which decreases with the square distance of the particles. The term \textit{weak} comes from the fact that, in the repulsive harmonic cost the, interaction cost function has value zero when the particles overlap, instead of $\infty$ in the Coulomb case. 

More precisely, we are interested in characterizing the minimizers of the following problem
\begin{equation}\label{OTHarm:pb:MKweak}
(\MK_{weak}) \quad \min_{ \gamma \in\Pi((\R^d)^N,\mu_1,\dots,\mu_N)}
\ds \int_{(\R^d)^N}\ds\sum^N_{i,j=1}-\vert x_j-x_i\vert^2 d\gamma(x_1,\dots,x_N),
\end{equation}
where $\mu_1,\dots,\mu_N$ are absolutely continuous probability measures in $\R^d$. 

From a mathematical viewpoint this cost has some advantages compared to the Coulomb one, since here we can do explicit examples. We can interpret the solutions of \eqref{OTHarm:pb:MKweak} as an \textit{ansatz} for DFT problem (Hohenberg-Kohn functional \eqref{eq.FHK}) for particles interacting under the repulsive harmonic potential. From a technical aspect, this could be an interesting toy model to approach the case of Coulomb cost. From applications, some minimizers of \eqref{OTHarm:pb:MKweak} seem to have no particular relevance in physics because, as we will see in the examples \ref{OTHarm:ex:breathing2marginals}, \ref{OTHarm:ex:evenmarginals} and \ref{OTHarm:ex:2Np} below, certain optimal $\gamma$ of \eqref{OTHarm:pb:MKweak} allow particles to overlap.

%

We will see that this problem has a very rich structure, that is very different from the classical $2$-marginal case. First of all we notice that minimizers of this problem are also minimizers of the problem with the cost $c(x_1, \ldots, x_n) = | x_1 + \ldots + x_n|^2$; in fact we have that
$$ \int_{\R^{dN}} \sum_{i,j=1}^N - \vert x_i - x_j \vert ^2 \, d \gamma = 2\int_{\R^{dN}} c \, d \gamma - (N+1) \sum_{i=1}^N \int_{\R^d} |x|^2 \, d \mu_i, $$
but this last additive term depends only on the marginals and not on the specific plan $\gamma$. The cost $c$ is very particular since it has a wide class of ``trivial'' optimal plans, that is the ones that are concentrated on $x_1 + \ldots + x_n=0$; however the structure is very rich, see Lemma \ref{OTHarm:stuplemma}.

Concerning the existence of minimizers of  \eqref{OTHarm:pb:MKweak}, we will assume that the measures $\mu_i$ have finite second moments; then existence follows immediately  from the equality
$$
\ds \underset{\gamma}{\operatorname{argmin}} \int -\sum^N_{i=1}\sum^N_{j=i+1}\vert x_j-x_i\vert^2 d\gamma =  \underset{\gamma}{\rm argmin} \int \vert x_1 + \dots + x_N \vert^2 d\gamma. 
$$
In particular, Corollary \ref{cor:CoDMa} holds in this case too. Notice that the fact that the repulsive harmonic cost is smooth and has linear gradient does not make the Multi-marginal Optimal Transportation problem easier compared to the Coulomb cost. In fact, in this case the problem is that if we write down the optimality conditions for the potentials, in the case described in Lemma \ref{OTHarm:stuplemma}, we simply find the condition $x_1+ \ldots +x_N = const$, since in this case the potentials are linear functions.

However, we can enjoy the symmetries of this problems and build easy Monge solutions for some particular cases of \eqref{OTHarm:pb:MKweak}, see examples \ref{OTHarm:ex:3p}, \ref{OTHarm:ex:breathing2marginals}, \ref{OTHarm:ex:evenmarginals}, \ref{OTHarm:ex:2Np} and \ref{OTHarm:ex:trueplan} below.

Before stating the main result in the multi-marginal setting, we start analyzing the problem \eqref{OTHarm:pb:MKweak} in the $2$-marginals case, where everything seems to work fine, just as in the square distance case.

\begin{prop}
Let $\mu,\nu \in \mathcal{P}(\R^d)$ and $c_w(x,y) = -\vert x-y\vert^2$ be the opposite of the square-distance in $\R^d$. Suppose that $\mu$ is an absolutely continuous with respect to the Lebesgue measure in $\R^d$ and $\nu$ has no atoms. Then, there exists a unique optimal transport map $T:\R^d\to\R^d$ for the problem
\[
\ds \min \bigg\lbrace\int c_w(x,y) d\gamma \ : \ \gamma \in \Pi_2(\R^{2d},\mu,\nu)\bigg\rbrace = \inf \bigg\lbrace \int c_w(x,T(x)) d \mu \ : \ T_{\sharp} \mu =\nu\bigg\rbrace 
\]
Moreover, $T = \nabla \phi$, where $\phi:\R^d\to\R$ is a  concave function and there exists a unique optimal transport plan $\overline{\gamma}$, that is $\overline{\gamma} = (\Id\times T)_{\sharp}\mu$ . 
\end{prop}

\begin{proof}
This result is a easy consequence of the Brenier's theorem. Indeed, it is enough to verify that, taking $C= 2 \int |x|^2\, d \mu  + 2 \int |x|^2 \, d \nu $, we have 
\[
 C + \underset{T_{\sharp} \mu =\nu}{\inf}  \int -\vert x-T(x) \vert^2 d \mu(x) =  \underset{ G_{\sharp} \mu =\tilde{\nu}}{\inf} \int \vert  x - G(x) \vert^2  d\mu(x)
\]
where $G = -T$ and $\tilde \nu = (- Id)_{\sharp} \nu$. Then by Brenier's theorem there exists an unique optimal map $G$ which can be written as $G(x) = \nabla \psi(x)$, and $\psi:\R^d\to\R$ is a convex map. In other words, $T$ is a gradient of a concave function. 
\end{proof}

Notice that if we suppose $\mu = f(x)dx$ and $\nu = g(x)dx$ are probability measures with densities $f,g$ concentrated, respectively, in convex sets $\Omega_f,\Omega_g$ and assume there exits a constant $\lambda >0$ such that $\lambda \leq f,g \leq 1/\lambda$. Then, $T$ is a $C^{1,\alpha}$ function inside $\Omega_f$ \cite{Caf,Caf2}.

\begin{exam}[2 marginals case, uniform measure in the $d$-dimensional cube] \label{OTHarm:Exam2Marginals}
Suppose that $\mu = \nu = \mathcal{L}\big|_{[0,1]^d}$. In this case, we can verify easily that the optimal map $T:[0,1]^d\to [0,1]^d$ is given by the anti-monotone map $T(x) = (1,\ldots,1)-x$. 
\end{exam}

Surprisingly, the next theorem says that we can not always expect Caffarelli's regularity for the optimal transport maps for the repulsive harmonic cost with finitely many marginals $N>2$ even when the support of the measures has convex interior (see corollary \ref{OTHARM:maincor} below). Before stating the theorem we will prove a lemma that characterizes many optimal plans:

\begin{lemma}\label{OTHarm:stuplemma}
Let $\{ \mu_i \}_{i=1}^N$ be probability measures on $\R^d$ and $h:\R^d\to\R$ be a strictly convex function and suppose $c:([0,1]^d)^N\to\R$ a cost function of the form $c(x_1,\dots,x_N) = h(x_1+\dots+x_N)$. Then, if there exists a plan $\gamma \in \Pi ( \mu_1, \ldots, \mu_N)$ concentrated on some hyperplane of the form $x_1+ \ldots x_N = k$, this plan is optimal for the multi marginal problem with cost $c$ and $ {\rm supp} ( \tilde \gamma ) \subset \{ x_1+ \ldots + x_N=k \} $ is a necessary and sufficient condition for $\tilde \gamma$ to be optimal. In this case we will say that $\gamma$ is a flat optimal plan and $\{ \mu_i \}_{i=1}^N$ is a flat $N$-tuple of measures.
\end{lemma}

\begin{proof} First of all we show that $k$ is fixed by the marginals $\mu_i$, and it is in fact the sum of the barycenters of these measures. Let $c_i = \int x \, d \mu_i$; then let us suppose that there exists $\gamma \in \Gamma ( \mu_1 , \ldots, \mu_N)$ that is concentrated on $\{x_1+ \ldots + x_N = k \}$. Then, using that the $i$-th marginal of $\gamma$ is $\mu_i$, we can compute
$$k = \int (x_1+ \ldots + x_N ) \, d \gamma = \sum_{i=1}^N \int x_i \, d \gamma = \sum_{i=1}^N \int x \, d \mu_i = \sum_{i=1}^N c_i.$$
In particular we notice also that for every admissible plan $\tilde \gamma$ we have $\int (x_1+ \ldots + x_N) \, d \tilde \gamma = k$ and so by Jensen inequality we have
$$\int h( x_1 + \ldots + x_N ) \, d \tilde \gamma \geq h \left( \int  (x_1+ \ldots + x_N) \, d \tilde \gamma \right ) = h (k) = \int h( x_1 + \ldots + x_N ) \, d \gamma.$$
This proves that $\gamma$ is an optimal plan. Thanks to the strict convexity of $h$, this shows also that if $\tilde \gamma$ is optimal then $\tilde \gamma$-a.e. we should have $x_1+ \ldots + x_N =k$.
\end{proof}

This reveals a very large class of minimizers in some cases, as we will see later. However not every $N$-tuple of measures is flat: it is clear that we can have marginals such that there is no plan with such property:

\begin{oss} Let $N=3$ and $\mu_i=\mu$ for every $i=1,2,3$ with $\mu=\nu_1 + \nu_2$ with $\nu_1=(-Id)_{\sharp} \nu_2$ and $\nu_1$ concentrated on $[2,3]$. Now it is clear that the barycenter of $\mu$ is $0$ but for every $3$ points $x_1$, $x_2$, $x_3$ in the support of $\mu$ we cannot have $x_1+x_2+x_3=0$: two of them have the same sign, let's say $x_1$ and $x_2$, but then we have $|x_1+x_2| \geq 4 > 3 \geq |x_3|$, which contradicts $x_1+x_2=-x_3$. So there is no an admissible plan concentrated on $H_0=\{ x_1+x_2+x_3=0 \}$; in fact we showed that for every $\gamma \in \Pi ( \mu) $ we have ${\rm supp} (\gamma) \cap H_0 = \emptyset$. \end{oss}

\begin{oss} In the case $N=2$ the condition for which there exists an admissible plan $\gamma$ concentrated on an hyperplane of the form $x_1 +x_2=k$ implies that the two measures $\mu$ and $\tilde \nu = (-Id )_{\sharp} \nu$ are equal up to translation. This condition is in fact very restrictive. However when we think at the DFT problem and in particular we assume that $\mu=\nu$ then the said condition amounts to have that $\mu$ is centrally symmetric about its barycenter and in the context of density of $2$ electrons around a nucleus this seems a fairly natural condition.
\end{oss}

\begin{teo}\label{OTHarm:mainthm}
Let $\mu_i = \mu = \mathcal{L}^d\big|_{[0,1]^d}, \forall i =1,\dots,N$ be the uniform measure on the $d$-dimensional cube $[0,1]^d \subset \R^d$, $h:\R^d\to\R$ a convex function and suppose $c:([0,1]^d)^N\to\R$ a cost function such that $c(x_1,\dots,x_N) = h( x_1+\dots+x_N )$. Then, there exits a transport map $T:[0,1]^d\to [0,1]^d$ such that $T^{N}(x) = x$ and 
$$
\ds \min_{\gamma \in \Pi_N(\mu)} \int c d\gamma = \min_{{\substack{T_{\sharp}\mu = \mu ,\\ T^{(N)}=I}}} \int c(x,T(x), \dots, T^{(N-1)}(x)) d \mu 
$$

Moreover, $T$ is not differentiable at any point and it is a fractal map, meaning that it is the unique fixed point of a ``self-similar'' linear transformation $\mathcal{F}$ acting on $L^{\infty}([0,1]^d;[0,1]^d)$.

\end{teo}

\begin{proof}
We express every $z \in [0,1]$ by its base-$N$ system, $z = \sum^{\infty}_{k=1}\frac{a_k}{N^k}$ with $a_k \in \lbrace 0,1\dots,N-1\rbrace$. Consider the map given by $S(z) = \sum^{\infty}_{k=1}\frac{\mathcal{S}(a_k)}{N^k}$, where $ \mathcal{S}$ is the permutation of $N$ symbols such that $\mathcal{S}(i) = i+1, \forall \lbrace i=0,1, \dots,N-2\rbrace$ and $\mathcal{S}(N-1) = 0$. A straightforward computation shows that
\begin{equation}\label{eqn:bary}
 z + \sum^N_{i=1}S^i(z)  = \dfrac{N}{2}
\end{equation}
Let $T:\R^d\to\R^d$ be a map defined by $T(x) = T(z_1,\dots,z_d) = (S(z_1),\dots, S(z_d))$ and denote by $T^{(j)}(x) = (S^{(j)}(z_1),\dots,S^{(j)}(z_d)), \ j=1,\dots,N-1$. We will first show that $S$ is a measure-preserving map. In fact, we can show that there exist functions $S_k:[0,1]\to[0,1]$ defined recursively by
$$
S_0(x) = x, \quad \text{ and } \quad 
S_{k+1}(x) =T_{k+1} ( S_k (x) )
$$
where $T_{k}$ acts only of the $k$-th digit: $T_{k}(x) = x - (N-1) \cdot N^{-k}$ if $x \in C_k$  and $T_{k} (x) =x+ N^{-k}$ if $x \in [0,1] \setminus C_k$. The sets $C_k$ are defined by 
$$C_{k}=\bigcup_{j=1}^{N^{k-1}} \left( \frac {j}{N^{k-1}}-\frac {1}{N^{k}} , \frac j{N^{k-1}} \right ]. $$
Moreover, it is easy to see that
\[
(T_k)_{\sharp}\mathcal{L}_{[0,1]} = \mathcal{L}_{[0,1]} \ \forall k \in \N, \quad \text{ and } \quad S_k \to S \ \textit{uniformly.} 
\]
Hence 
\[
\int f(x) S_{\sharp}\mathcal{L}^d = \int f(x)dx ,\quad \forall f \in C_0([0,1]).
\]
Now, it remains to show that $T$ is optimal. But this is true thanks to the fact that \eqref{eqn:bary} implies that the plan induced by $T$ satisfies the hypothesis of Lemma \ref{OTHarm:stuplemma}.

It is clear that we can reduce to prove the non-differentiability and the fractal properties only for the map $S$. The non-differentiability comes from the fact that for each $z \in [0,1]$, we have that the base-$N$ representation $z = \sum^{\infty}_{k=1}\frac{a_k}{N^k}$ is such that $a_k$ can't be definitively $N-1$. Now it is sufficient to choose those $k_j$ such that $a_{k_j} \neq N-1$ and consider the numbers $z_j=z+(N-1-a_{k_j})/N^{k_j}$ and $z'_j=z \pm 1/N^{k_j}$ (depending on whether $a_{k_j}=0$ or not) which have the same digits as $z$ apart from the $j_k$-th digit. Then it is straightforward to see that $S(z_j)-S(z)= - (a_{k_j} + 1) / N^{k_j}$ while $S(z'_j)-S(z)=z'_j-z$; in particular we have $\frac{S(z'_j) - S(z)}{z'_j -z} = 1$ while $ \frac{S(z_j) - S(z)}{z_j -z} \leq - \frac 1N$ and so, letting $j \to \infty$, we get that $S$ is not differentiable at the point $z$.

As for the fractal property: we consider the transformation
$$\mathcal{F}(g) (x) = \begin{cases} \frac 1N g(Nx - i) + \frac { i+1}N \qquad &\text{ for } \frac iN \leq x <  \frac {i+1}N  \text{, with } i=0, \ldots, N-2\\
\frac 1N g(Nx - N-1)  & \text{ for } \frac {N-1}N \leq x < 1. \end{cases}$$
In order to see what the construction is doing we imagine to divide $[0,1]^2$ in a grid $N \times N$ of squares and then we are putting scaled copies (by a factor $N$) of the original function in the above diagonal squares and in the rightmost bottom one. 
Now it is easy to see that $S$ is the unique\footnote{We notice that $\mathcal{F}$ is a contraction: $\| \mathcal{F} (g) - \mathcal{F} (g')\|_{\infty}\leq \frac 1N \| g- g'\|_{\infty}$} fixed point of $\mathcal{F}$, and this gives the property of self-similarity.
\end{proof}

\begin{oss} Another proof of Theorem \ref{OTHarm:mainthm} can be done noticing that the transformation $\mathcal{F}$ (the one defined in the proof) acts as a $1/N$-contraction in the space of measure preserving $L^{\infty}$-bijections from $[0,1)$ to $[0,1)$.
\end{oss}

\begin{cor}\label{OTHARM:maincor}
Let $ \mu = \mathcal{L}^d\big|_{[0,1]^d}$ be the uniform measure on the $d$-dimensional cube $[0,1]^d$ in $\R^d$ and suppose $c:([0,1]^d)^N\to\R$ the $N$-dimensional repulsive harmonic cost
\[
c(x_1,\dots,x_N) = -\sum^N_{i=1}\sum^N_{j=i+1}\vert x_j-x_i\vert^2, \quad \quad (x_1,\dots, x_N) \in ([0,1]^d)^N
\]
Then, there exits an optimal cyclical transport map $T:[0,1]^d\to [0,1]^d$: in particular 
$$
\ds \min_{\gamma \in \Pi_N(\mu)}\int c d\gamma = \min_{{\substack{T_{\sharp}\mu = \mu ,\\ T^{(N)}=I}}} \int c(x,T(x), T^{(2)}(x), \dots, T^{(N-1)}(x)) d \mu.
$$

Moreover, $T$ is not differentiable at any point. 
\end{cor}

\begin{proof}
As we already observed, the problem with the cost $c$ is equivalent to the problem with the cost $|x_1+ \ldots + x_N|^2$ and the result follows from the Theorem \ref{OTHarm:mainthm}.
\end{proof}

We remark the construction of $T$ in the proof of the Theorem \ref{OTHarm:mainthm} also works if $N=2$ and $T$ is exactly the optimal transport map described in the example \eqref{OTHarm:Exam2Marginals}.  

The unexpected aspect of Corollary \ref{OTHARM:maincor} is the existence - for $N\geq 3$ - of an optimal transport map which is not differentiable almost everywhere. It turns out that this optimal map $T$ could be not unique if $d>1$ and $N\geq 3$ and, in that case, we can construct explicitly a regular optimal map. This kind of richness appears already in the case $d=1$ with the Coulomb cost; however in that case we have uniqueness if we restrict ourselves to the symmetric optimal plan. Here this is not the case as it is easy to see that modifying the action $\mathcal{S}$ on the digits (the important thing is that when we see $\mathcal{S}$ as a permutation, it is a cycle), we obtain another map, and the symmetrized plan is not equal to the one generated by the map described before.

In the following, we are going to present some concrete examples where we have other explicit solutions for such kind of optimal maps for $(\mathcal{MK}_{weak})$ in \eqref{OTHarm:pb:MKweak}. For these examples the goal is to show that there can be smooth optimal maps but they don't necessarily satisfy the ``group rule'', that is, we can find maps $T_2, \ldots T_N$ such that $(Id, T_1, T_2, \ldots , T_{N-1} )_{\sharp} \mu_1$ is an optimal plan but $T_1 \circ T_1 \neq T_i$ for any index $i \in \lbrace 1,\dots,N-1\rbrace$.

\begin{exam}[3 particles, asymmetric]\label{OTHarm:ex:3p}
Consider the case when three particles are distributed in $\R^3$ as Gaussians, $\mu_1 = \mu_2 = \frac{1}{(2\pi)^{3/2}}\exp(-\frac{1}{2}(x_1^2+x_2^2+x_3^2))$ and $\mu_3 =  \frac{1}{2\pi^{3/2}}\exp(-(x_1^2+x_2^2+x_3^2))$. In this case, we can verify that the  couple $(T,S)$ of  maps $T,S:\R^d\to\R^d$, $T(x) = x$ a.e. and $S(x) = -2x$ a.e. is admissible and it is optimal since $ x+ T(x) +S(x) = 0$.
\end{exam}

\begin{exam}[2N particles on $S^1$, 
] \label{OTHarm:ex:breathing2marginals} Suppose $\mu_1,\dots, \mu_{2N}$ uniform probability measures on the circle $S^1$. The rotation map $R_{\theta}:\R^2\to\R^2$ with angle $\theta = \pi/N$ is an optimal transport map. Also, the maps $R_{k\theta}, \ k=2,\dots,N$, are optimal transport maps for the repulsive harmonic cost.
\end{exam}

\begin{exam}[2N particles, 
\text{breathing} map]\label{OTHarm:ex:evenmarginals} Suppose $\mu_1, \dots, \mu_{2N}$ uniform probability measures on $S^2 \subset \R^3$. Consider the vector $v$, $A:\R^3\to\R^3$ the antipodal map $A(x) = -x$ and $R^{v}_{\theta}:\R^3\to\R^3$ the rotation of angle $\theta = \pi/N$ and direction $v$. Then, $T = A\circ R^v_{\theta}$ is a cyclic optimal transport map. 

The optimal solution $\gamma = (x,T,T^{(2)},\dots,T^{(N-1)})_{\sharp}\mu_1$ is called \textit{``breathing'' solution} \cite{SeiGorSav}: the coupling $\gamma$ represent the configuration where the $2N$ electrons are always at the same distance from the center, opposite to each other in the equilibrium configuration.  Notice that, the map $G = T^2$ is also an cyclical optimal transportation map. Moreover, $T$ is a $C^{\infty}$ function and it is a gradient of the convex-concave function $\phi(x_1,x_2,\dots,x_{2N-1},x_{2N}) = \frac{1}{2}(x_1^2+x_3^2+\dots+x_{2N-1}^2) -\frac{1}{2}(x_2^2+x_4^2+\dots+x_{2N}^2)$. 
\end{exam}

\begin{exam}[2N particles in $\R^d$, symmetric $\rho$]\label{OTHarm:ex:2Np}
In this case, on can consider the maps $T(x) = x$ a.e. and $S(x) = -x$ a.e. which are such that $$c(T(x),S(x),\dots,T(x),S(x)) = 0.$$
We notice that in this case we have $S^{(2)}(x)=x$ and so in particular this solution has the cyclic structure $$(T(x),S(x),\dots,T(x),S(x))=(x,S(x),S^{(2)}(x), \ldots, S^{(2N-1)}).$$ 
\end{exam}

\begin{exam}[Optimal Maps which doesn't satisfy a group law]\label{OTHarm:ex:twomapsplan} Let us consider the case $d=1$ and $N=3$ where the measures  are $\mu_1=\mu_2=\mu_3=\frac 12 \mathcal{L}|_{[-1,1]}$.

Then we define the maps 
 $$ T(x) =  \begin{cases} x+1 \quad & \text{ if } x\leq 0 \\ x-1  & \text{ if }x >0, \end{cases} \qquad S(x) =  \begin{cases} -1-2x \quad & \text{ if } x\leq 0 \\ 1-2x  & \text{ if }x >0. \end{cases}$$
 We have that $T_{\sharp} \mu = \mu $ and $S_{\sharp} \mu = \mu$, and moreover $x+T(x)+S(x)=0$; in particular $(Id, S, T)_{\sharp} \mu$ is a flat optimal plan and so the thesis.
\end{exam}

\begin{exam}[Optimal Transport diffuse plan]\label{OTHarm:ex:trueplan} Let us consider the same problem as in Example \ref{OTHarm:ex:twomapsplan}. Now we consider a general symmetric plan $\gamma =\frac 12 \mathcal{H}^{2}|_{ H}  f(\max \{ |x| , | y| , |z| \}) $, where $H$ is defined as $H= \{ x+y+z=0 \} \cap \{ |x| \leq 1, |y| \leq 1, |z| \leq 1 \}$ and $f$ is a function to be chosen later. This is a symmetric flat optimal plan; now we compute the marginals. Since it is clear that $\gamma$ is invariant under $x \mapsto -x$, it is sufficient to consider the marginal on the set $x >0$. But then we can make the computation

\begin{align*}
\int_{x \geq 0} \phi(x) \, d \gamma(x,y,z) & = \iint_{\substack{|x|, |y|, |x+y| \leq 1,  \\ x \geq 0}} \sqrt{3}\phi(x) f(\max\{ |x|,|y|,|x+y| \} ) \, dx dy  \\
&=  \int_0^1 \int_0^{1-x} \sqrt{3}\phi(x) f(x+y) \, dy \, dx + \sqrt{3} \int_0^1 \int_{-x}^0 \phi (x) f(x) \, dy \, dx \\ & \qquad + \sqrt{3} \int_0^1 \int_{-1}^{-x} \phi(x) f(|y|) \, dy \, dx \\
& = \int_0^1 \sqrt{3}\phi(x) \left( xf(x) + 2\int_x^1 f(t) \, dt\right) \, dx.
\end{align*}
In particular the choice $f(x)=\frac {\sqrt{3}}6 x$ gives the marginals equal to $\frac 12 \mathcal{L}|_{[-1,1]}$. We notice also that any even density $\rho=h(|x|)$ for some decreasing function $h:[0,1] \to [0, \infty)$ can be represented in this way: in fact it is sufficient to choose $f(x) =\frac 1{\sqrt{3}} \Bigl( \frac{h(x)}x - 2x \int_x^1 \frac {h(t)}{t^3} \, d t \Bigr)$.
\end{exam}

\begin{exam}[A Counterexample on the uniqueness $N>3$\label{Passconterexample}] As mentioned in theorem \ref{th:sym:Pass} and in \cite{FMPCK}, it was already understood by Pass that in these high dimensional cases, the solution of repulsive costs may also be non unique, as opposite to the two marginals case. On a higher dimensional surface there can be enough wiggle room to construct more than one measure with common marginals, as shown in the examples \ref{OTHarm:ex:3p}--\ref{OTHarm:ex:trueplan}. In most of the cases, the non-uniqueness seems to be given by the symmetries of the problem, but in Example \ref{OTHarm:ex:trueplan} and Corollary \ref{OTHARM:maincor} this is not the case, as we exploit the fact that the dimension of the set $c(x+y+z) - \phi(x) -\phi(y) -\phi(z)=0$, is greater than the minimal one.
\end{exam}

Finally, the last proposition of this section states that when $d=1$ and odd $N$ we have no hope in general to find piecewise regular cyclic optimal transport maps.

\begin{prop}
Let $\mu = \mathcal{L}|_{[0,1]}$ and $N\geq 3$ be an odd number. Then, the infimum 
\[
\inf \bigg\lbrace \int c(x,T(x), T^{(2)}(x),\dots,T^{(N-1)}) d \mu \ : \ \begin{array}{ll}
T_{\sharp}\mu = \mu ,\\
\ T^N=I \end{array} \bigg\rbrace.
\]
is not attained by a map $T$ which is differentiable almost everywhere. 
\end{prop}

\begin{proof} 
First of all we notice that if $T$ is differentiable almost everywhere then also $T^{(2)}$ has the same property, thanks to the fact that $T_{\sharp} \mu =\mu$. In particular, since the Lusin property holds true for $T^{(i)}$ for every $i=1, \ldots, N$ and $T_{\sharp}\mu=\mu$, the change of variable formula holds and in particular we have that $ (T^{(i)})' (x) = \pm 1$ for almost every $x$ (notice also that $T$ is bijective almost everywhere since $T^{(N)}(x)=x$). 

Since $\mu$ is $N$-flat we have that the condition on $T$ in order to be an optimal cyclical map is $x + T(x) + T^{(2)}(x) + \ldots + T^{(N-1)}(x) = N/2$; now we can differentiate this identity and so we will get 
$$ 1+ T'(x) +(T^{(2)})'(x) + \ldots + (T^{(N-1)})'(x) = 0 \qquad \text{ for a.e. }x.$$
But this is absurd since on the left hand side we have an odd number of $\pm 1$ and their sum will be an odd number.
\end{proof}

In conclusion, also in the case of the repulsive harmonic cost, the picture is far from being clear: an interesting structure appears when $\lbrace {\mu_i}\rbrace_{i=1}^N$ is a flat $N$-tuple of measures but we still can't characterize this property. Moreover, in the flat case in which $\mu_i=\mu$, for example when $\mu= \frac 12 \mathcal{L}^1|_{[-1,1]}$, we have both diffuse optimal plan and a cyclical optimal map.

An interesting open problem is whether for any $N$-flat measure $\mu$, say absolutely continuous with respect to the Lebesgue measure, we have a cyclical optimal map and a diffuse plan.

%% file: multimarginalOTproblemdeterminant.tex

\section{Multi-marginal OT for the Determinant}\label{sec:Det}

$\quad$ We are going to give a short overview of the main results in \cite{CarNa}, where Carlier and Nazaret consider the following optimal transport problems for the determinant:
\begin{equation}\label{Det:pb:MKDet}
(\MK_{Det}) \quad \sup_{\gamma \in \Pi((\R^d)^d,\mu_1,\dots,\mu_d)}
\ds \int_{(\R^d)^d}\ds \operatorname{det}(x_1,\dots,x_d) d\gamma(x_1,\dots,x_d)
\end{equation}
and
\begin{equation}\label{Det:pb:MKDetm}
(\MK_{\vert Det\vert}) \quad \sup_{\gamma \in \Pi((\R^d)^d,\mu_1,\dots,\mu_d)}
\ds \int_{(\R^d)^d}\ds \vert \operatorname{det}(x_1,\dots,x_d)\vert d\gamma(x_1,\dots,x_d), 
\end{equation}
where $\mu_1,\dots,\mu_d$ are the uniform probability measures in $\R^d$. In addition, in order to guarantee existence of a solution, we assume that there exist $p_1,\dots,p_d \in [1,\infty[$ such that
\[
\sum^d_{i=1}\dfrac{1}{p_i} = 1, \quad \text{and} \quad \sum^d_{i=1}\int_{\R^d} \dfrac{\vert x_i\vert^{p_i}}{p_i}d\mu(x_{i}) < +\infty.
\]

Notice that for this particular cost the problem \eqref{pb:MKN} makes sense only when $N=d$. In the following, we will focus on problem \eqref{Det:pb:MKDet} and exhibit explicit minimizers $\gamma$ in the radial case. Clearly, the difference between \eqref{Det:pb:MKDet} and \eqref{Det:pb:MKDetm} is that the second one admits positively and negatively oriented basis of vectors, while the first one ``chooses" only the positive ones. Moreover, if we assume that among  marginals $\mu_1,\dots,\mu_d$ there exist two symmetric probability measures $\mu_i,\mu_j, i\neq j, \ i.e. \ \mu_i = (\operatorname{-Id})_{\sharp}\mu_i$ and $\mu_j = (\operatorname{-Id})_{\sharp}\mu_j$, then any solution $\gamma$ of \eqref{Det:pb:MKDet} satisfies $\det(x_1,\dots,x_d) \geq 0  \ \gamma$-almost everywhere and so solves also \eqref{Det:pb:MKDetm} (Proposition 6, \cite{CarNa}).

Similarly to the Gangbo-\'Swi\c{e}ch cost \cite{GaSw}, the Monge-Kantorovich problem for the determinant \eqref{Det:pb:MKDet} can be seen as a natural extension of classical optimal transport problem with two marginals and so, it is equivalent to the $2$-marginals repulsive harmonic cost \eqref{OTHarm:pb:MKweak}. 

Indeed, we can write in the two marginals case, $\det(x_1,x_2) = \langle x_1,Rx_2 \rangle$, where $R:\R^2\to\R^2$ is the rotation of angle $-\pi/2$. Hence, since $\mu_1$ and $\mu_2$ have finite second moments, up to a change of variable $\tilde{x}_2 = Rx_2$, the problem $(\MK_{Det})$ in \eqref{Det:pb:MKDet} is equivalent to the classical Brenier's optimal transportation problem: 

\[
\begin{array}{ll}
\ds\ds\underset{\gamma \in \Pi(\mu_1,\mu_2)}{\operatorname{argmax}} \int_{\R^2}\det(x_1,x_2)d\gamma(x_1,x_2)
&=\ds\underset{\gamma \in \Pi(\mu_1,\mu_2)}{\operatorname{argmax}} \int_{\R^2} \langle x_1,\tilde{x}_2 \rangle d\gamma(x_1,x_2) \medskip \\ 
&=\ds\underset{\tilde{\gamma} \in \Pi(\mu_1,\tilde{\mu}_2)}{\operatorname{argmax}} \int_{\R^2} \langle x_1,x_2 \rangle d\tilde{\gamma} \medskip\\ 
&=\ds\underset{\gamma \in \Pi(\mu_1,\tilde{\mu}_2)}{\operatorname{argmin}} \int_{\R^2}\frac{\vert x_1-x_2 \vert^2}{2}d\gamma(x_1,x_2)- C \\  \end{array}
 \]
 where $C= 1/2(\int |x_1|^2\, d \mu_1+ \int |x_2|^2 \, d \mu_2) $ and $\tilde{\mu}_2=R_{\sharp} \mu_2$. \medskip

In the sequel, we are going to construct maximizers for \eqref{Det:pb:MKDet}, thanks to some properties of the Kantorovich potentials of the dual problem associated to \eqref{Det:pb:MKDet} (see theorem \ref{Det:carcgamma} bellow), 
\begin{equation}\label{pb:KNDet}
(\KN^{Det}) \quad \inf \Bigg\lbrace \int_{\R^d} \sum^d_{i=1}u_i(x_i)d\mu_i(x_i)
: \operatorname{det}(x_1,\dots,x_d) \leq \sum^N_{i=1}u_i(x_i) \Bigg\rbrace.
\end{equation}

In \cite{CarNa}, the authors provide a useful characterization of optimal transport plans through the potentials $u_i$, given by theorem \ref{Det:carcgamma} . In addition, by means of a standard convexification trick we obtain regularity results on the Kantorovich potentials.

\begin{teo}\label{Det:carcgamma}
A coupling $\gamma \in \Pi((\R^d)^d, \mu_1,\dots,\mu_d)$ is optimal in \eqref{Det:pb:MKDet} if and only if there exists
lower semi-continuous convex functions $u_i:\R^d\to \R\cup\lbrace \infty\rbrace$ such that for all $i \in \lbrace 1, \dots, d \rbrace$,

\[
\sum^d_{i=1} u_i(x_i) \leq \sum^d_{i=1} u^*_i((-1)^{i+1}\bigwedge_{i\neq j} x_j), \quad \textit{ on } \quad (\R^d)^d; 
\]
\[
\sum^d_{i=1} u_i(x_i) = \sum^d_{i=1} u^*_i((-1)^{i+1}\bigwedge_{i\neq j} x_j), \quad \gamma-\textit{almost everywhere};
\]
\[
(-1)^{i+1}\bigwedge_{i\neq j} x_j \in \partial u_i(x_i), \quad \gamma-\textit{almost everywhere}.
\]
where $\bigwedge^d_{i=1} x_j$ denotes the wedge product and, for every $i$, $u^*_i$ is the convex dual of the Kantorovich potential $u_i$.
\end{teo}

Now, the main idea is to use the geometrical constraints on the Kantorovich potentials $u_i$ \eqref{pb:KNDet}, given by the theorem \eqref{Det:carcgamma}, in order to construct an explicit solution. 

We illustrate the theorem \eqref{Det:carcgamma} and explain how to construct a particular optimal $\gamma$ for \eqref{Det:pb:MKDet} by an example in the three marginals case. Let $\mu_i = \rho_i\mathcal{L}^3, \ i=1,2,3$ radially symmetric  probability measures on the $3$-dimensional ball $B$. 

In this particular situation, the optimizers of \eqref{Det:pb:MKDet} and \eqref{Det:pb:MKDetm} have a natural geometric interpretation: what is the best way to place three random vectors $x,y,z$, distributed by probability measures $\mu_1,\mu_2,\mu_3$ on the sphere, such that the simplex generated by those three vectors $(x,y,z)$ has maximum average volume? \medskip

Suppose $\gamma \in \Pi(B, \mu_1,\mu_2,\mu_3)$ optimal in \eqref{Det:pb:MKDet} when $d=3$. From optimality of $\gamma$, we have
\[
u_1(x) + u_2(y) + u_3(z) = det(x,y,z), \quad \gamma-\textit{almost everywhere}.
\]
Applying the theorem \eqref{Det:carcgamma}, we get
\[
\Bigg\lbrace \begin{array}{c} u_2(y) + u_3(z) = u^*_1(y\wedge z)  \\
u_1(x) + u_3(z) = u^*_2(-x\wedge z)  \\
u_1(x) + u_2(y) = u^*_3(x\wedge y)   \end{array}, \quad \gamma-\textit{almost everywhere},
\]
and,
\begin{equation}
\label{sys:perp}
\Bigg\lbrace \begin{array}{c} \nabla u_1(x) = y\wedge z  \\
\nabla u_2(y) = -x\wedge z  \\
\nabla u_3(z) = x\wedge y   \end{array}, \quad \gamma-\textit{almost everywhere}. \medskip
\end{equation}

It follows from \eqref{sys:perp}, given a vector $x$ on in the ball, the conditional probability of $y$ given $x$ is supported in a \textit{``meridian"} $M(x)$
\[
M(x) = \lbrace y \in S^2 :  \langle \nabla u_1 (x),y \rangle = 0 \rbrace,
\]
where $S^2$ is the $2-$sphere. Finally, assuming that $\langle x, \nabla u_1(x) \rangle \neq 0$, the conditional probability of $z$ given the pair $(x,y)$ is simply given by a delta function on $z$
\[
z = \dfrac{\nabla u_1(x)\wedge \nabla u_2(y)}{\langle x, \nabla u_1(x) \rangle}; \medskip
\]
In particular, we have
\[
\langle x,\nabla u_1(x) \rangle + \langle y,\nabla u_2(y) \rangle + \langle z,\nabla u_3(z) \rangle = \det(x,y,z) = \det(\nabla u_1(x),\nabla u_2(y),\nabla u_3(z)). \medskip
\]

\begin{exam}[An explicit solution in the ball  $B \subset \R^3$] Suppose $\mu_i = \mathcal{L}^3_{B}, \ i=1,2,3$, the $3$-dimensional Lebesgue measure in the ball $B\subset \R^3$. The following coupling $\gamma^*$
\be\label{Det:examplegamma}
\int_{B^3} fd\gamma^* = \dfrac{1}{\mathcal{L}^3(B)}\int_B\bigg(\int_{M(x)}f(x,\vert x\vert y, x\wedge y)\frac{d\mathcal{H}^1(y)}{2\pi}\bigg)dx, \ \forall f \in C(B^3,\R). \bigskip
\ee
is an optimizer for \eqref{Det:pb:MKDet} with $d=3$. Indeed, from we can show explicit potentials $u_1^*(x) = u_2^*(x) = u_3^*(x) = \vert x\vert^3/3$; clearly we have 
\[
\det(x,y,z) \leq \vert x\vert\vert y \vert\vert z\vert \leq \vert x\vert^3/3 + \vert y\vert^3/3 + \vert z\vert^3/3, \quad \forall \ (x,y,z) \in B,
\]
with equality when $|x|=|y|=|z|$ and $x,y,z$ are orthogonal. Since $\gamma^*$ is concentrated on this kind of triples of vector we have the optimality. Finally, by a suitable change of variable, it is easy to see that $\gamma^* \in \Pi_3(\mathcal{L}^3_B)$.
\end{exam}

\emph{Some comments on the radially symmetric $d$-marginals case:} In \cite{CarNa}, for $d$ radially symmetric probability measures, the authors exhibit explicit optimal couplings $\gamma^*$. In their proof, two aspects were crucial: the first one is remark that if $\lbrace \mu_i\rbrace^d_{i=1}$ are radially symmetric measures in $\R^d$, then the optimal Kantorovich potentials $u_i(x_i) = u_i(\vert x_i \vert)$ are also radially symmetric; in particular the system \eqref{sys:perp} for general $d$ implies that the support of $\gamma^*$ is contained in the set of an orthogonal basis.
The second observation is to notice that in the support of $\gamma^*$ we have $H_i(|x_1|)=|x_i|$, where $H_i$ is the unique monotone increasing map such that $(H^*_i)_{\sharp} \mu_1=\mu_i$, where $H^*_i(x)=\frac x{|x|} H_i(|x|)$. This is done analyzing the correspondent radial problem (with cost $c(r_1, \ldots, r_d) =  r_1 \cdots r_d$), using the optimality condition $ \phi_i'(r_i) = \partial_i c$ and the fact that in this case $r_i \phi_i'(r_i) = r_1 \phi_1'(r_1) =c \geq 0$; then, exploiting the convexity of $\phi_i$ we get $r_i=H_i (r_1)$ for some increasing function $H_i$, that is uniquely determined. \medskip
%

\emph{Existence of Monge-type solutions:} In the $3$-marginals case in the unit ball, by construction of the the coupling $\gamma^*$ in \eqref{Det:examplegamma} or, more generally, the optimal coupling in the $d$-marginal case (see theorem 4 in \cite{CarNa}), we can see that their support are not concentrated in the graph of cyclic maps $T,T^2,\dots,T^{d-1}$ or simply on the graph of maps $T_1,\dots, T_{d-1}$ as we could expect from corollary \eqref{cor:CoDMa}. In other words, $\gamma^*$ in \eqref{Det:examplegamma} is not Monge-type solution. 

The existence of Monge type solutions for the determinant cost is still an open problem for odd number of marginals. From the geometric conditions we discussed above, in the case in which $\mu_i=\mu$ a radial measure, if Monge solutions exists then, for every $x \in \R^d$, $(x,T_1(x),\dots,T_{d-1}(x))$ should be an orthogonal basis, and $|T_i(x)| = \vert x\vert \ , \ i=1,\dots,d-1$.

For the interesting even dimensional case, we can observe a similar phenomena remarked in the repulsive harmonic costs, concerning the existence of trivial even dimensional solutions for the Monge problem in \eqref{Det:pb:MKDet}. \textbf{We expect the existence of non-regular optimal transport map also to this case.}

\begin{exam}[Carlier \& Nazaret, \cite{CarNa}]
\textit{The even dimensional phenomena:} As in the repulsive harmonic cost, it is easy to construct Monge minimizers for the determinant cost for even number of marginals $\geq 4$. For instance, suppose $c(x_1,x_2,x_3,x_4) = \det(x_1,x_2,x_3,x_4)$, define transport maps $T_1,T_2,T_3:B\to\R^4$ by, for $x = (x_1,x_2,x_3,x_4) \in B$
$$
T_1(x) = \left( \begin{array}{c}
-x_2  \\
x_1 \\
-x_4 \\
x_3  \end{array} \right), \quad \quad T_2(x) = \left( \begin{array}{c}
-x_3  \\
x_4 \\
x_1 \\
-x_2  \end{array} \right), \quad \quad T_3(x) = \left( \begin{array}{c}
-x_4  \\
-x_3 \\
x_2 \\
x_1  \end{array} \right).
$$
We can see that $\gamma_T = (Id,T_1,T_2,T_3)_{\sharp}\gamma$ is a Monge-type optimal transport plan for \eqref{Det:pb:MKDet} and \eqref{Det:pb:MKDetm}.
\end{exam}

%% file: Numerics.tex

\section{Numerics}
\label{numerics}
Numerics for the multi-marginal problems have so far not been extensively developed.
Discretizing the multi-marginal problem leads to a linear program where the number of constraints grows exponentially in the number of
marginals.
In \cite{CarObeOud} Carlier, Oberman and Oudet studied the matching of teams problem  and they were able to reformulate the problem
as a linear program whose number of constraints grows only linearly in the number of marginals.
More rencently a numerical method based on an entropic regularization has been developed and it has been applied to various optimal transport problem 
in \cite{Ben,BenCaNen,Cut}.
Let us mention that the Coulomb cost has been treated numerically in various paper as
in \cite{Malet2012,GorSeiVig}, where the solutions are based on the analytical form of transport maps given in \cite{GorSeiVig},
in \cite{Mendl2013}, where parameterized
functional forms of the Kantorovich potential are used, in \cite{ChenFriMendl,ChenFri}, where a linear programming approach has been developed to solve the 1-dimensional problem, and, as already mentioned, in \cite{BenCaNen}.
In this section we focus on the regularized method proposed in \cite{Ben,BenCaNen,Cut} and we, finally, present some numerical experiments
for the costs studied above.
\subsection{The regularized problem and the Iterative Proportional Fitting Procedure}
\label{subIPFP}
Let us consider the problem  

\be\label{pb:MKN2}
(\MK) \quad \inf_{ \gamma \in \Pi(\R^{dN},\mu_1,\dots,\mu_N)}
\int_{\R^{dN}}c(x_1,\dots,x_N)\gamma(x_1,\dots,x_N)dx_1,\dots,dx_N,
\ee

where $N$ is the number of marginals $\mu_i$ which are probability distributions over $\mathbb{R}^d$, $c(x_{1},...,x_{N})$ is the cost function, $\gamma$ the coupling,  is the probability 
distribution over   $\mathbb{R}^{dN}$ 
\begin{oss}
From now on the marginals $\mu_i$ and the coupling $\gamma$ are densities and when the optimal coupling $\gamma$ is induced by maps $T_i$, we write it 
as $\gamma=\mu_1(x_1)\delta(x_2-T_2(x_1))\cdots\delta(x_N-T_N(x_1))$.
\end{oss}

In order to discretize (\ref{pb:MKN2}), we use a discretisation with $M_d$ points 
of the support of the $k$th marginal as $\{x_{j_k}\}_{j_k=1,\cdots,M_d}$. If the densities  
 $\mu_k$ are  approximated  by $\mu_k\approx\sum_{j_k} \mu_{j_k} \delta_{x_{j_k}}$, we get
\be
\label{pb:MKdiscrete}
 \min_{\gamma\in\mathcal{\Pi_h}}\sum_{j_1,\cdots j_{N}} c_{j_1,\cdots,j_N}\gamma_{j_1,\cdots,j_N},
\ee
where $\Pi_k$ is the discretization of $\Pi$,    $c_{j_1,\cdots,j_N} = c(x_{j_{1}},\cdots,x_{j_{N}})$  and   the coupling support for each coordinate is restricted to the points  
$\{x_{j_k}\}_{j_k=1,\cdots,M_d}$ thus becoming a $(M_d)^N$ matrix again denoted $\gamma$ with elements 
 $\gamma_{j_1,\cdots,j_N}$. The marginal constraints  $\mathcal{C}_{i}$ (such that $\Pi_k=\bigcap_{k=1}^{N}\mathcal{C}_{k}$) becomes 
  \begin{equation}
\label{constraint}
 \mathcal{C}_{k}:=\Big\lbrace \gamma \in\mathbb{R}_+^{  (M_d)^N } : \quad\sum_{j_{1},...,j_{k-1},j_{k+1},...,j_{N}}\gamma_{j_{1},...,j_{N}}=\mu_{j_k} ,  \,  \forall j_k = 1,\cdots,M_d  \Big\rbrace.
\end{equation}
As in the continous framework the problem (\ref{pb:MKN2}) admits a dual formulation
\begin{equation}
 \label{DualDiscrete}
 \begin{split}
  \max_{u_{j_k}}&\sum_{k=1}^{N}\sum_{j=1}^M u_{j_k}\mu_{j_k}\\
  s.t. & \sum_{k=1}^{N}u_{j_k} \leq c_{j_{1}\cdots j_{N}}\quad\forall \ j_{k}=1,\cdots,M_d,
 \end{split}
 \end{equation}
where $u_{j_k}=u_k(x_{j_k})$ is the $k$th Kantorovich potential.
One can notice that  the primal  (\ref{pb:MKN2})  has $(M_d)^N$ unknown and $M_d\times N$ linear constraints and
the dual problem (\ref{DualDiscrete}) has $M_d\times N$ unknown, but $(M_d)^N$ constraints. 
This actually makes the problems computationally unsolvable with standard linear programming methods even for small cases.
\begin{oss}
 We underline that in many applications we have presented (as in DFT) the marginals $\mu_{j_k}$ are equal ($\mu_{j_k}=\mu_j$, $\forall \ k\in\{1,\cdots,N\}$). Thus the dual problem
 can be re-written in a more convenient way (but still computationally unfeasible for many marginals) 
 \begin{equation}
 \label{DualDiscreteBis}
 \begin{split}
  \max_{u_{j}}&\sum_{j=1}^M Nu_{j}\mu_j\\
  s.t. & \sum_{k=1}^{N}u_{j_k} \leq c_{j_{1}\cdots j_{N}}\quad\forall j_{k}=1,\cdots,M_d,
 \end{split}
 \end{equation}
where $u_{j}=u_{j_k}=u(x_{j_k})$. Now the dual problem has $M_d$ unknown, but $(M_d)^N$ linear constraints.
\end{oss}

A different approach consists in computing the problem (\ref{pb:MKN2}) regularized by the entropy of the joint coupling. 
This regularization dates to E. Schr\"{o}dinger \cite{Schrodinger31} and, as mentioned above, it has been recently introduced in many applications
involving optimal transport \cite{Ben, BenCaNen,Cut, GalichonEconomics}.
Thus, we consider the following discrete regularized problem
\begin{equation}
 \min_{\gamma\in\mathcal{C}}\sum_{j_1,\cdots j_{N}}c_{j_1,\cdots,j_N}\gamma_{j_1,\cdots,j_N}+\epsilon E(\gamma)
\end{equation}
where  $E(\gamma)$ is defined as follows
\begin{equation}
 E(\gamma)=\begin{cases} \sum_{j_1,\cdots j_{N}}\gamma_{j_1,\cdots,j_N}\log(\gamma_{j_1,\cdots,j_N})\mbox{ if }  \gamma\geq 0 \\ +\infty \mbox{ otherwise},\end{cases}
\end{equation}
 and $\mathcal{C}$ is the intersection of the set associated to the marginal constraints (we remark that the entropy is a penalization of the non-negative contraint on $\gamma$).
After elementary computations, we can re-write the problem as
\begin{equation}
 \label{eq22}
 \min_{\gamma\in \mathcal{C}}KL(\gamma\lvert\bar{\gamma})
\end{equation}
where $KL(\gamma\lvert\bar{\gamma})=\sum_{i_{1},...,i_{N}}\gamma_{i_{1},...,i_{N}}\log(\dfrac{\gamma_{i_{1},...,i_{N}}}{\bar{\gamma}_{i_{1},...,i{N}}})$ is the
Kullback-Leibler distance and 
\begin{equation} 
\label{barg} 
\bar{\gamma}_{i_{1},...,i_{N}}=e^{-\dfrac{c_{j_1,\cdots,j_N}}{\epsilon}}.
\end{equation} 

As  explained in section \ref{sec:DFTOT}, when the transport plan $\gamma$ is  concentrated  on the graph of a transport map which solves the 
Monge problem,  after discretisation of the densities, this property is lost along but we still expect the matrix $\gamma$ to be  sparse.
The entropic regularization spreads the  support and this helps to stabilize the computation as
it defines a strongly convex program with a unique solution $\gamma^{\epsilon}$.
Moreover the solution $\gamma^{\epsilon}$ can be obtained through elementary operations. 
The regularized solutions $\gamma^{\epsilon}$ then converge to $\gamma^\star$  (see figure \ref{figure:regularization}), the solution of (\ref{pb:MKN2}) with minimal entropy, 
as $\epsilon\rightarrow 0$ (see \cite{CominettiAsympt} for a detailed asymptotic analysis and the proof of exponential convergence). 

\begin{figure}[htbp]

\begin{tabular}{@{}c@{\hspace{1mm}}c@{\hspace{1mm}}c@{}}

\centering
\includegraphics[ scale=0.12]{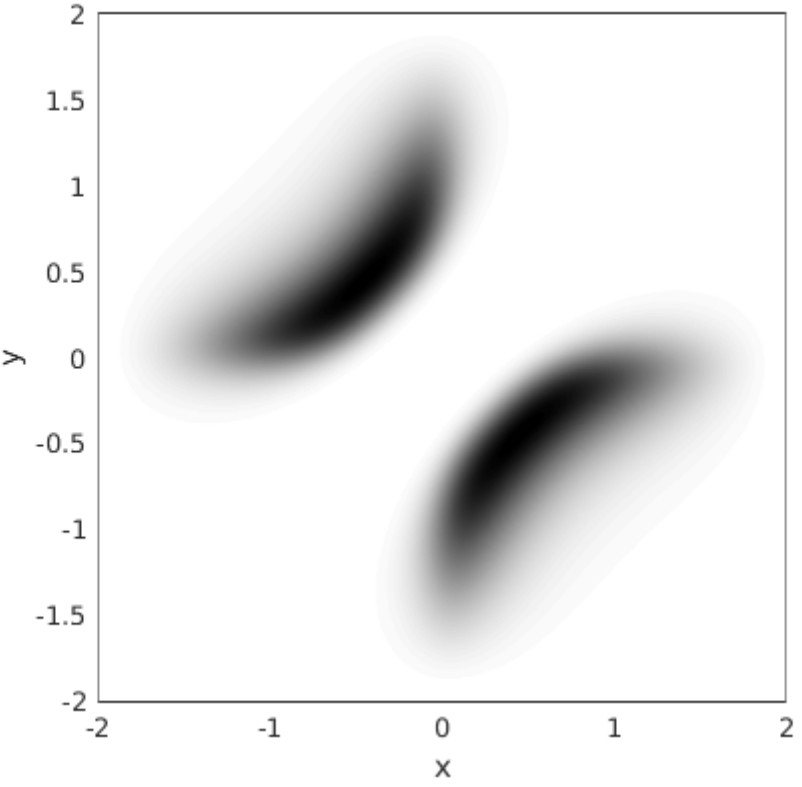}&
\includegraphics[ scale=0.12]{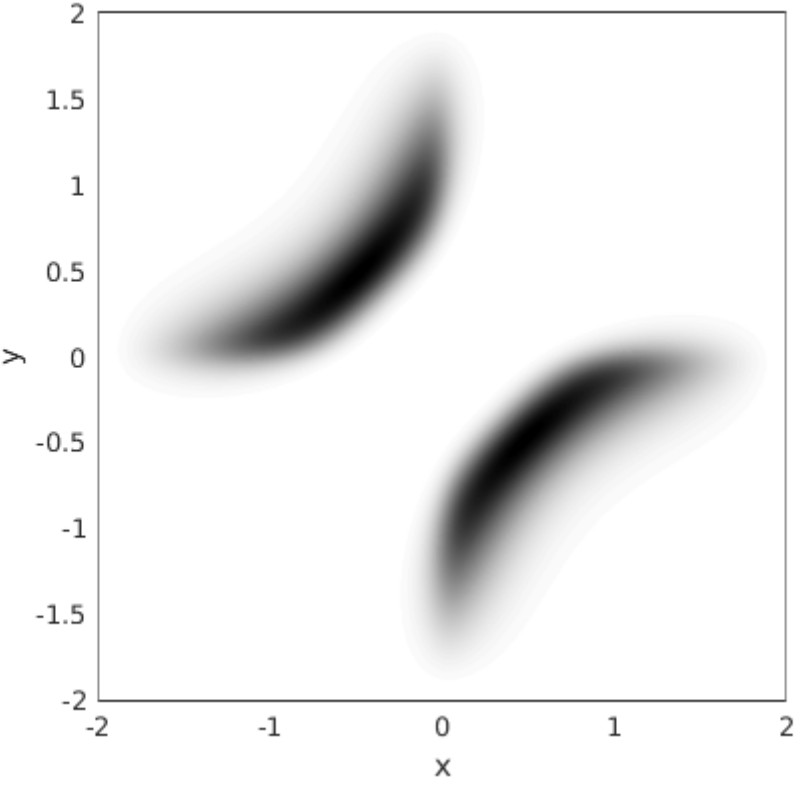}&
\includegraphics[ scale=0.123]{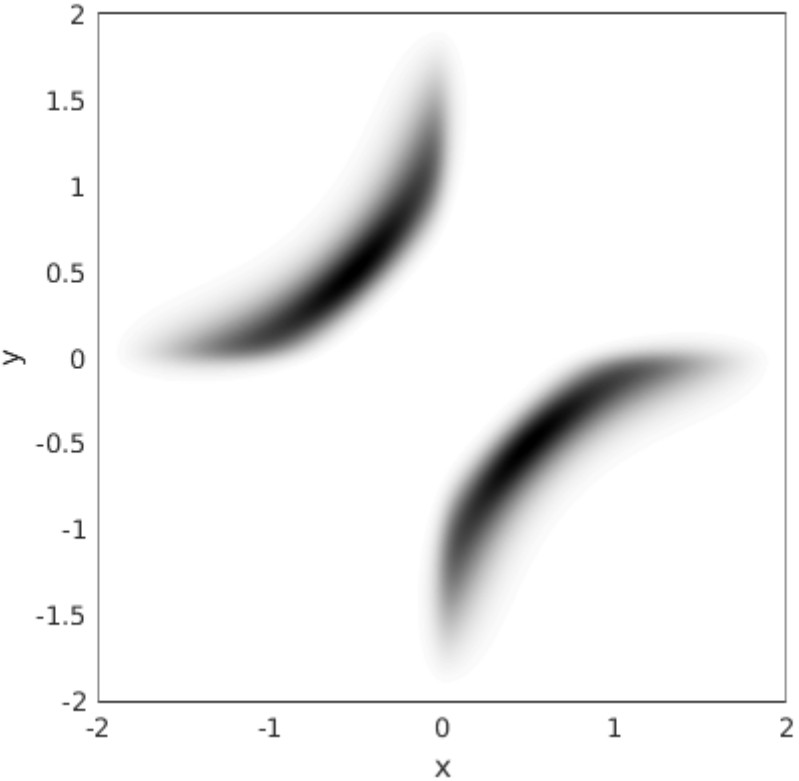}\\
$\epsilon=0.2$ &$\epsilon=0.1$&$\epsilon=0.05$\\  
\includegraphics[ scale=0.12]{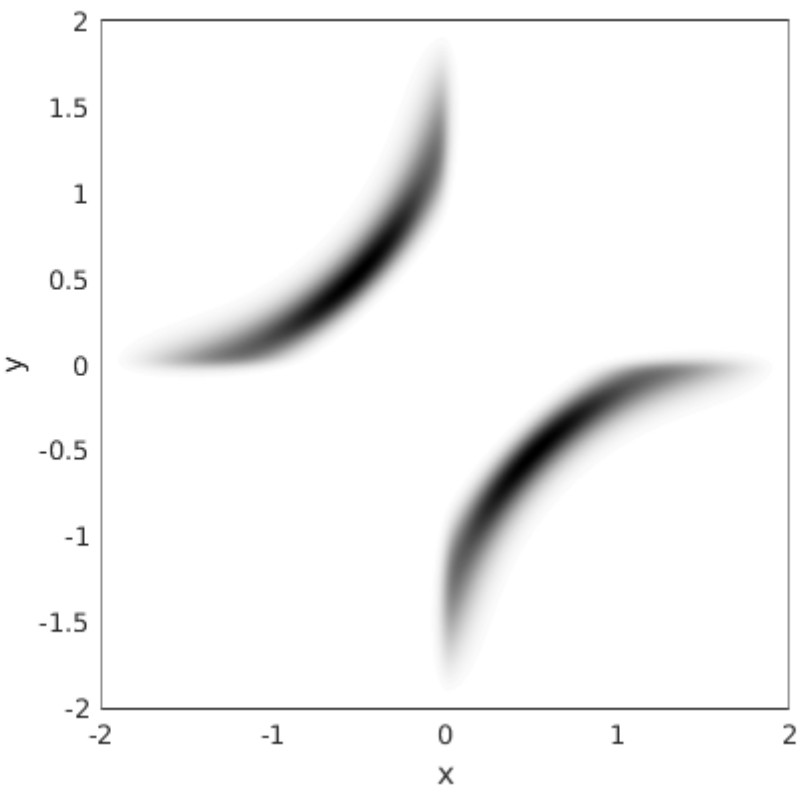}&
\includegraphics[ scale=0.12]{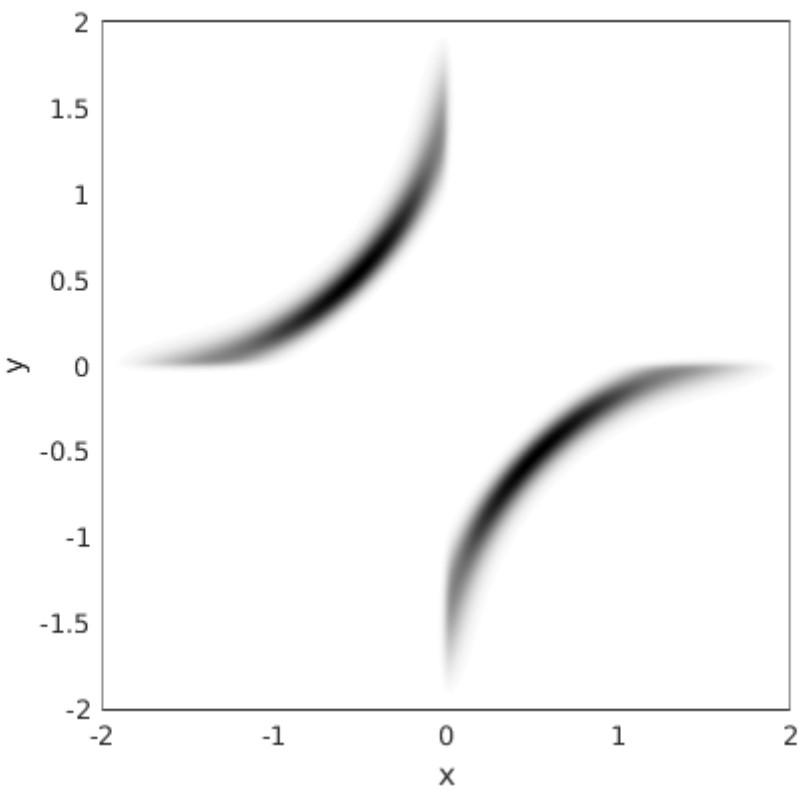}&
\includegraphics[ scale=0.12]{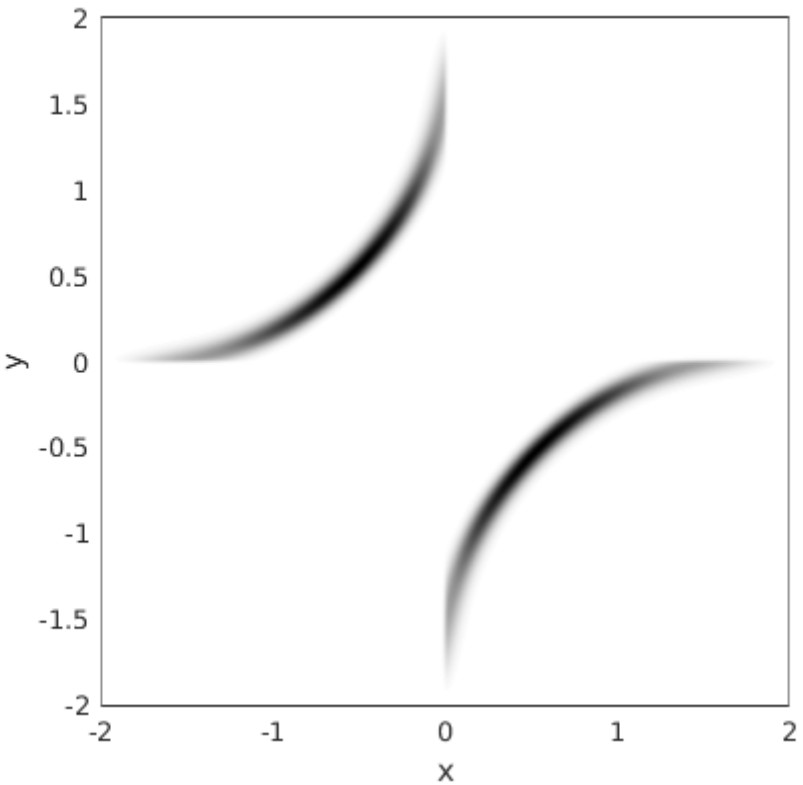}\\
$\epsilon=0.025$ &$\epsilon=0.0125$&$\epsilon=0.006$\\ 
\end{tabular}
\caption{\textit{Support of the coupling $\gamma^\epsilon$ for the Coulomb cost and $\mu_{1}=\mu_{2}=(1+cos(\pi x/2))/2$. The simulation has been performed
on a discretization of $[-2,2]$ with $M_d=1000$.}}
\label{figure:regularization}
\end{figure}

In order to introduce the Iterative Proportional Fitting Procedure (IPFP), we consider the two marginals problem 
\begin{equation}
\label{pd:MKdiscrete2}
 \min_{\gamma\in\mathcal{C}}\sum_{i,j} c_{ij}\gamma_{ij}.
\end{equation}
The aim of the IPFP is to find the KL projection of $\bar{\gamma}$ on the set $\mathcal{C}=\{\gamma_{ij}\in\mathbb{R}^{M_d}\times\mathbb{R}^{M_d} \ : \ \sum_{j}\gamma_{ij}=\mu_{i}\}\cap \{\gamma_{ij}\in\mathbb{R}^{M_d}\times\mathbb{R}^{M_d} \ : \ \sum_{i}\gamma_{ij}=\nu_{j}\}$.
By writing down the lagrangian associated to (\ref{pd:MKdiscrete2}) and computing the optimality condition, we find that $\gamma_{ij}$ can be written
as
\begin{equation}
 \gamma_{ij}=a_ib_j\bar{\gamma}_{ij}\quad with \quad a_i=e^{u_{i}/\epsilon},\quad b_j=e^{v_{j}/\epsilon},
\end{equation}
where $u_{i}$ and $v_{j}$ are the regularized Kantorovich potential.
Then, $a_i$ and $b_j$ can be uniquely determined by the marginal constraint
\begin{equation}
 a_{i}=\dfrac{\mu_{i}}{\sum_{j}b_{j}\bar{\gamma}_{ij}},\quad  b_{j}=\dfrac{\nu_{j}}{\sum_{i}a_{i}\bar{\gamma}_{ij}}.
\end{equation}
Thus, we can now define the following iterative method
\begin{equation}
 \label{IPFP}
 b_{j}^{n+1}=\dfrac{\nu_{j}}{\sum_{i}a_{i}^{n}\bar{\gamma}_{ij}},\quad a_{i}^{n+1}=\dfrac{\mu_{i}}{\sum_{j}b_{j}^{n+1}\bar{\gamma}_{ij}}.
\end{equation}
\begin{oss}
 In \cite{Rusch95} R\"{u}schendorf proves that the iterative method (\ref{IPFP}) converges to the KL-projection of $\bar{\gamma}$ on $\mathcal{C}$.
\end{oss}

\begin{oss}
R\"{u}schendorf and  Thomsen (see \cite{RuschendorfThomsen}) proved, in the continous measure framework, that a unique KL-projection exists and takes the form
$\gamma(x,y)=a(x)\otimes b(y)\bar{\gamma}(x,y)$ (where $a(x)$ and $b(y)$ are non-negative functions).
\end{oss}
The extension to the multi-marginal framework  is straightforward but 
cumbersone to write. 
It leads to a problem set on $N$  $M_d$-dimensional vectors $a_{j,i_{(\cdot)}},\quad \, j = 1,\cdots,N, \quad i_{(\cdot)}= 1,\cdots,M_d $. Each update takes the form 
$$  a_{j,i_{j}}^{n+1}  =\dfrac{\rho_{i_{j}}}{\sum_{i_1,i_2,... i_{j-1},i_{j+1},...,i_N   }\bar{\gamma}_{i_1,...,i_N} \,
  a_{1,i_1}^{n+1}   \,    a_{2,i_2}^{n+1}  ...   a_{j-1,i_{j-1}}^{n+1}   \,   a_{j+1,i_{j+1}}^{n}   ...   a_{N,{i_N}}^{n}   \,         }. \bigskip $$
\begin{exam}[IPFP and 3 marginals]
In order to clarify the extension of the IPFP to the multi-marginal case , we consider 3 marginals and we write down the updates of the algorithm
$$a_{1,i_{1}}^{n+1}=\dfrac{\rho_{i_{1}}}{\sum_{i_2,i_3}\bar{\gamma}_{i_1,i_2,i_3} \, a_{2,i_2}^{n}   \,    a_{3,i_3}^{n}}, \bigskip$$
$$a_{2,i_{2}}^{n+1}=\dfrac{\rho_{i_{2}}}{\sum_{i_1,i_3}\bar{\gamma}_{i_1,i_2,i_3} \, a_{1,i_1}^{n+1}   \,    a_{3,i_3}^{n}}, \bigskip$$
$$a_{3,i_{3}}^{n+1}=\dfrac{\rho_{i_{3}}}{\sum_{i_1,i_2}\bar{\gamma}_{i_1,i_2,i_3} \, a_{1,i_1}^{n+1}   \,    a_{2,i_2}^{n+1}}. \bigskip$$
\end{exam}
In the following sections we present some numerical results obtained by using the IPFP.

\subsection{Numerical experiments: Coulomb cost}

We present now some results for the multi-marginal problem with Coulomb cost in the real line.
We consider the case where the $N$ marginals are equal to a density $\mu$ (we work in a \textit{DFT} framework so the marginals rapresent the electrons which are indistinguishable).
We recall that if we split  $\mu$ into  $N$ $\tilde{\mu}_{i}$ with equal mass ($\int\tilde{\mu}_{i}(x)dx=\frac{1}{N}\int\mu(x)dx$), then we expect an optimal plan induced by a cyclical map
such that  $T_\sharp\tilde{\mu}_{i}=\tilde{\mu}_{i+1}$ $i=1,\cdots,N-1$ and
$T_\sharp\tilde{\mu}_{N}=\tilde{\mu}_{1}$.

The simulations in figure \ref{figure:coulombNcos} are all performed on a discretization of $[-5,5]$ with $M_d=200$,  with marginals $\mu_i=\mu(x)=\frac{N}{10}(1+cos(\frac{\pi}{5}x))\quad i=1,\cdots,N$ and $\epsilon=0.02$.
If we focus on the support of the coupling $\tilde{\gamma}_{12}(x,y)=(e_1,e_2)\gamma(x,y,z)$  we can notice that the numerical solution correctly reproduces the prescribed behavior: the transport plan is induced by a cyclical optimal
map (see section \ref{sec:Coulomb}). We refer the reader to \cite{BenCaNen} and \cite{DMaGGNe} for examples in higher dimension.

\begin{figure}[htbp]

\begin{tabular}{@{}c@{\hspace{4mm}}c@{}}

\centering
\includegraphics[ scale=0.25]{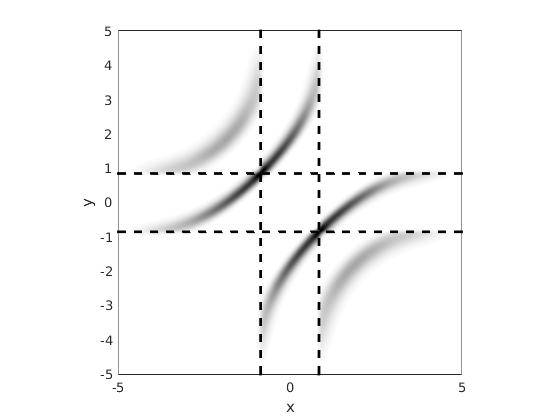}&
\includegraphics[ scale=0.25]{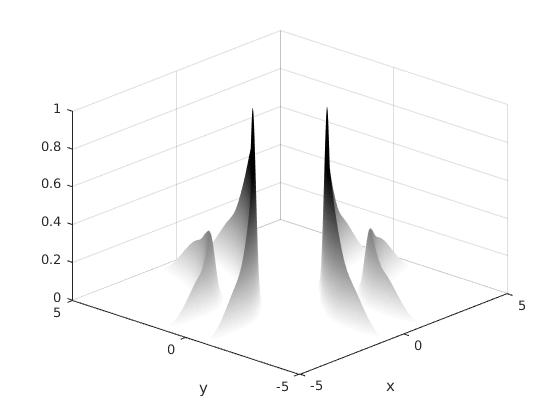}\\
\includegraphics[ scale=0.25]{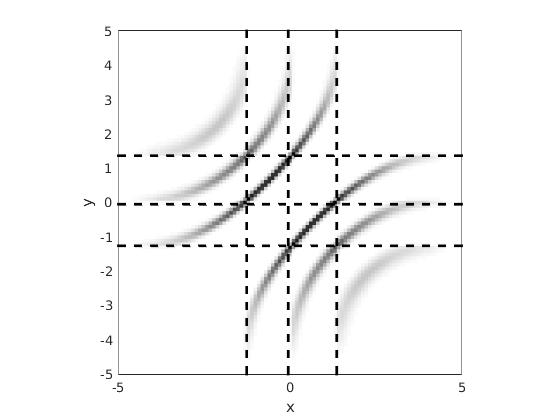}&
\includegraphics[ scale=0.25]{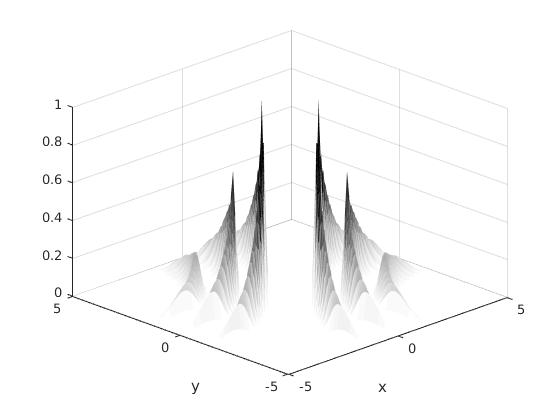}\\
\includegraphics[ scale=0.25]{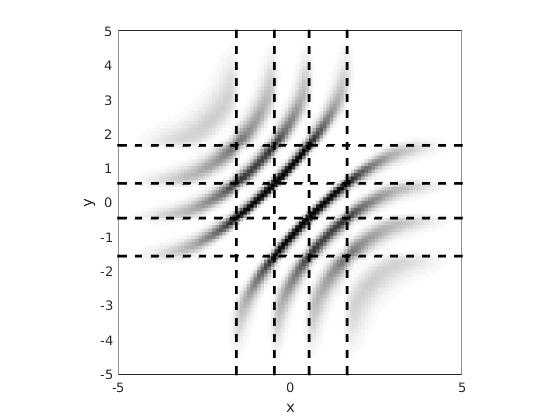}&
\includegraphics[ scale=0.25]{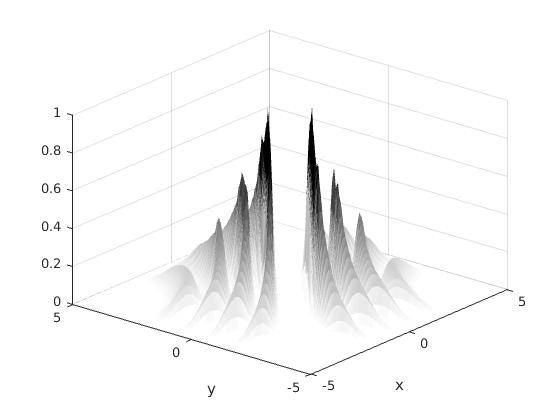}\\
\end{tabular}
\caption{\textit{On the Left: support of the optimal coupling $\tilde{\gamma}_{12}(x,y)$, on the Right: graph of the optimal coupling $\tilde{\gamma}_{12}(x,y)$
, in the cases $N=3$ (1st Row), $N=4$ (2nd Row), $N=5$ (3rd Row). The dashed lines delimit the supports of $\tilde{\mu}_{i}$ , with $i =1,\cdots,N$. }}

\label{figure:coulombNcos}
\end{figure}
\begin{oss}
 Theorem \ref{teo:1DN} actually works also for other costs functions as $c(x_{1},\cdots,x_{N})=\sum_{i<j}^{N} -log(\lvert x_i - x_j\rvert)$. This means that if we use the density
 $\mu(x)=\frac{N}{10}(1+cos(\frac{\pi}{5}x))$, we expect to obtain the same solution as for the Coulomb cost. Thus if we now consider the 3-marginals case and the discretized cost 
 $c_{j_1,j_2,j_3}=-log(\lvert x_{j_1}-x_{j_2}\rvert)-log(\lvert x_{j_1}-x_{j_3}\rvert)-log(\lvert x_{j_2}-x_{j_3}\rvert)$, we can notice (see figure \ref{figure:logNcos}) that we recover the same optimal coupling as the corresponding
 case for Coulomb (first row of figure \ref{figure:coulombNcos}). The simulation is performed on a discretizetion $[-5,5]$ with $M_d=200$,  with marginals $\mu(x)=\frac{N}{10}(1+cos(\frac{\pi}{5}x))\quad N=3$ and $\epsilon=0.01$.
\end{oss}
 \begin{figure}[htbp]

\begin{tabular}{@{}c@{\hspace{4mm}}c@{}}

\centering
\includegraphics[ scale=0.25]{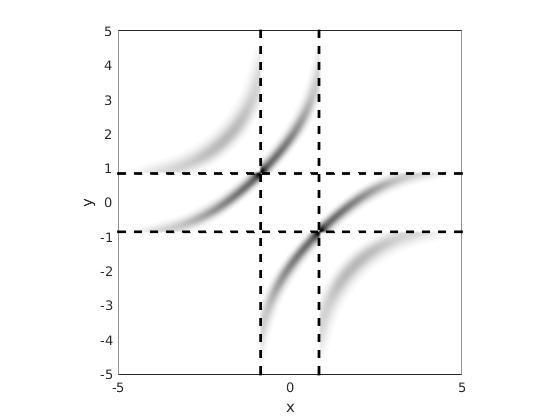}&
\includegraphics[ scale=0.25]{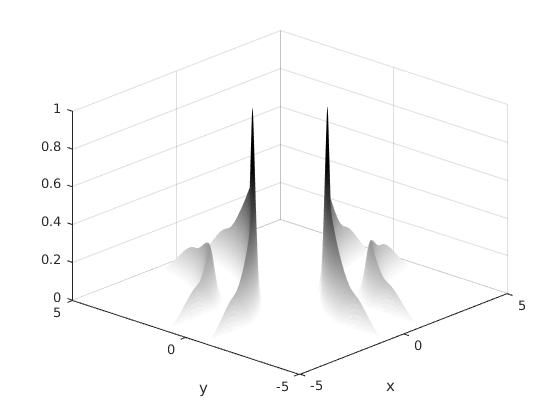}\\

\end{tabular}
\caption{\textit{(logarithmic cost) Left: support of the optimal coupling $\tilde{\gamma}_{12}(x,y)$ for $N=3$. Right: support of the optimal coupling $\tilde{\gamma}_{12}(x,y)$ for $N=3$.
The dashed lines delimit the intervals where $\tilde{\mu}_{i}$ , with $i =1,\cdots,3$, are defined.}}
\label{figure:logNcos}
\end{figure}

\subsection{Numerical experiments: repulsive Harmonic cost}
As shown in section \ref{sec:OTHarm} the minimizers of the OT problem with the harmonic cost are also minimizers of the problem with $c(x_{1},\cdots,x_{N})=\lvert x_{1}+\cdots+x_{N}\rvert^{2}$,
so for the discrete problem (\ref{pb:MKdiscrete}) we take $c_{j_{1}\cdots j_{N}}=\lvert x_{j_{1}}+\cdots+x_{j_{N}}\rvert^{2}$.
Let us first consider the two marginal case and the uniform density on $[0,1]$ (see example \ref{OTHarm:Exam2Marginals}), and, as expected, we find a deterministic coupling
given by $\gamma(x,y)=\mu(x)\delta(y-T(x))$ where $T(x)=1-x$, see figure \ref{figure:harmonic2marginals}.
The simulation in figure \ref{figure:harmonic2marginals} has been performed on a discretization of $[0,1]$ with $M_d=1000$ gridpoints
and $\epsilon=0.005$.

\begin{figure}[htbp]

\begin{tabular}{@{}c@{\hspace{3mm}}c@{}}

\centering
\includegraphics[ scale=0.10]{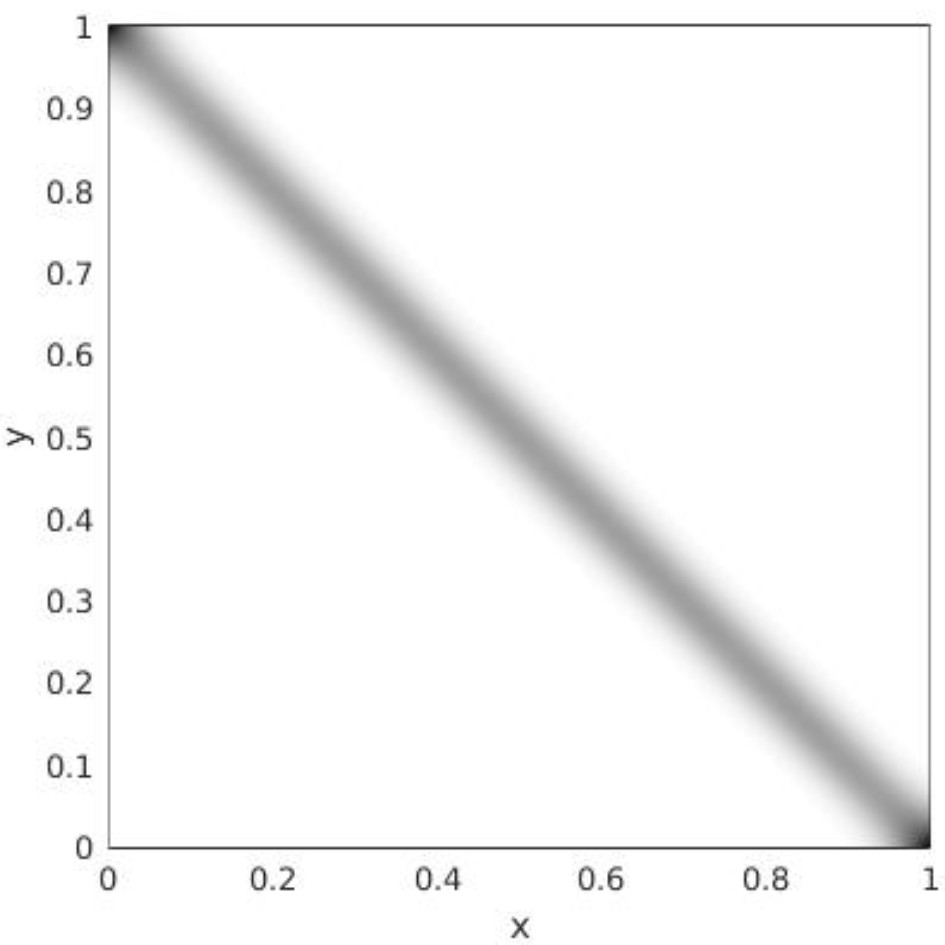}&
\includegraphics[ scale=0.15]{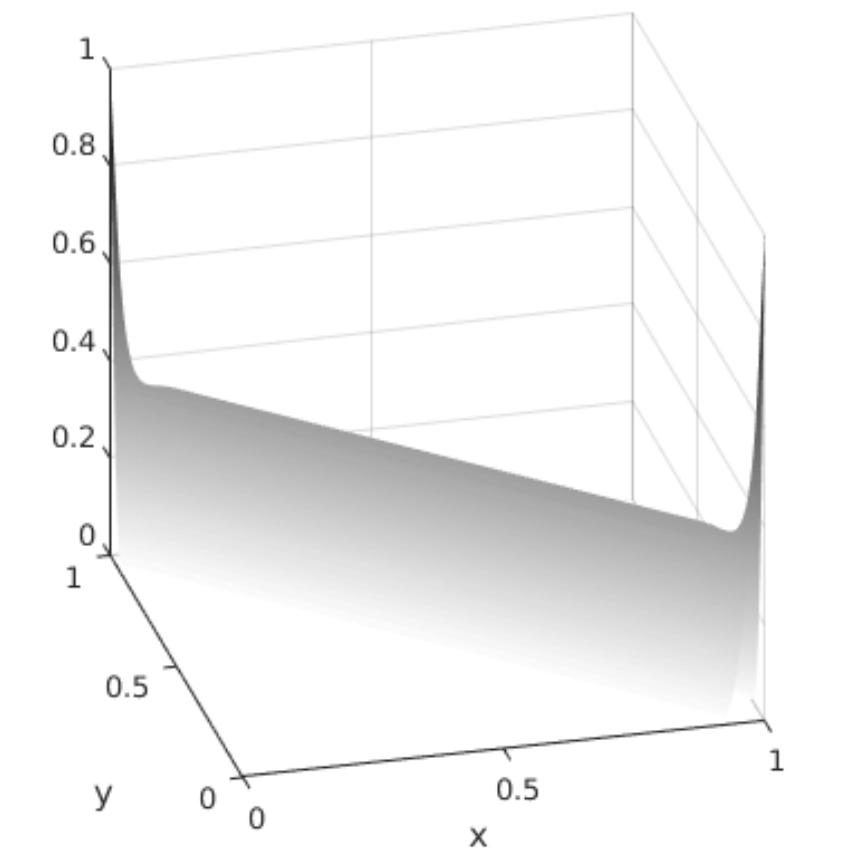}\\

\end{tabular}
\caption{\textit{Left: support of the optimal coupling. Right: optimal coupling.}}
\label{figure:harmonic2marginals}
\end{figure}
The multi-marginals case is more delicate to treat: the original OT problem does not admit a unique solution whereas the regularized problem does.
As we have explained in section \ref{subIPFP} the regularized problem is a strongly convex problem which admits a unique solution and the resulting coupling
is the one with the minimal entropy.
However we are able to make the IPFP algorithm converge to a selected coupling among the optimal ones for the original problem.
Let us focus on the example \ref{OTHarm:ex:twomapsplan}. In this case we have all marginals equal to $\mu=\frac{1}{2}\mathcal{L}|_{[-1,1]}$ and we can find a deterministic coupling
given by $\gamma(x,y,z)=\mu(x)\delta(y-T(x))\delta(z-S(x))$ for the maps $T(x)$ and $S(x)$).
In order to select this coupling, the idea is to modify the cost function by adding a penalization term $p(x,y)=\frac{\tau}{\lvert x-y\rvert}$ which makes $\mu_1$ and $\mu_2$ be \textit{as far as possible} (if we consider $\mu_1$ and $\mu_2$ as particles).
Thus, the discretized cost now reads as 
\begin{equation}
c_{j_{1},j_{2},j_{3}}=\lvert x_{j_{1}}+x_{j_{2}}+x_{j_{3}}\rvert^{2}+\dfrac{\tau}{\lvert x_{j_{1}}-x_{j_{2}}\rvert}.
\end{equation}

Then we expect, as shown in figure \ref{figure:twomaps}, that the projection $\tilde{\gamma}_{12}(x,y)=(e_1,e_2)\gamma(x,y,z)$ and $\tilde{\gamma}_{13}(x,z)=(e_1,e_3)\gamma(x,y,z)$ 
of the computed coupling are induced by map $T$ and $S$ respectively.
The simulation has been performed on a discretization of $[-1,1]$ with $M_d=1000$ gridpoints, $\epsilon=0.0005$ and $\eta=0.1$.\\
\begin{figure}[htbp]

\begin{tabular}{@{}c@{\hspace{3mm}}c@{}}

\centering
\includegraphics[ scale=0.10]{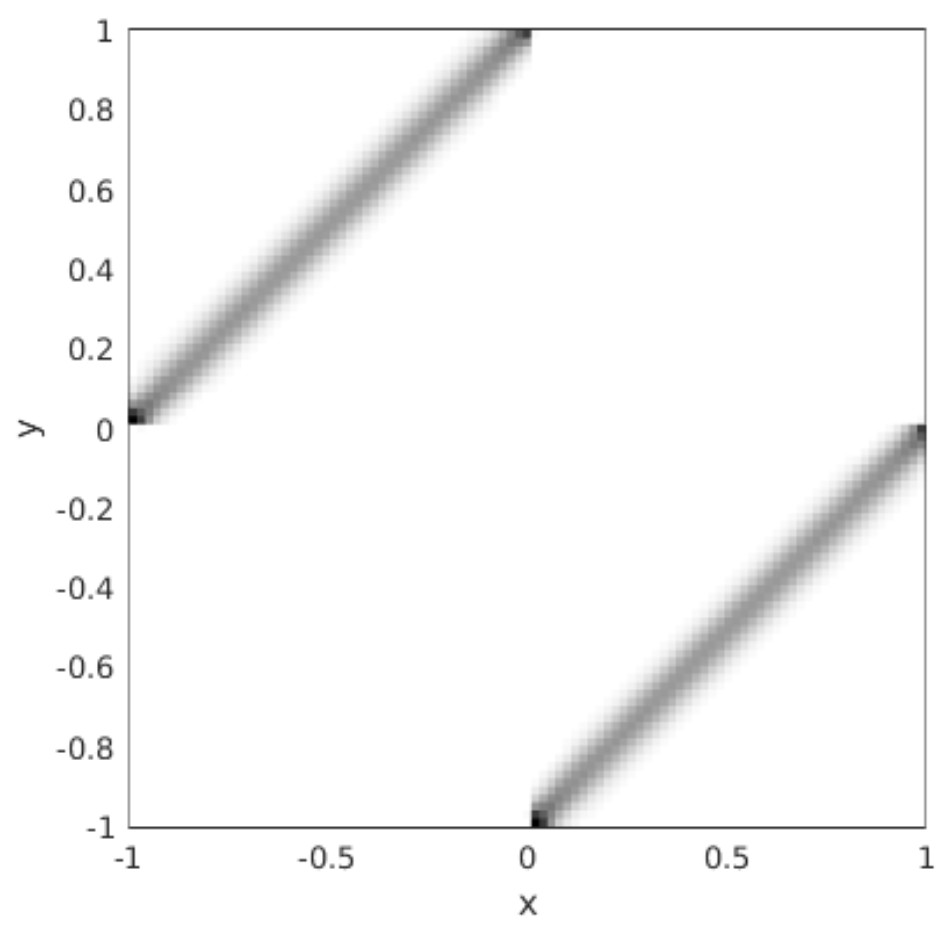}&
\includegraphics[ scale=0.14]{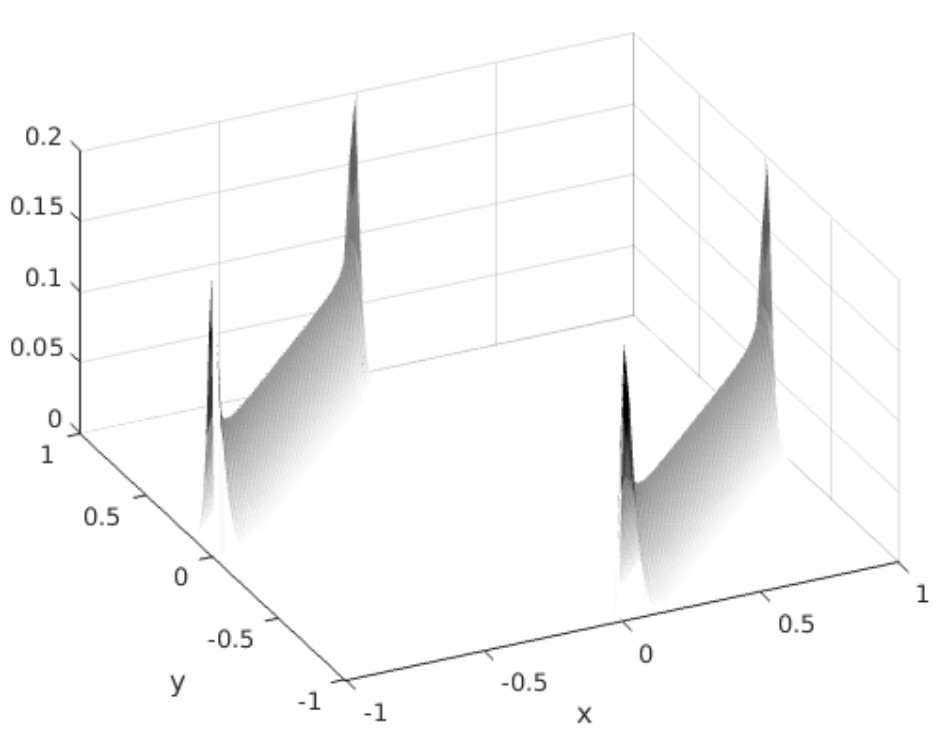}\\
\includegraphics[ scale=0.10]{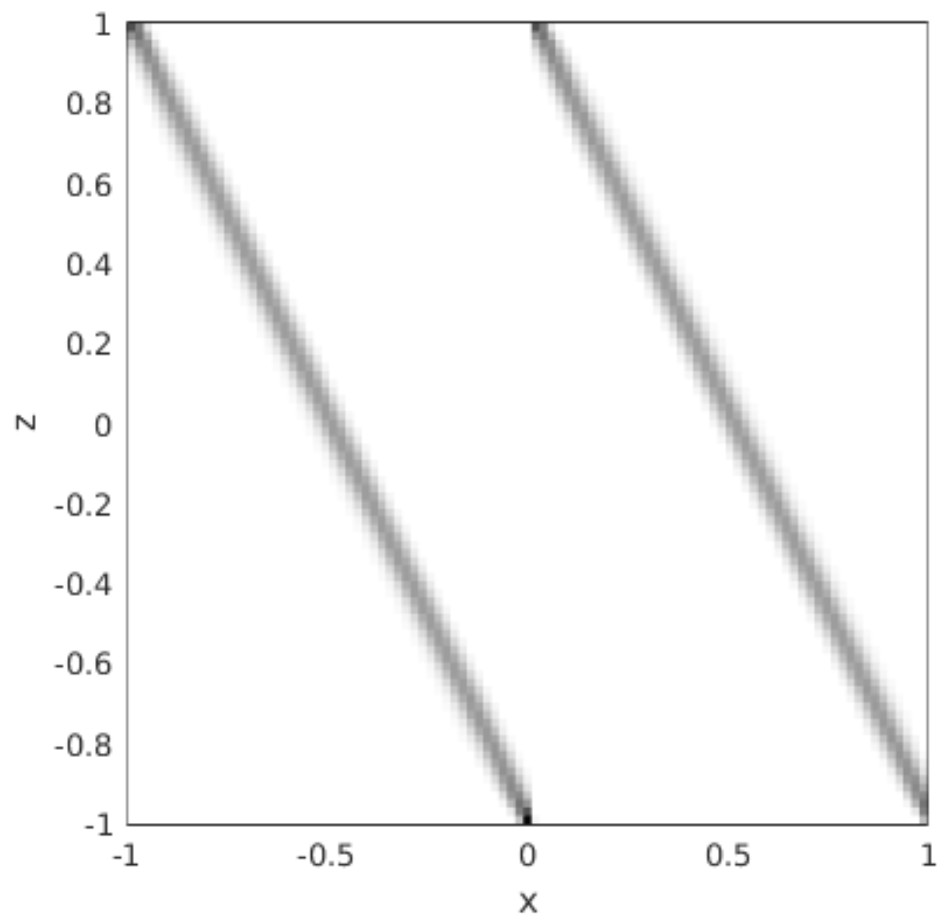}&
\includegraphics[ scale=0.14]{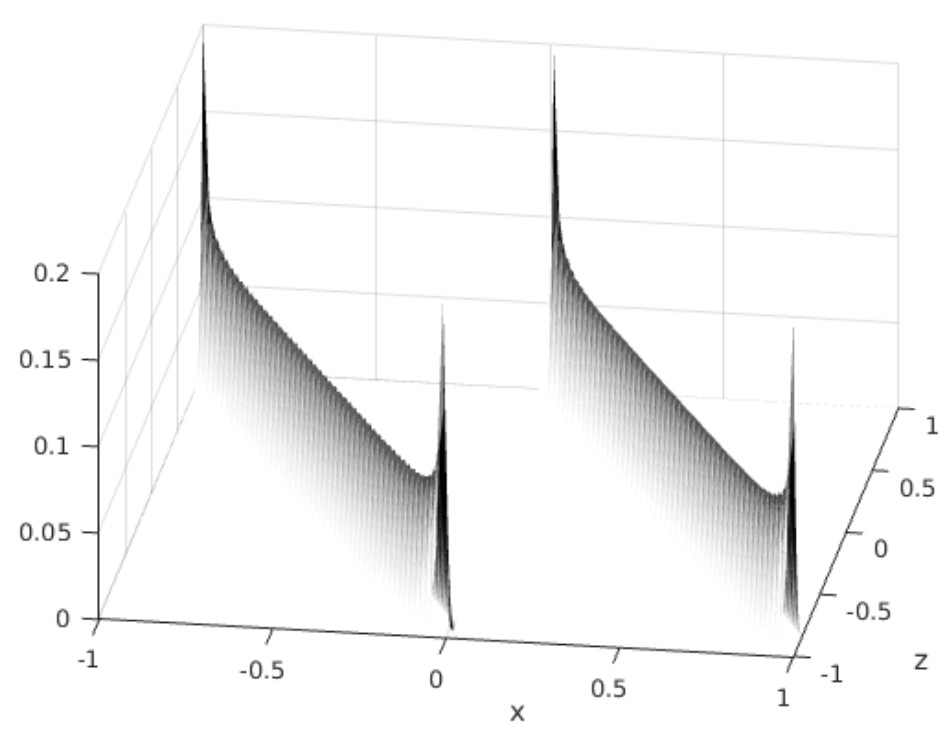}\\
\end{tabular}
\caption{\textit{Top-Left: support of the optimal coupling $\tilde{\gamma}_{12}(x,y)$. Top-Right: optimal coupling $\tilde{\gamma}_{12}(x,y)$.
Bottom-Left: support of the optimal coupling $\tilde{\gamma}_{13}(x,z)$. Bottom-Right: optimal coupling $\tilde{\gamma}_{13}(x,z)$.}}
\label{figure:twomaps}
\end{figure}
We, finally, take the same marginals as in the previous example, but we do not add the penalization. In this case 
 we expect
a diffuse plan since, if we forget about the marginal constraint, it is simple to show that the second order $\Gamma$-limit of the regularized problem as $\epsilon \to 0$ is the relative entropy of $\gamma$ with respect to $\mathcal{H}^{2}|_{\{x+y+z=0\}}$; we expect that the $\Gamma$-limit is the same also when we add the marginal constraint. 

In figure \ref{figure:diffplan} we present  the projection $\tilde{\gamma}_{12}(x,y)=(e_1,e_2)\gamma(x,y,z)$ of the computed coupling
(it is enough to visualize only this projection because of the symmetries of the problem) and, as expected, we observe a diffuse plan. .
The simulation has been performed on a discretization of $[-1,1]$ with $M_d=1000$ gridpoints, $\epsilon=0.0005$.
\begin{figure}[htbp]

\begin{tabular}{@{}c@{\hspace{3mm}}c@{}}
\centering
\includegraphics[ scale=0.12]{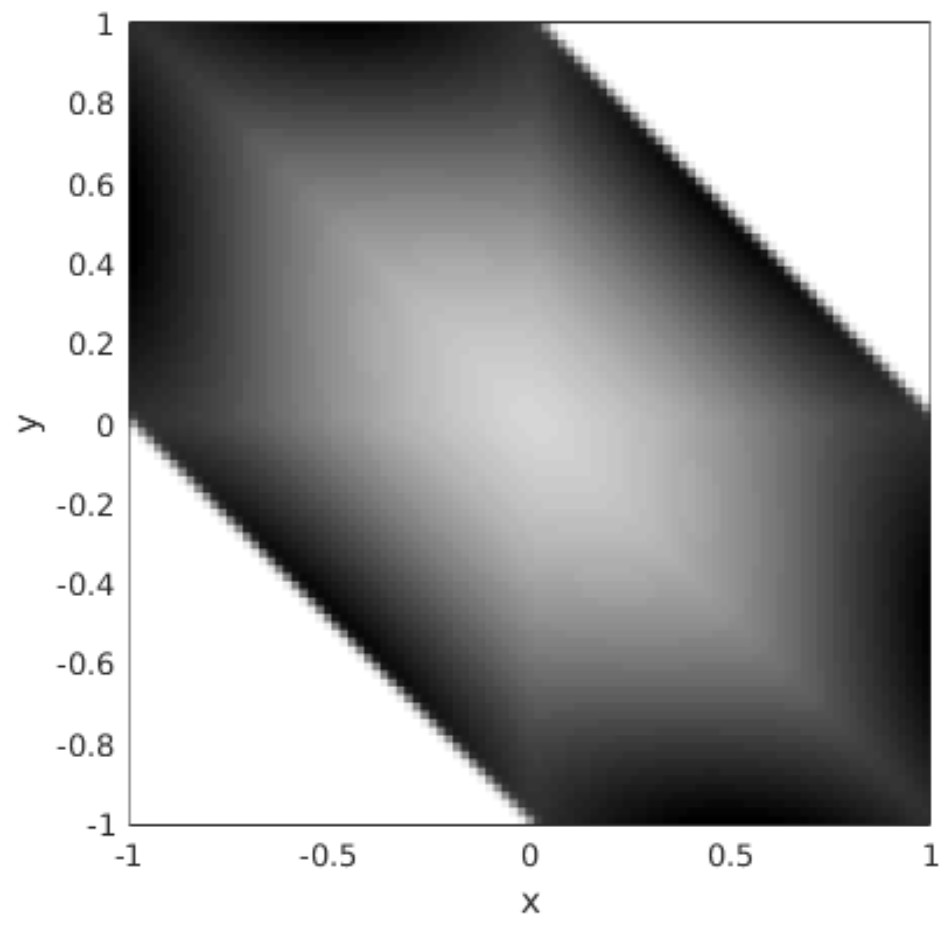}&
\includegraphics[ scale=0.16]{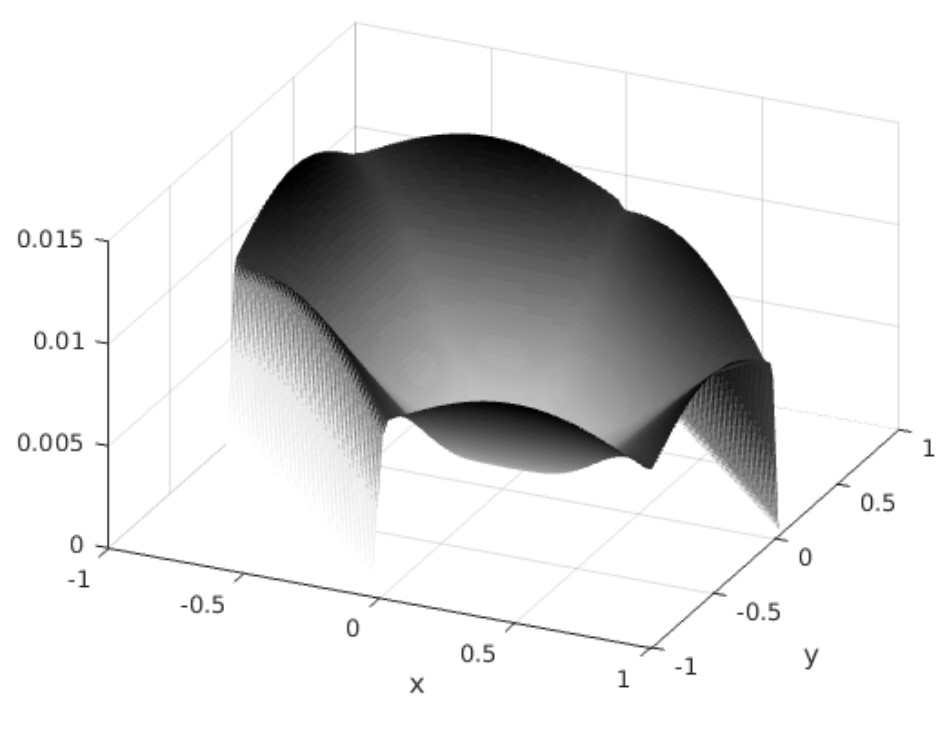}\\

\end{tabular}
\caption{\textit{Left: support of the optimal coupling $\tilde{\gamma}_{12}(x,y)$. Right: optimal coupling $\tilde{\gamma}_{12}(x,y)$.
}}
\label{figure:diffplan}
\end{figure}

\subsection{Numerical experiments: Determinant cost}
Numerical simulations for the multi-marginal problem with the determinant cost present an obvious computational difficulty: in order to compute the solution
just for the 3-marginals case, we have to introduce a discretization of $\mathbb{R}^3$ (as the cost function now is $c(x_1,\cdots,x_N)=det(x_1,\cdots,x_N)$ with $x_{i}\in\mathbb{R}^N$).
However if we take radially symmetric densities $\mu_{i}$ as marginals (see section \ref{sec:Det} for a more complete description of the problem in the measure framework), Carlier and Nazaret show that  ($\MK$) can be reduced to a 1-dimensional problem:
the only unknowns are the relations between the norms  $r_{i}=\| x_{i} \|$ of each vector.
These relations can be obtained by solving the following problem
\begin{equation}
 \label{radialPb}
 \min_{\gamma\in\Pi(\lambda_1(r),\cdots,\lambda_N)}-\int(\prod_{i=1}^{N}r_i)\gamma(r_1,\cdots,r_N)dr_1\cdots dr_N,
\end{equation}
where $\lambda_{i}(r)=\sigma(r)\mu_{i}(\textbf{r})$ (e.g.  $N=3$ then $\textbf{r}=(r,\theta,\varphi)$ and $\sigma(r)=4\pi r^2$).
Moreover, for this problem we know that there exists a unique (deterministic) optimal coupling (see Proposition $5$ in \cite{CarNa}).
The discretized cost now reads $c_{j_{1}\cdots j_{N}}=\prod_{k=1}^{N} r_{j_{k}}$.
Let us consider the 3-marginals case and all densities $\lambda_{i}$ equal to the uniform density on the ball $\mathcal{B}_{0.5}=\{ x\in\mathbb{R}^{3}\lvert \|x\|\leq 0.5\}$.
As proved in \cite{CarNa}, the optimal coupling is  actually supported by the graph of  maps which rearrange measure (see figure \ref{figure:detUniform} ). 
The simulation in figure \ref{figure:detUniform} has been performed on a discretization of $[0,0.5]$ with $M_d=100$ and $\epsilon=0.01$.
\begin{figure}[htbp]

\begin{tabular}{@{}c@{\hspace{1mm}}c@{}}
\centering
\includegraphics[ scale=0.20]{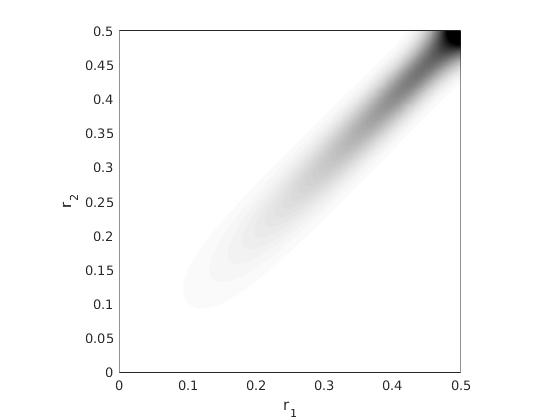}&
\includegraphics[ scale=0.20]{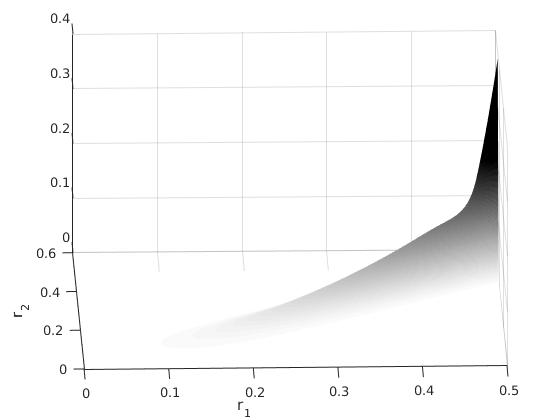}\\
\end{tabular}
\caption{\textit{(Uniform density) Left: support of the optimal coupling $\tilde{\gamma}_{12}(r_1,r_2)$. Right: optimal coupling $\tilde{\gamma}_{12}(r_1,r_2)$.
}}
\label{figure:detUniform}
\end{figure}
In the same way we can take all marginals $\lambda_{i}(r)=4\pi r^{2}e^{-4r}$ and we obtain again a deterministic coupling, figure \ref{figure:detExponential}.
The simulation has been performed on a discretization of $[0,3]$ with $M_d=300$ and $\epsilon=0.05$.
\begin{figure}[htbp]

\begin{tabular}{@{}c@{\hspace{1mm}}c@{}}
\centering
\includegraphics[ scale=0.20]{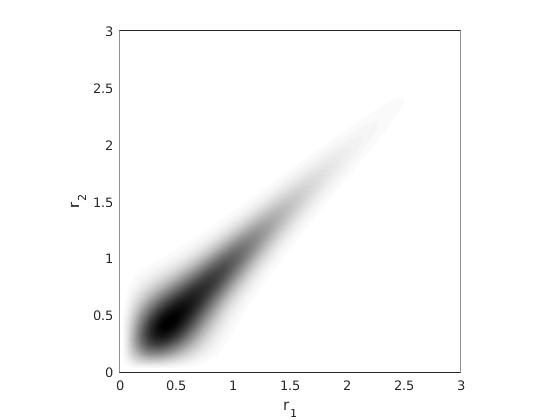}&
\includegraphics[ scale=0.20]{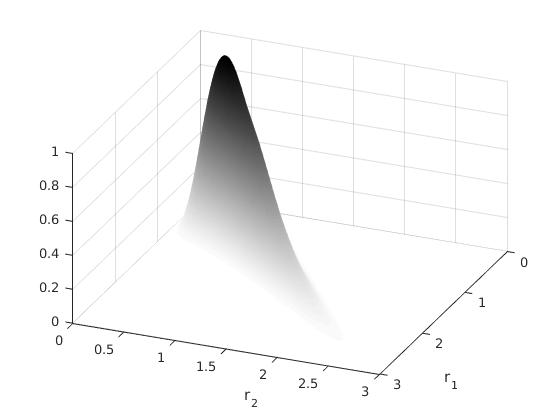}\\
\end{tabular}
\caption{\textit{(Exponential density)Left: support of the optimal coupling $\tilde{\gamma}_{12}(r_1,r_2)$. Right: optimal coupling $\tilde{\gamma}_{12}(r_1,r_2)$.
}}
\label{figure:detExponential}
\end{figure}

We, finally, present a simulation (see figure \ref{figure:planden}) with three different marginals $\lambda_{i}$ (see figure \ref{figure:den}).
As one can observe we obtain that the projection of the optimal coupling are concentrated on the graph of a map which rearranges the densities in a monotone way.
The simulation has been performed on a discretization of $[0,4]$ with $M_d=300$ and $\epsilon=0.07$.
\begin{figure}[htbp]

\begin{tabular}{@{}c@{\hspace{1mm}}c@{\hspace{1mm}}c@{}}
\centering
\includegraphics[ scale=0.18]{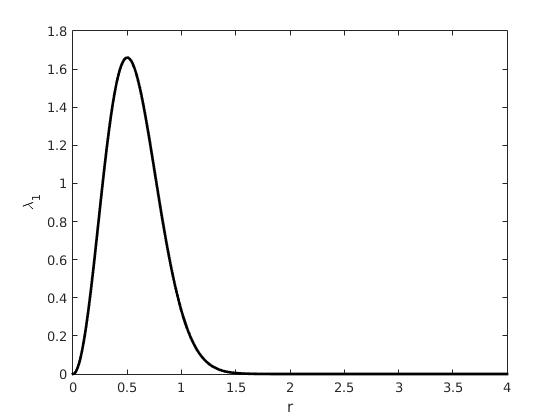}&
\includegraphics[ scale=0.18]{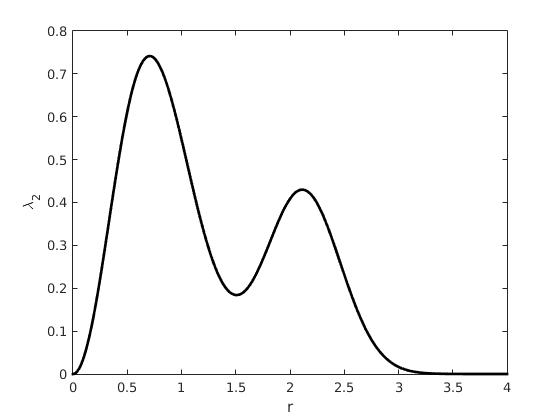}&
\includegraphics[ scale=0.18]{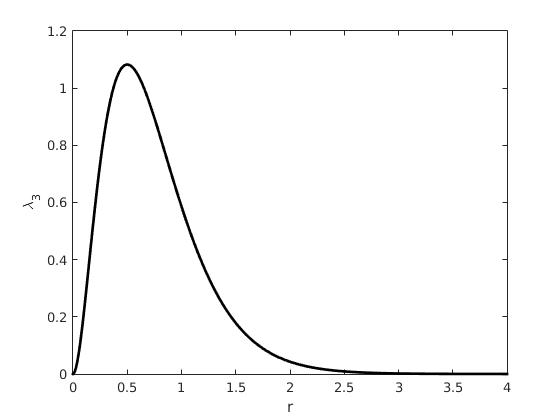}\\
$\lambda_1$&
$\lambda_2$&
$\lambda_3$\\
\end{tabular}
\caption{\textit{Densities $\lambda_1$, $\lambda_2$ and $\lambda_3$.
}}
\label{figure:den}
\end{figure}

\begin{figure}[htbp]

\begin{tabular}{@{}c@{\hspace{1mm}}c@{\hspace{1mm}}c@{}}
\centering
\includegraphics[ scale=0.18]{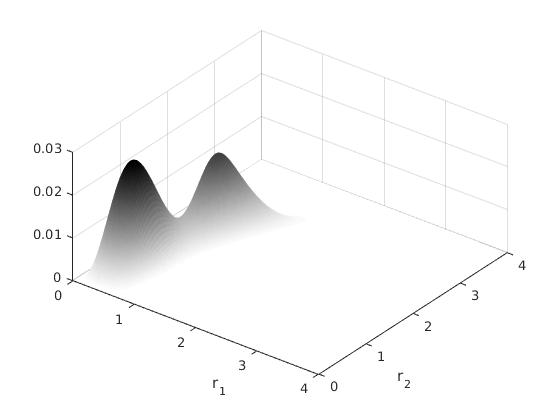}&
\includegraphics[ scale=0.18]{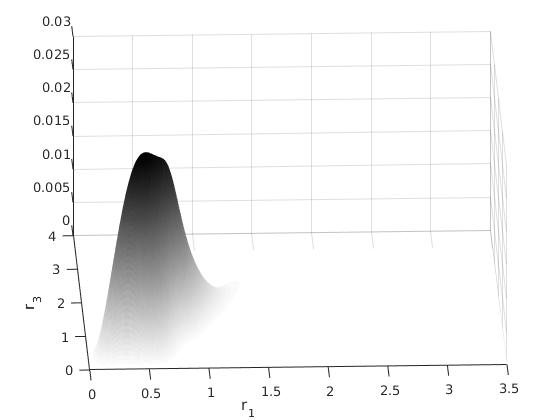}&
\includegraphics[ scale=0.18]{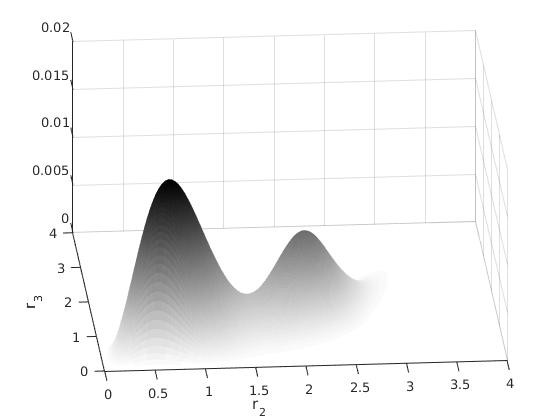}\\
\includegraphics[ scale=0.18]{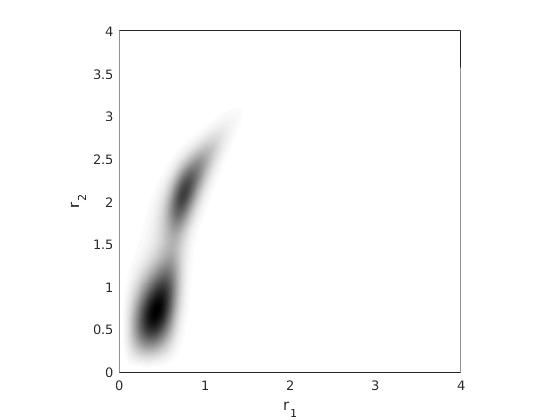}&
\includegraphics[ scale=0.18]{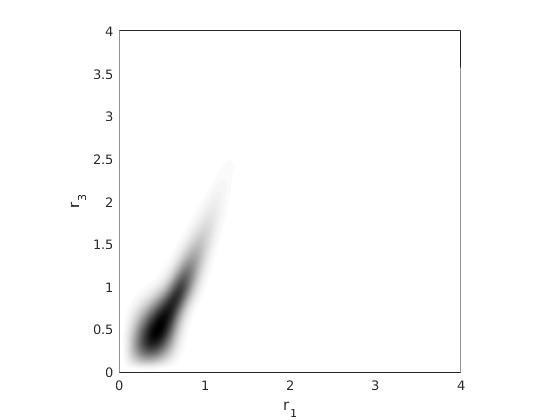}&
\includegraphics[ scale=0.18]{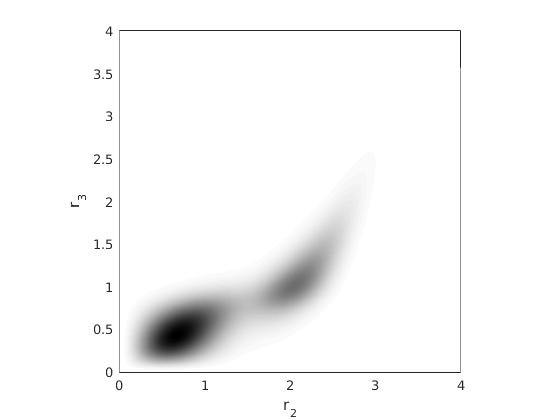}\\
\end{tabular}
\caption{\textit{Top-Left: optimal coupling $\tilde{\gamma}_{12}(r_1,r_2)$. Top-Center: optimal coupling $\tilde{\gamma}_{13}(r_1,r_3)$. Top-Right: optimal coupling $\tilde{\gamma}_{23}(r_2,r_3)$.
Bottom: supports of those couplings.
}}
\label{figure:planden}
\end{figure}

%% file: conclusion.tex
\section{Conclusion}
In this survey we presented a general famework for the multi-marginal optimal transportation with repulsive cost.
We also gave some new results which highlight the importance of the multi-marginal approach.
As mentioned there are still a lot of open questions about the solutions for the problems presented and we hope to treat them in future works.

\section*{Acknowledgements}

We would like to thank Paola Gori-Giorgi for many helpful and stimulating discussions as well as
Jean-David Benamou and Guillaume Carlier for a critical reading of the manuscript.

Simone Di Marino and Luca Nenna gratefully acknowledge the support of the ANR, through the project ISOTACE (ANR-12-MONU-0013).
Luca Nenna also ackwoledges INRIA through the ``action exploratoire'' MOKAPLAN.

Augusto Gerolin is member of the \textit{Gruppo Nazionale per l'Analisi Matematica, la Probabilit\`a e le loro Applicazioni (GNAMPA)} of the \textit{Instituto Nazionale di Alta Matematica (INdAM)}. He also acknowledges and thanks CNPq for the financial support of his Ph.D, when this work was done.

%% file: SurveyRicam.bbl
\begin{thebibliography}{10}

\bibitem{AbrAbrBerCar}
I~Abraham, R~Abraham, M~Bergounioux, and G~Carlier.
\newblock {Tomographic reconstruction from a few views: a multi-marginal
  optimal transport approach}.
\newblock {\em Preprint Hal-01065981}, 2014.

\bibitem{AguCar}
M~Agueh and G~Carlier.
\newblock {Barycenters in the \{W\}asserstein space}.
\newblock {\em SIAM J. on Mathematical Analysis}, 43(2):904--924, 2011.

\bibitem{AGS}
Luigi Ambrosio, Nicola Gigli, and Giuseppe Savar{\'e}.
\newblock {\em Gradient flows: in metric spaces and in the space of probability
  measures}.
\newblock Springer Science \& Business Media, 2006.

\bibitem{BeiGri}
Mathias Beiglb{\"o}ck and Claus Griessler.
\newblock An optimality principle with applications in optimal transport.
\newblock {\em arXiv preprint arXiv:1404.7054}, 2014.

\bibitem{BeiHenPen}
Mathias Beiglb{\"o}ck, Pierre Henry-Labord{\`e}re, and Friedrich Penkner.
\newblock Model-independent bounds for option prices---a mass transport
  approach.
\newblock {\em Finance and Stochastics}, 17(3):477--501, 2013.

\bibitem{BeiLeoSch}
Mathias Beiglb{\"o}ck, Christian L{\'e}onard, and Walter Schachermayer.
\newblock A general duality theorem for the monge--kantorovich transport
  problem.
\newblock {\em arXiv preprint arXiv:0911.4347}, 2009.

\bibitem{BeiCoxHue}
Mathias Beiglboeck, Alexander Cox, and Martin Huesmann.
\newblock Optimal transport and skorokhod embedding.
\newblock {\em arXiv preprint arXiv:1307.3656}, 2013.

\bibitem{BeiJul}
Mathias Beiglboeck and Nicolas Juillet.
\newblock On a problem of optimal transport under marginal martingale
  constraints.
\newblock {\em arXiv preprint arXiv:1208.1509}, 2012.

\bibitem{Ben}
Jean-David Benamou, Guillaume Carlier, Marco Cuturi, Luca Nenna, and Gabriel
  Peyr{\'e}.
\newblock Iterative bregman projections for regularized transportation
  problems.
\newblock {\em SIAM Journal on Scientific Computing}, 37(2):A1111--A1138, 2015.

\bibitem{BenCaNen}
Jean-David Benamou, Guillaume Carlier, and Luca Nenna.
\newblock A numerical method to solve optimal transport problems with coulomb
  cost.
\newblock {\em arXiv preprint arXiv:1505.01136, to appear as a chapter in
  "Splitting Methods in Communication and Imaging, Science and Engineering",
  Editors R. Glowinski, S. Osher, and W. Yin., Springer}, 2015.

\bibitem{Bodo}
Enrico Bodo.
\newblock Applicazioni di meccanica quantistica: appunti per le lezioni.
  lecture notes for a course in quantum chemistry.
\newblock {\em Available (April/2015) at
  http://w3.uniroma1.it/bodo/dispense/appl.pdf}, 2008.

\bibitem{BuDePGor}
Giuseppe Buttazzo, Luigi De~Pascale, and Paola Gori-Giorgi.
\newblock Optimal-transport formulation of electronic density-functional
  theory.
\newblock {\em Physical Review A}, 85(6):062502, 2012.

\bibitem{Caf}
Luis~A Caffarelli.
\newblock A localization property of viscosity solutions to the monge-ampere
  equation and their strict convexity.
\newblock {\em Annals of Mathematics}, pages 129--134, 1990.

\bibitem{Caf2}
Luis~A Caffarelli.
\newblock Some regularity properties of solutions of monge ampere equation.
\newblock {\em Communications on pure and applied mathematics},
  44(8-9):965--969, 1991.

\bibitem{CarEke}
G~Carlier and I~Ekeland.
\newblock {Matching for teams}.
\newblock {\em Econom. Theory}, 42(2):397--418, 2010.

\bibitem{CarGalChe}
Guillaume Carlier, Victor Chernozhukov, and Alfred Galichon.
\newblock Vector quantile regression.
\newblock {\em arXiv preprint arXiv:1406.4643}, 2014.

\bibitem{CarNa}
Guillaume Carlier and Bruno Nazaret.
\newblock Optimal transportation for the determinant.
\newblock {\em ESAIM: Control, Optimisation and Calculus of Variations},
  14(04):678--698, 2008.

\bibitem{CarObeOud}
Guillaume Carlier, Adam Oberman, and Edouard Oudet.
\newblock Numerical methods for matching for teams and wasserstein barycenters.
\newblock {\em arXiv preprint arXiv:1411.3602}, 2014.

\bibitem{ChenFri}
Huajie Chen and Gero Friesecke.
\newblock Pair densities in density functional theory.
\newblock {\em arXiv preprint arXiv:1503.07539}, 2015.

\bibitem{ChenFriMendl}
Huajie Chen, Gero Friesecke, and Christian~B Mendl.
\newblock Numerical methods for a kohn--sham density functional model based on
  optimal transport.
\newblock {\em Journal of Chemical Theory and Computation}, 10(10):4360--4368,
  2014.

\bibitem{ChiGalSal}
Pierre-Andr{\'e} Chiappori, Alfred Galichon, and Bernard Salani{\'e}.
\newblock The roommate problem is more stable than you think.
\newblock 2014.

\bibitem{ChiMcCNes}
Pierre-Andr{\'e} Chiappori, Robert~J McCann, and Lars~P Nesheim.
\newblock Hedonic price equilibria, stable matching, and optimal transport:
  equivalence, topology, and uniqueness.
\newblock {\em Economic Theory}, 42(2):317--354, 2010.

\bibitem{CoMoYa}
Aron~J Cohen, Paula Mori-S{\'a}nchez, and Weitao Yang.
\newblock Challenges for density functional theory.
\newblock {\em Chemical Reviews}, 112(1):289--320, 2011.

\bibitem{CoDMa}
Maria Colombo and Simone Di~Marino.
\newblock Equality between monge and kantorovich multimarginal problems with
  coulomb cost.
\newblock {\em Annali di Matematica Pura ed Applicata}, pages 1--14, 2013.

\bibitem{CoStra}
Maria Colombo and Federico Stra.
\newblock Counterexamples to multimarginal optimal transport maps with coulomb
  cost and radial measures.
\newblock {\em in preparation}.

\bibitem{CominettiAsympt}
Roberto Cominetti and Jaime San~Mart{\'\i}n.
\newblock Asymptotic analysis of the exponential penalty trajectory in linear
  programming.
\newblock {\em Mathematical Programming}, 67(1-3):169--187, 1994.

\bibitem{CFK}
Codina Cotar, Gero Friesecke, and Claudia Kl{\"u}ppelberg.
\newblock Density functional theory and optimal transportation with coulomb
  cost.
\newblock {\em Communications on Pure and Applied Mathematics}, 66(4):548--599,
  2013.

\bibitem{CotFriPass}
Codina Cotar, Gero Friesecke, and Brendan Pass.
\newblock Infinite-body optimal transport with coulomb cost.
\newblock {\em Calculus of Variations and Partial Differential Equations},
  pages 1--26, 2013.

\bibitem{Cut}
Marco Cuturi.
\newblock Sinkhorn distances: Lightspeed computation of optimal transport.
\newblock In {\em Advances in Neural Information Processing Systems}, pages
  2292--2300, 2013.

\bibitem{DeP}
Luigi De~Pascale.
\newblock Optimal transport with coulomb cost. approximation and duality.
\newblock {\em arXiv preprint arXiv:1503.07063}, 2015.

\bibitem{CoDePDMa}
Simone Di~Marino, Luigi De~Pascale, and Maria Colombo.
\newblock Multimarginal optimal transport maps for $1 $-dimensional repulsive
  costs.
\newblock {\em CANADIAN JOURNAL OF MATHEMATICS-JOURNAL CANADIEN DE
  MATHEMATIQUES}, 67:350--368, 2015.

\bibitem{DMaGGNe}
Simone Di~Marino, Augusto Gerolin, Luca Nenna, Michael Seidl, and Paola
  Gori-Giorgi.
\newblock The strictly-correlated electron functional for spherically symmetric
  systems revisited.
\newblock {\em in preparation}.

\bibitem{DolSon}
Yan Dolinsky and Halil~Mete Soner.
\newblock Martingale optimal transport and robust hedging in continuous time.
\newblock {\em Probability Theory and Related Fields}, 160(1-2):391--427, 2014.

\bibitem{DolSon2}
Yan Dolinsky and Halil~Mete Soner.
\newblock Robust hedging with proportional transaction costs.
\newblock {\em Finance and Stochastics}, 18(2):327--347, 2014.

\bibitem{Eke}
Ivar Ekeland.
\newblock An optimal matching problem.
\newblock {\em ESAIM: Control, Optimisation and Calculus of Variations},
  11(01):57--71, 2005.

\bibitem{EvaZwo}
Lawrence~C Evans and Maciej Zworski.
\newblock Lectures on semiclassical analysis, 2007.

\bibitem{Fri03}
Gero Friesecke.
\newblock The multiconfiguration equations for atoms and molecules: charge
  quantization and existence of solutions.
\newblock {\em Archive for Rational Mechanics and Analysis}, 169(1):35--71,
  2003.

\bibitem{FMPCK}
Gero Friesecke, Christian~B Mendl, Brendan Pass, Codina Cotar, and Claudia
  Kl{\"u}ppelberg.
\newblock N-density representability and the optimal transport limit of the
  hohenberg-kohn functional.
\newblock {\em The Journal of chemical physics}, 139(16):164109, 2013.

\bibitem{GalHenTou}
Alfred Galichon, Pierre Henry-Labord{\`e}re, Nizar Touzi, et~al.
\newblock A stochastic control approach to no-arbitrage bounds given marginals,
  with an application to lookback options.
\newblock {\em The Annals of Applied Probability}, 24(1):312--336, 2014.

\bibitem{GalichonEconomics}
Alfred Galichon and Bernard Salani{\'e}.
\newblock Matching with trade-offs: Revealed preferences over competing
  characteristics.
\newblock 2010.

\bibitem{GaSw}
Wilfrid Gangbo and Andrzej Swiech.
\newblock Optimal maps for the multidimensional monge-kantorovich problem.
\newblock {\em Communications on pure and applied mathematics}, 51(1):23--45,
  1998.

\bibitem{GhoMau}
Nassif Ghoussoub and Bernard Maurey.
\newblock Remarks on multi-marginal symmetric monge-kantorovich problems.
\newblock {\em arXiv preprint arXiv:1212.1680}, 2012.

\bibitem{GhoMoa}
Nassif Ghoussoub and Abbas Moameni.
\newblock Symmetric monge--kantorovich problems and polar decompositions of
  vector fields.
\newblock {\em Geometric and Functional Analysis}, 24(4):1129--1166, 2014.

\bibitem{GhoPass}
Nassif Ghoussoub and Brendan Pass.
\newblock Decoupling of degiorgi-type systems via multi-marginal optimal
  transport.
\newblock {\em Communications in Partial Differential Equations},
  39(6):1032--1047, 2014.

\bibitem{GorSeiVig}
Paola Gori-Giorgi, Michael Seidl, and Giovanni Vignale.
\newblock Density-functional theory for strongly interacting electrons.
\newblock {\em Physical review letters}, 103(16):166402, 2009.

\bibitem{Kant2}
Leonid Kantorovich.
\newblock On a problem of monge.
\newblock {\em Journal of Mathematical Sciences}, 133(4):1383--1383, 2006.

\bibitem{Kant1}
Leonid Kantorovitch.
\newblock On the translocation of masses.
\newblock {\em Management Science}, 5(1):1--4, 1958.

\bibitem{Kel}
Hans~G Kellerer.
\newblock Duality theorems for marginal problems.
\newblock {\em Zeitschrift f{\"u}r Wahrscheinlichkeitstheorie und verwandte
  Gebiete}, 67(4):399--432, 1984.

\bibitem{KimPass}
Young-Heon Kim and Brendan Pass.
\newblock Multi-marginal optimal transport on riemannian manifolds.
\newblock {\em arXiv preprint arXiv:1303.6251}, 2013.

\bibitem{KimPass14}
Young-Heon Kim and Brendan Pass.
\newblock A general condition for monge solutions in the multi-marginal optimal
  transport problem.
\newblock {\em SIAM Journal on Mathematical Analysis}, 46(2):1538--1550, 2014.

\bibitem{KimPasspreprint}
Young-Heon Kim and Brendan Pass.
\newblock Wasserstein barycenters over riemannian manfolds.
\newblock {\em arXiv preprint arXiv:1412.7726}, 2014.

\bibitem{KitPass}
Jun Kitagawa and Brendan Pass.
\newblock The multi-marginal optimal partial transport problem.
\newblock {\em arXiv preprint arXiv:1401.7255}, 2014.

\bibitem{KnoSmi}
Martin Knott and Cyril~S. Smith.
\newblock On a generalization of cyclic monotonicity and distances among random
  vectors.
\newblock {\em Linear algebra and its applications}, 199:363--371, 1994.

\bibitem{KolZae}
Alexander~V Kolesnikov and Danila~A Zaev.
\newblock Optimal transportation of processes with infinite kantorovich
  distance. independence and symmetry.
\newblock {\em arXiv preprint arXiv:1303.7255}, 2013.

\bibitem{LanPer}
David~C Langreth and John~P Perdew.
\newblock The exchange-correlation energy of a metallic surface.
\newblock {\em Solid State Communications}, 17(11):1425--1429, 1975.

\bibitem{Lieb79}
Elliott~H Lieb.
\newblock A lower bound for coulomb energies.
\newblock {\em Physics Letters A}, 70(5):444--446, 1979.

\bibitem{Lieb83}
Elliott~H Lieb.
\newblock Density functionals for coulomb systems.
\newblock In {\em Inequalities}, pages 269--303. Springer, 2002.

\bibitem{LiebOxford}
Elliott~H Lieb and Stephen Oxford.
\newblock Improved lower bound on the indirect coulomb energy.
\newblock {\em International Journal of Quantum Chemistry}, 19(3):427--439,
  1981.

\bibitem{LiebSimon}
Elliott~H Lieb and Barry Simon.
\newblock The hartree-fock theory for coulomb systems.
\newblock {\em Communications in Mathematical Physics}, 53(3):185--194, 1977.

\bibitem{Lions87}
Pierre-Louis Lions.
\newblock Solutions of hartree-fock equations for coulomb systems.
\newblock {\em Communications in Mathematical Physics}, 109(1):33--97, 1987.

\bibitem{LopMen}
Artur~O Lopes and Jairo~K Mengue.
\newblock Duality theorems in ergodic transport.
\newblock {\em Journal of Statistical Physics}, 149(5):921--942, 2012.

\bibitem{LopOliThi}
Artur~O Lopes, Elismar~R Oliveira, and Philippe Thieullen.
\newblock The dual potential, the involution kernel and transport in ergodic
  optimization.
\newblock {\em arXiv preprint arXiv:1111.0281}, 2011.

\bibitem{Malet2012}
Francesc Malet and Paola Gori-Giorgi.
\newblock {Strong Correlation in Kohn-Sham Density Functional Theory}.
\newblock {\em Physical Review Letters}, 109(24):246402, 2012.

\bibitem{Mendl2013}
Christian~B. Mendl and Lin Lin.
\newblock {Kantorovich dual solution for strictly correlated electrons in atoms
  and molecules}.
\newblock {\em Physical Review B}, 87(12):125106, 2013.

\bibitem{Monge}
Gaspard Monge.
\newblock {\em M{\'e}moire sur la th{\'e}orie des d{\'e}blais et des remblais}.
\newblock De l'Imprimerie Royale, 1781.

\bibitem{OlkRac}
Ingram Olkin and Svetlozar~T. Rachev.
\newblock Maximum submatrix traces for positive definite matrices.
\newblock {\em SIAM journal on matrix analysis and applications},
  14(2):390--397, 1993.

\bibitem{Passthesis}
Brendan Pass.
\newblock {\em Structural results on optimal transportation plans}.
\newblock PhD thesis, University of Toronto, 2011.

\bibitem{Pass3}
Brendan Pass.
\newblock Uniqueness and monge solutions in the multimarginal optimal
  transportation problem.
\newblock {\em SIAM Journal on Mathematical Analysis}, 43(6):2758--2775, 2011.

\bibitem{Pass2}
Brendan Pass.
\newblock On the local structure of optimal measures in the multi-marginal
  optimal transportation problem.
\newblock {\em Calculus of Variations and Partial Differential Equations},
  43(3-4):529--536, 2012.

\bibitem{Pass4}
Brendan Pass.
\newblock Remarks on the semi-classical hohenberg--kohn functional.
\newblock {\em Nonlinearity}, 26(9):2731, 2013.

\bibitem{Pass14}
Brendan Pass.
\newblock Multi-marginal optimal transport and multi-agent matching problems:
  uniqueness and structure of solutions.
\newblock {\em Discrete Contin. Dyn. Syst., 34:1623-1639}, 2014.

\bibitem{PassSurvey}
Brendan Pass.
\newblock Multi-marginal optimal transport: theory and applications.
\newblock {\em arXiv preprint arXiv:1406.0026}, 2014.

\bibitem{PazMorGorBac}
Simone Paziani, Saverio Moroni, Paola Gori-Giorgi, and Giovanni~B Bachelet.
\newblock Local-spin-density functional for multideterminant density functional
  theory.
\newblock {\em Physical Review B}, 73(15):155111, 2006.

\bibitem{Pra12}
Aldo Pratelli.
\newblock On the equality between monge's infimum and kantorovich's minimum in
  optimal mass transportation.
\newblock In {\em Annales de l'Institut Henri Poincare (B) Probability and
  Statistics}, volume~43, pages 1--13. Elsevier, 2007.

\bibitem{Pra07}
Aldo Pratelli.
\newblock On the sufficiency of c-cyclical monotonicity for optimality of
  transport plans.
\newblock {\em Mathematische Zeitschrift}, 2007.

\bibitem{Rob87}
Didier Robert.
\newblock {\em Autour de l'approximation semi-classique}, volume~68.
\newblock Birkh{\"a}user Basel, 1987.

\bibitem{Rusch95}
Ludger Ruschendorf.
\newblock Convergence of the iterative proportional fitting procedure.
\newblock {\em The Annals of Statistics}, pages 1160--1174, 1995.

\bibitem{RuschendorfThomsen}
Ludger R{\"u}schendorf and Wolfgang Thomsen.
\newblock Closedness of sum spaces andthe generalized schr{\"o}dinger problem.
\newblock {\em Theory of Probability \& Its Applications}, 42(3):483--494,
  1998.

\bibitem{RusUck}
Ludger R{\"u}schendorf and Ludger Uckelmann.
\newblock {\em On optimal multivariate couplings}.
\newblock Springer, 1997.

\bibitem{SchTei}
Walter Schachermayer and Josef Teichmann.
\newblock Characterization of optimal transport plans for the
  monge-kantorovich-problem.
\newblock {\em proceedings of the A.M.S.}, Vol. 137.(2):519--529, 2009.

\bibitem{Schrodinger31}
Erwin Schr{\"o}dinger.
\newblock {\em {\"U}ber die umkehrung der naturgesetze}.
\newblock Verlag Akademie der wissenschaften in kommission bei Walter de
  Gruyter u. Company, 1931.

\bibitem{Sei}
Michael Seidl.
\newblock Strong-interaction limit of density-functional theory.
\newblock {\em Physical Review A}, 60(6):4387, 1999.

\bibitem{SeiGorSav}
Michael Seidl, Paola Gori-Giorgi, and Andreas Savin.
\newblock Strictly correlated electrons in density-functional theory: A general
  formulation with applications to spherical densities.
\newblock {\em Physical Review A}, 75(4):042511, 2007.

\bibitem{SmiKno}
Cyril Smith and Martin Knott.
\newblock On hoeffding-fr{\'e}chet bounds and cyclic monotone relations.
\newblock {\em Journal of multivariate analysis}, 40(2):328--334, 1992.

\bibitem{TouColSav}
Julien Toulouse, Francois Colonna, and Andreas Savin.
\newblock Short-range exchange and correlation energy density functionals:
  Beyond the local-density approximation.
\newblock {\em The Journal of chemical physics}, 122(1):014110, 2005.

\bibitem{VilON}
C{\'e}dric Villani.
\newblock {\em Optimal transport: old and new}, volume 338.
\newblock Springer Science \& Business Media, 2008.

\bibitem{XferPeyAu}
Gui-Song Xia, Sira Ferradans, Gabriel Peyr{\'e}, and Jean-Fran{\c{c}}ois Aujol.
\newblock Synthesizing and mixing stationary gaussian texture models.
\newblock {\em SIAM Journal on Imaging Sciences}, 7(1):476--508, 2014.

\bibitem{Yang}
Weitao Yang.
\newblock Generalized adiabatic connection in density functional theory.
\newblock {\em The Journal of chemical physics}, 109(23):10107--10110, 1998.

\bibitem{Zae}
Danila Zaev.
\newblock On the monge-kantorovich problem with additional linear constraints.
\newblock {\em arXiv preprint arXiv:1404.4962}, 2014.

\bibitem{Zh60}
Grigorii~M. Zhislin.
\newblock Discussion of the spectrum of schr{\"o}dinger operators for systems
  of many particles.
\newblock {\em Trudy Moskovskogo matematiceskogo obscestva}, 9:81--120, 1960.

\end{thebibliography}
